%% file: Y9_arxiv.tex
\documentclass[hidelinks,onefignum,onetabnum]{siamart220329}

\input{ex_shared}

\ifpdf
\hypersetup{
  pdftitle={Isola and mushroom dynamics in a predator-prey system},
  pdfauthor={Yancong Xu et al.}
}
\fi

\input{ex_shared}

\begin{document}

\maketitle

\begin{abstract}
This paper investigates a predator-prey system with an additive Allee effect and a generalized Holling IV functional response using a dynamical system approach. By means of a mixture of analytical and numerical procedures, we find the existence of codimension two and three Bogdanov-Takens bifurcation, codimension-three generalized Hopf bifurcation, and codimension-two cusp of limit cycles. We also found mushroom and isola bifurcations of limit cycles as the first examples of such phenomena in a predatory interaction. The model predicts that
extinction of both populations may only occur if the Allee effect is strong. However, long term coexistence is possible in both weak and strong Allee regimes indicating that predation has a balancing role in the interaction dynamics.
Nonetheless, a weak Allee effect can result in complex dynamics as well, including the presence of isolas, mushrooms and cusps of limit cycles.
\end{abstract}

\begin{keywords}
Additive Allee effect, generalized Holling IV functional response, codimension three bifurcations, isola and mushroom bifurcations of limit cycles.
\end{keywords}

\begin{MSCcodes}
37G15, 34C23,  37M20,  	92D25
\end{MSCcodes}

\section{Introduction}
This article is about the interaction of two populations, a prey and a predator species, modeled using nonlinear differential equations.
We are especially interested in the case where the prey population faces difficulties in growing from low densities and avoiding extinction; this phenomenon is traditionally known as the Allee effect. It was proposed by Allee~\cite{allee1931study} to investigate the interaction between species growth and density. The Allee effect refers to the phenomenon where the population size decreases and even faces extinction (with an extinction threshold) \cite{bascompte2003extinction} when it becomes too sparse. The Allee effect can be caused by various factors, including extinction, population establishment and biological invasion~\cite{am2009evidence}, genetic inbreeding depression \cite{stephens1999consequences}, difficulties in finding mating partners~\cite{contreras2020finding,courchamp2008allee,mccarthy1997allee}, predator-prey interactions~\cite{gascoigne2004allee}, cooperative defense, cooperative feeding~\cite{courchamp1999inverse}, environmental conditioning~\cite{stephens1999consequences}, social dysfunction at low densities~\cite{luque2013allee}, and others. The Allee effect can be classified as either a strong Allee effect~\cite{berec2007multiple, Wen2023}, where a threshold population level exists~\cite{bascompte2003extinction, clark1974mathematical}, or a weak Allee effect~\cite{stephens1999consequences} where this threshold is absent. In the case of a strong Allee effect, the species may experience additional mortality, ultimately leading to extinction~\cite{stephens1999consequences}. The Allee effect is also known as negative competition effect~\cite{wang1999competitive}; in fishing models, it is called depensation~\cite{clark1974mathematical}; in epidemiology, its analogous is the eradication threshold, the population level of susceptible individuals below which an infectious disease is eliminated from a population~\cite{bascompte2003extinction}.

Recently, there has been a growing understanding that the Allee effect plays a significant role in the possibility of local and global extinction, as well as various dynamic behaviors that profoundly impact species reproduction and conservation. For some marine species, per capita growth has been shown to decrease as population size is reduced below a critical level and two proposed causes for such a depensation or Allee effect are: Reduced reproductive success at low population densities~\cite{contreras2020finding} and increased relative predation on small populations~\cite{gascoigne2004allee}. This situation has occurred in many real fisheries as a consequence of overfishing when man acts as a predator~\cite{clark1974mathematical}.
Therefore, it is both engaging and important to explore the role of the Allee effect in predator-prey interactions.  

Let $x(t)$ and $y(t)$ denote the densities of prey and predator populations at time $t$, respectively.
We consider the following nonlinear differential equations modeling our populations:
\begin{equation}\label{A1}
\begin{array}{rcl}
\dfrac{dx}{dt}&=&\left(r \left(1-\dfrac{x}{k}\right)-\dfrac{m}{x+b}\right)x-\dfrac{q x y}{x^2+c x+a},\vspace{2mm} \\
\dfrac{dy}{dt}&=&sy\left(1-\dfrac{y}{nx}\right).
\end{array}
\end{equation}
From a biological feasibility perspective, we require that
$(x,y)\in \Gamma_{1}=\{(x,y)|x>0, y\geq0\}$
 and  $(r,s,k,q,a,n,m,b)\in \mathbb{R}_{+}^{8}$. 
Additionally, we set parameter $c\in \mathbb{R}$ to satisfy $c<2\sqrt{a}$ to guarantee that $x^2+cx+a>0$ for all $x>0$.

Model~\eqref{A1} follows a  Leslie-Gower formulation~\cite{hsu1995global,may1973stability,zhu2017bifurcations}. Here, the parameters $r$ and $s$ represent the intrinsic growth rates of the prey and predator, respectively. The parameter $k$ represents the carrying capacity of the prey in the absence of predation, while $nx$, proportional to the prey density -- i.e., a measure of the abundance of prey --, represents the carrying capacity of the predator.
For instance, Leslie's model has been used to model vole-weasel dynamics~\cite{hanski2001small}. This modeling scheme differs from Gause type models~\cite{HSU2019} in which the predator equation is based on a mass action principle -- as the actual numerical response is assumed to be depending on the functional response.

In~\eqref{A1}
the term $\frac{m}{x+b}$ models an Allee effect that involves the prey population. Since $\frac{m}{x+b}$ is effectively adding to the familiar logistic term in the prey equation, it is usually called an additive Allee effect~\cite{aguirre2009three, aguirre2009two, arsie2022predator,shang2022bifurcation, xu2019regime, zhang2020dynamic, lai2020stability, molla2022dynamics}.
This formula was proposed in~\cite{dennis1989allee} and~\cite{stephens1999consequences} and is one of the simplest forms that can capture both the strong and weak Allee effects~\cite{aguirre2009three,aguirre2009two}.
From a biological perspective, a strong Allee effect indicates that losses due to undercrowding at low population levels outweigh the growth. On the other hand, a weak Allee effect does not have this feature~\cite{berec2007multiple}.
In the absence of predation, the strength of the Allee effect in~\eqref{A1} can be determined by the condition $0 < br < m$. If this condition holds, it is considered a strong Allee effect. Conversely, if the condition $0 < m < br$ holds, it is classified as a weak Allee effect.

On the other hand, the functional response in~\eqref{A1} -- reflecting the capture ability of the predator with respect to the prey--- is modeled as a  {\it generalized} Holling type IV~\cite{arsie2022predator,collings1995bifurcation} or Monod-Haldane~\cite{xiao2001global}; it represents a low predation rate at low prey densities with a sharp acceleration as $x$ increases non-monotonically until it reaches its asymptotic value $q$.
This type of functional response seems a reasonable possibility if prey and webbing densities are assumed to be
directly related. Non-monotonic functional responses -- such as generalized Holling type IV -- are used to model the phenomenon of aggregation, a social behavior of prey, in which prey congregates on a fine scale relative to predator, so the predator's hunt is not spatially homogeneous, as is the case with miles-long schools of a certain class of fish~\cite{taylor1984predation}. In this case, one of the main advantages of schooling seems to be the confusion of the predator when it attacks. The most important benefits of aggregation are increased stealth. Additionally, aggregation can decrease vulnerability to being attacked and increase the time that group members can spend on activities other than surveillance~\cite{taylor1984predation}.
Other related examples of non-monotonous consumption occur at the microbial level where evidence indicates that in the face of an overabundance of nutrients, the effectiveness of the consumer may begin to decrease. This is often seen when microorganisms are used for waste breakdown or water purification, a phenomenon called inhibition~\cite{freedman1986predator,xiao2001global}.

The combination of prey population exhibiting Allee effect for low densities and a non-monotonic functional response has already been reported in a wide range of cases~\cite{gascoigne2004allee}.
For instance, there is the case of the Atlantic cod ({\it Gadus morhua}) that forms schools during the day, since commercial fishing (man as predator) causes the collapse of the stock because a greater proportion of this aggregate population is caught per unit of effort when the population decreases~\cite{clark2006worldwide,gascoigne2004marine}.
Furthermore, for obligately cooperative breeders such as the African wild dog ({\it Lycaon pictus}) and the meerkat ({\it Suricata suricatta}), a similar situation exists, as juvenile survival is lower in small groups than in large groups in areas with high predator densities, but less in large groups than in small groups in areas with low density of predators~\cite{courchamp2000impact,gascoigne2004allee}.

Other approaches to modeling the functional response typically fall into two main types: either prey-dependent~\cite{ 2006Qualitative,li2008traveling,seo2008comparison,shang2022bifurcation,wang2019canards,yang2023global,zegeling2020singular} or ratio-dependent~\cite{lajmiri2018bifurcation,negi2007dynamics,xiao2001global}.
 For instance, Aguirre et al.~\cite{aguirre2009two} studied the bifurcation diagram of limit cycles and the existence of two limit cycles in an early version of~\eqref{A1} with simplified Holling IV functional response; later they demonstrated that the same model can exhibit the coexistence of three limit cycles~\cite{aguirre2009three}.
Arsie et al.~\cite{arsie2022predator} analyzed high codimension bifurcations in a predator-prey system with generalized Holling type IV functional response and Allee effects in the prey population. They confirmed the occurrence of degenerate Hopf bifurcation of codimension 3 and heteroclinic bifurcation of codimension 2. Additionally, they discovered a new unfolding of a nilpotent saddle of codimension 3 with a fixed invariant line. Lai et al.~\cite{lai2020stability} concluded that saddle-node bifurcation, transcritical bifurcation, and Hopf bifurcation exist in a predator-prey model with additive Allee effect and fear effect. Molla et al.~\cite{molla2022dynamics} developed a modified Lotka-Volterra model that incorporates variable prey refuge and Holling type II functional response. They showed that this model exhibits saddle-node bifurcation, Hopf bifurcation, and Bogdanov-Takens bifurcation.

In this paper, we find specific conditions such that model (\ref{A1}) may undergo saddle-node bifurcation, Hopf bifurcation, Bogdanov-Takens bifurcation, homoclinic bifurcation, and saddle-node bifurcation of limit cycles. We also prove the existence of Bogdanov-Takens point of codimension 3 and Hopf bifurcation of codimension 3 under certain parameter conditions.
Through bifurcation diagrams and phase portraits, we reveal surprising biological consequences and emphasize the role of the additive Allee effect in determining a survival threshold to prevent the collapse of the system and the extinction of prey and predator populations.
We also identify and describe the presence of isola and mushrooms of limit cycles in system~(\ref{A1}). The so-called isola bifurcation refers to a closed locus of a solution branch delimited by two fold points -- or, in general, an even number of folds; it may correspond to branches of either equilibrium points, limit cycles, homoclinic orbits, or heteroclinic connections. On the other hand, a mushroom bifurcation is an open locus -- i.e., homeomorphic to an open interval of the real line -- that contains at least two fold points, which is its key distinguishing feature. While mushroom bifurcation and isola bifurcation of equilibrium points are phenomena observed in various contexts, including biological, epidemiological, chemical, and physical models~\cite{avitabile2012numerical, das2023origin,gray1985sustained,otero2023automated,sandstede2012snakes,xu2020vectored},
to the best of our knowledge, we are the first to find isola and mushrooms of limit cycles in a predator-prey system with an additive Allee effect.
Furthermore,  we also find a codimension-2 cusp point of limit cycles -- also, a novelty in predator-prey models with Allee effect.
Indeed, these are some of the highlights of this work as these findings indicate novel mechanisms for the emergence of sustained oscillations and coexistence in a predator-prey interaction.

The remaining sections of this article are organized as follows. In Section 2, we present the analysis of the dynamics near the extinction scenario $y=0$ and $x\rightarrow0$, and provide preliminary results regarding the existence and type of equilibria in system (\ref{A1}). In Section~3, we conduct a bifurcation analysis that includes proving the presence of Bogdanov-Takens bifurcation of codimension 3 and Hopf bifurcation of codimension 3. Section~4 employs numerical bifurcation analysis to visualize some of the obtained results, including saddle-node bifurcations, homoclinic bifurcations, isola  and mushrooms bifurcations of limit cycles, saddle-node bifurcation of limit cycles, and codimension-2 cusp of limit cycles. Finally, the paper concludes with a summary and discussion of the results.

\section{Dynamics near $(0,0)$, the existence and the type of equilibria}

\subsection{The dynamics near $(0,0)$}

It is worth noting that system (\ref{A1}) is not well-defined as $x\rightarrow0$; nevertheless, the important interpretation of $(x,y)=(0,0)$ motivates us to explore the dynamics near this point. It is easy to see that, in the absence of predator ($y=0$), $x(t)$ is a decreasing function of $t$ if $m>br$ and, hence, $x(t)\rightarrow0$ as $t\rightarrow\infty$; on the other hand, $x(t)$ is increasing if $m<br$ and $x(t)\rightarrow0$ as $t\rightarrow-\infty$.

By time rescaling $t=knx(x+b)(x^2+cx+a)\tau$, we can get the polynomial system
\begin{equation}\label{B0}
\begin{array}{rcl}
\dfrac{dx}{d\tau}&=&nx^2(x^2+cx+a)(r(x+b)(k-x)-mk)-knq(x+b)x^{2}y, \vspace{2mm}\\
\dfrac{dy}{d\tau}&=&ksy(x+b)(x^2+cx+a)(nx-y),
\end{array}
\end{equation}
which is $C^{\infty}$-equivalent to system (\ref{A1}) in the region $\Gamma_{1}.$ Obviously, system (\ref{B0}) can be extended continuously to the axis $x=0$ and, hence, it is well-defined in the entire first quadrant.
In particular, $(0,0)$ is an equilibrium of (\ref{B0}).

\begin{theorem}
 Let us define the following regions in the first quadrant of the $(r,m)$-plane:
$$
\begin{array}{ccl}
I&=&\{(r,m)\in\mathbb{R}^2_+:\ br-m<0\},\\
II&=&\{(r,m)\in\mathbb{R}^2_+:\ br-m>0,\ m+b(s-r)>0\},\\
III&=&\{(r,m)\in\mathbb{R}^2_+:\  m+b(s-r)<0\}.
\end{array}
$$
The following holds:
\begin{enumerate}
\item [(1)] If $(r,m)\in I$, the origin of \eqref{B0} is a non hyperbolic atractor.
\item [(2)] If $(r,m)\in II$, the origin of \eqref{B0} has a hyperbolic sector and a parabolic repelling sector.
\item [(3)] If $(r,m)\in III$, the origin of \eqref{B0} is a non hyperbolic saddle.
\end{enumerate}
\end{theorem}

\begin{proof}
Note that the Jacobian matrix of (\ref{B0}) at $(0,0)$ is the null matrix. This observation enlightens us to make the polar blow-up transformation  $x=u\cos(v)$, $y=u\sin(v)$, $t=u\tau$ to obtain
\begin{equation}\label{B1B1}
\begin{array}{rl}
&\dfrac{du}{dt}=-a k u \left(-b n r \cos ^3(v)-b n s \sin ^2(v) \cos (v)+b s \sin ^3(v)+m n \cos ^3(v)\right)+O(u^2), \vspace{2mm}\\
&\dfrac{dv}{dt}=a k \sin (v) \cos (v) (-b n r \cos (v)+b n s \cos (v)-b s \sin (v)+m n \cos (v))+O(u),
\end{array}
\end{equation}
where $(u, v)\in[0,+\infty)\times[0,\frac{\pi}{2}]$. System (\ref{B1B1}) has a maximum of three equilibria at $u=0$:  $(0,0)$, $(0,\frac{\pi}{2})$, and $(0,v^*)$ with $v^*=\arctan\left(\frac{n(m+b(s-r))}{bs}\right)$, provided $m+b(s-r)>0$.
The corresponding Jacobian matrices $J(u^*,v^*)$ are:
$$
J(0,0)=\left(
\begin{array}{cc}
 a k n (b r-m) & 0 \\
 0 & a k n (m+b (s-r)) \\
\end{array}
\right),
\hspace{3mm}
%
%
J\left(0,\frac{\pi}{2}\right)=\left(
\begin{array}{cc}
 -a b k s & 0 \\
 0 & a b k s \\
\end{array}
\right),
$$
%
and
$$
J(0,v^*)=\left(
\begin{array}{cc}
 \frac{a k n (b r-m)}{\sqrt{\left(\frac{n (m+b (s-r))}{b s}\right)^2+1}} &0\\
-\frac{b^2 n^2 s^2 (b (s-r)+m) \left(b^2 (a r s+k n q (s-r))-a k m s+b k m n q\right)}{\left(b^2 \left(n^2+1\right) s^2+2 b n^2 s (m-b r)+n^2 (m-b r)^2\right)^2} & -\frac{a k n (m+b (s-r))}{\sqrt{\left(\frac{n (m+b (s-r))}{b s}\right)^2+1}} \\
\end{array}
\right).
$$
%
Note that condition $br-m<0$ implies $m+b(s-r)>0$; on the other hand, if $m+b(s-r)<0$, then  $br-m>0$.
The result follows from ``blowing down'' the phase portraits of \eqref{B1B1} near the arc $u=0$ with $0\leq v\leq \frac{\pi}{2}$ back to system \eqref{B0}, see {\bf Figure} \ref{origin_blow_up} for details.
\end{proof}

\begin{figure}[h!]
\begin{center}
\includegraphics[width=0.94\textwidth]{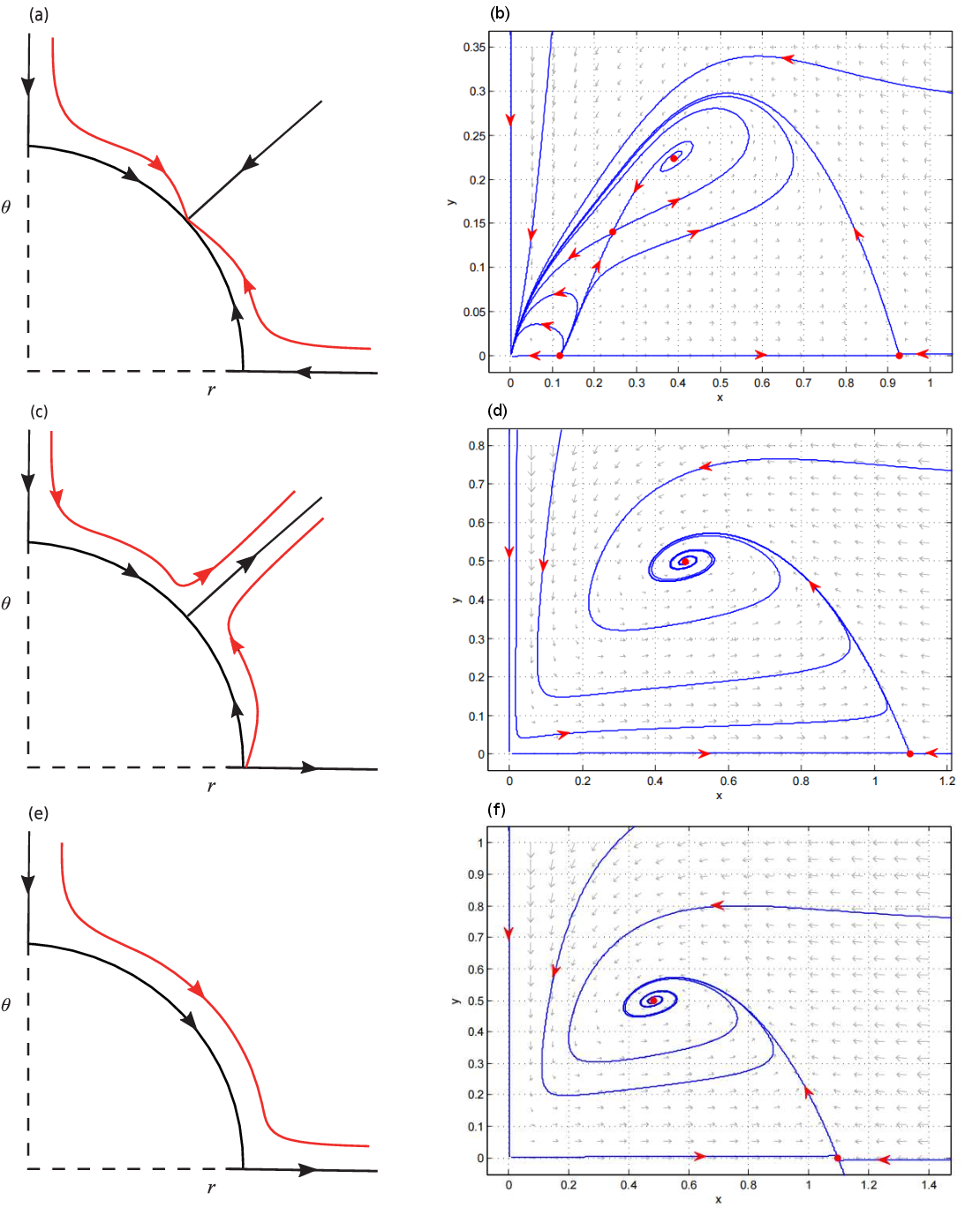}
\end{center}
\caption{\footnotesize (a) Three equilibria of system \eqref{B1B1} when $br-m<0$. (b) The trajectories near $(0,0)$ of system \eqref{B0} when $br-m<0$. (c) Three equilibria of system \eqref{B1B1} when $m+b(s-r)>0, br-m>0$. (d) The trajectories near $(0,0)$ of system \eqref{B0} when $m+b(s-r)>0, br-m>0$.
(e) Two equilibria of system \eqref{B1B1} when $m+b(s-r)<0$. (f) The trajectories near $(0,0)$ of system \eqref{B0} when $m+b(s-r)<0$.}
\label{origin_blow_up}
\end{figure}

\subsection{Existence and type of equilibria}

\begin{lemma}
 The rectangular region $\Gamma_{2}=\{(x,y)|0<x\leq k, 0\leq y \leq nk\}$ is a positively invariant set of the system (\ref{A1}).
\end{lemma}

\begin{proof}
By the first equation of system (\ref{A1}), we see that $\frac{dx}{dt}|_{x>k}<0$. Consequently, we just focus on $0<x \leq k$. On the other hand, we can easily get $\frac{dy}{dt}=sy(1-\frac{y}{nx})<sy(1-\frac{y}{nk})$ for $0<x\leq k$, which leads to $\frac{dy}{dt}|_{y>nk}< 0$.
Furthermore, $\frac{dy}{dt}|_{y=0}= 0$, so the $x$-axis is invariant.
Therefore, all solutions of system (\ref{A1}) will ultimately move towards the region $\Gamma_{2}=\{(x,y)|0<x\leq k, 0\leq y \leq nk\}$ remaining there for $t\rightarrow\infty$. This completes the proof.
\end{proof}

For the equilibria of system (\ref{A1}) in the boundary of  region $\Gamma_{2}$, by a  straightforward analysis, we have the following results.

\begin{lemma}
For the boundary equilibria, system (\ref{A1}) has\\
$(I)$ no boundary equilibrium if $m>\frac{r (b+k)^2}{4 k}$;\\
$(II)$ a unique boundary equilibrium given by\\
$(i)$ $B_{1,2}=\left(\frac{k-b}{2},0\right)$ if $m=\frac{r (b+k)^2}{4 k}$, which is a degenerate equilibrium;\\
$(ii)$ $B_{2}=\left(k-b,0\right)$ if $m=br<\frac{r (b+k)^2}{4 k}$, which is a saddle;\\
$(iii)$ $B_{2}=\left(\frac{1}{2} \left(k-b+\sqrt{\frac{r (b+k)^2-4 k m}{r}}\right),0\right)$ if $m<\min\{\frac{r (b+k)^2}{4 k},br\}$, which is a saddle;\\
$(III)$ two boundary equilibria if $br<m<\frac{r (b+k)^2}{4 k}$, given by $B_{1}=\left(\frac{1}{2} \left(k-b-\sqrt{\frac{r (b+k)^2-4 k m}{r}}\right),0\right)$ which is an unstable node, and $B_{2}=\left(\frac{1}{2} \left(k-b+\sqrt{\frac{r (b+k)^2-4 k m}{r}}\right),0\right)$ which is a saddle.
\end{lemma}

Next, we turn to consider the existence of the positive equilibria of (\ref{A1}). Assuming that $E=(x,y)$ is any positive equilibrium of system (\ref{A1}), we can easily obtain that $y=nx$ from the second equation of (\ref{A1}). Substituting it into the first equation, we conclude that $x$ is determined by the following equation:
\begin{eqnarray}\label{a00}
F(x)=\frac{-x(a_{4}x^4+a_{3}x^{3}+a_{2}x^{2}+a_{1}x+a_{0})}{h(x)},
\end{eqnarray}
where
\begin{eqnarray*}
&&a_{4}=r, \ \ a_{3}=r (b+c-k), \ \ a_{2}=r (a+b c)+k (-r (b+c)+m+n q), \\
&&a_{1}=a r (b-k)+c k (m-b r)+b k n q, \ \ a_{0}=a k (m-b r),\\
&&h(x)=k (b+x) (a+c x+x^2).
\end{eqnarray*}

By observation, we only need to study the positive roots of
\begin{eqnarray}\label{A5}
\bar{F}(x):= a_{4}x^4+a_{3}x^{3}+a_{2}x^{2}+a_{1}x+a_{0}=0.
\end{eqnarray}
According to Descartes's rule of signs, equation (\ref{A5}) can have a maximum of four positive roots, which we denote in ascending order as $x_{1}<x_{2}<x_{3}<x_{4}$. The corresponding equilibria of~(\ref{A1}) will be denoted as $E_{1}=(x_{1},y_{1}), E_{2}=(x_{2},y_{2}), E_{3}=(x_{3},y_{3})$, and $E_{4}=(x_{4},y_{4})$, respectively. For a detailed analysis of the existence of positive equilibria, please refer to~\textbf{Appendix \ref{AppA}}.

The Jacobian matrix of~\eqref{A1} at any positive equilibrium $E=(x,nx)$ is
\begin{eqnarray*}
J(E)=
\left(
\begin{array}{cc}
\frac{n q x (x^2-a)}{(a+x (c+x))^2}-\frac{b m}{(b+x)^2}-\frac{2 r x}{k}+r & -\frac{q x}{a+x (c+x)}\\
n s & -s \\
\end{array}
\right),
\end{eqnarray*}
and the corresponding characteristic polynomial is
\begin{eqnarray}\label{a000}
\lambda^{2}-p(x)\lambda+q(x)=0,
\end{eqnarray}
in which
\begin{equation}\label{a0}
\begin{array}{rl}
&q(x)=\det(J(E))=s\left(\frac{nqx(2a+cx)}{(a+x(c+x))^2}+\frac{bm}{(b+x)^2}+r\left(\frac{2x}{k}-1\right)\right)=-sF'(x),\\
&p(x)=\mathrm{tr}(J(E))=\frac{nqx(x^2-a)}{(a+x(c+x))^2}-\frac{bm}{(b+x)^2}-\frac{2rx}{k}+r-s, \ \
{\rm and}\\
&p'(x)=\frac{nq(-a^2+ax(c+6x)+x^3(c-x))}{(a+x(c+x))^3}+\frac{2bm}{(b+x)^3}-\frac{2r}{k}.
\end{array}
\end{equation}

By standard local analysis  we get the following result.

\begin{proposition}
Let $x_{i}$ be a positive root of~\eqref{A5}.
The following statements hold:\\
(I) if $F'(x_{i})>0$, then $E_{i}$ is a saddle;\\
(II) if $F'(x_{i})<0$ and $p(x_i)<0$, then $E_{i}$ is an attracting node or focus;\\
(III) if $F'(x_{i})<0$ and $p(x_i)>0$,  then $E_{i}$ is a repelling node or focus.\\
(IV) if $F'(x_{i})<0$ and $p(x_i)=0$,  then $E_{i}$ is a linearized center. 
\end{proposition}

By (\ref{a00}) and (\ref{A5}), we see that $F'(x)=-\frac{\bar{F}(x)}{h(x)}-\frac{x(\bar{F}'(x)h(x)-\bar{F}(x)h'(x))}{h^{2}(x)}
=-\frac{x\bar{F}'(x)}{h(x)}$. Furthermore, from \textbf{Appendix \ref{AppA}}, we obtain that  $F'(x_{1}), F'(x_{3})<0$ when~(\ref{A1}) has four simple positive equilibria $E_1, E_2, E_3$ and $E_4$,
 with the order $x_{1}<x_{2}<x_{3}<x_{4}$ specified above. Therefore, in this situation, $E_{1}$ and $E_{3}$ are saddles.

Based on the analysis in \textbf{Appendix \ref{AppA}}, we can conclude that two positive equilibria of system (\ref{A1}) may coincide at $E_{1,2}$, $E_{2,3}$, or $E_{3,4}$, where $E_{i,i+1}$ refers to the ``collision'' of $E_{i}$ with $E_{i+1}$ (where $i=1,2,3$). For simplicity of notation, we denote these coincidence points as $E^{*}=(x^{*}, nx^{*})$.

\begin{theorem}\label{tha1} Let $F$ be as in \eqref{a00} and let $p$ be as in \eqref{a0}.
Suppose that $F(x^{*})=F'(x^{*})=0$ and $F''(x^{*})\neq 0$. The following statements hold for the equilibrium $E^{*}=(x^{*}, nx^{*})$ of  (\ref{A1}):\\
(I) if $p(x^{*}) \neq 0$, then $E^{*}$ is a saddle-node point;\\
(II) if $p(x^{*})=0$ and $p'(x^{*})\neq 0$, then $E^{*}$ is a cusp of codimension 2;\\
(III) if $p(x^{*})=p'(x^{*})= 0$ and $b\neq b_{\pm}$, then $E^{*}$ is a cusp of codimension 3, where
\begin{eqnarray}\label{a22}
b_{\pm}=\frac{x^{*}(C\pm (a+x^{*}(c+x^{*}))^2\sqrt{E})}{4B},
\end{eqnarray}
and
\begin{eqnarray*}
&&B=a^6+a^5 x^{*} (5 c+3 x^{*})+2 a^4 x^{*2} (5 c^2+14 c x^{*}+18 x^{*2})-a^3 x^{*4} (7 c^2+14 c x^{*}+30 x^{*2})-a^2\\
&& \ \ \ \ \ \ \times x^{*4}(12 c^4+97 c^3 x^{*}+309 c^2 x^{*2}+476 c x^{*3}+275 x^{*4})+a x^{*6} (4 c^4+34 c^3 x^{*}+119 c^2 x^{*2}\\
&& \ \ \ \ \ \ +217 c x^{*3}+155 x^{*4})-x^{*9} (c^3+5 c^2 x^{*}+16 c x^{*2}+18 x^{*3}),\\
&&C=-a^5 c x^{*}+20 a^5 x^{*2}-29 a^4 c^2 x^{*2}-105 a^4 c x^{*3}-184 a^4 x^{*4}-43 a^3 c^3 x^{*3}-280 a^3 c^2 x^{*4}-670\\
&& \ \ \ \ \ \ \times a^3 c x^{*5}-424 a^3 x^{*6}+9 a^2 c^4 x^{*4}+91 a^2 c^3 x^{*5}+450 a^2 c^2 x^{*6}+1138 a^2 c x^{*7}+960 a^2 x^{*8}-6\\
&& \ \ \ \ \ \ \times a c^4 x^{*6}-21 a c^3 x^{*7}-40 a c^2 x^{*8}-225 a c x^{*9}-332 a x^{*8}+c^4 x^{*8}+5 c^3 x^{*9}-5 c^2 x^{*10}-9 c\\
&& \ \ \ \ \ \ \times x^{*11}+24 x^{*12},\\
&&E=(a^2+3 a x^{*} (c+2 x^{*})-x^{*3} (c+3 x^{*})) (8 a^6+8 a^5 x^{*} (7 c+29 x^{*})-a^4 x^{*2} (23 c^2+216 c x^{*}\\
&& \ \ \ \ \ \ +848 x^{*2})+a^3 x^{*3} (27 c^3+296 c^2 x^{*}+1256 c x^{*2}+2416 x^{*3})-a^2 x^{*5} (27 c^3+142 c^2 x^{*}+840\\
&& \ \ \ \ \ \times c x^{*2}+2216 x^{*3})+a x^{*7} (9 c^3-80 c^2 x^{*}-208 c x^{*2}+520 x^{*3})+x^{*9} (-(c^3+3 c^2 x^{*}-48 c x^{*2}\\
&& \ \ \ \ \ +48 x^{*3}))).
\end{eqnarray*}
\end{theorem}

\begin{proof}
We begin with the proof of statement $(I)$. From conditions $F(x^{*})=F'(x^{*})=0$ we can conclude that $m=\frac{r (b+x^{*})^2 (a k+x^{*2} (c-k+2 x^{*}))}{k (a (b+2 x^{*})+x^{*2} (c-b))}$ and $q=\frac{r (a+x^{*} (c+x^{*}))^2 (-b+k-2 x^{*})}{k n (a (b+2 x^{*})+x^{*2} (c-b))}$. Next, carrying out the transformation $X=x-x^{*}, y=y-n x^{*}$ to bring $E^{*}$ to the origin and expanding the resulting system around the origin up to third order terms yields:
\begin{eqnarray*}
\nonumber&&\frac{dx}{dt}=\hat{a}_{10}x+\hat{a}_{01}y+\hat{a}_{20}x^{2}+\hat{a}_{11}xy
+\hat{a}_{30}x^{3}+\hat{a}_{21}x^{2}y+o(|x,y|^{3}),\\
&&\frac{dy}{dt}=\hat{b}_{10}x+\hat{b}_{01}y+\hat{b}_{20}x^{2}+\hat{b}_{11}xy
+\hat{b}_{02}y^2+\hat{b}_{30}x^{3}+\hat{b}_{21}x^{2}y+\hat{b}_{12}xy^{2}+o(|x,y|^{3}),
\end{eqnarray*}
where $\hat{a}_{ij}$ and $\hat{b}_{ij}$ are given in \textbf{Appendix \ref{AppB}}.

Define the change of coordinates $X=\frac{\hat{b}_{01} x-\hat{a}_{01} y}{\hat{a}_{10}+\hat{b}_{01}}, Y=\frac{\hat{a}_{10} x+\hat{a}_{01} y}{\hat{a}_{10}+\hat{b}_{01}}$. Then we have
\begin{equation}\label{a1111}
\begin{array}{rrl}
&&\dfrac{dx}{dt}=\hat{c}_{20}x^{2}+\hat{c}_{11}xy+\hat{c}_{02}y^{2}+\hat{c}_{30}x^{3}+\hat{c}_{21}x^{2}y+\hat{c}_{12}xy^{2}
+\hat{c}_{03}y^{3}+o(|x,y|^{3}), \vspace{2mm}\\
&&\dfrac{dy}{dt}=\hat{d}_{01}y+\hat{d}_{20}x^{2}+\hat{d}_{11}xy+\hat{d}_{02}y^{2}+\hat{d}_{30}x^{3}+\hat{d}_{21}x^{2}y+\hat{d}_{12}xy^{2}
+\hat{d}_{03}y^{3}+o(|x,y|^{3}),
\end{array}
\end{equation}
where $\hat{c}_{ij}$ and $\hat{d}_{ij}$ are given in \textbf{Appendix \ref{AppB}}.

Note that $\hat{d}_{01}=p(x^{*})\neq0$ and $\hat{c}_{20}=\frac{k s (a (b+2 x^{*})+x^{*2} (c-b))}{2 A}F''(x^{*})$, in which $A=x^{*2} (2 a r+c r (b-k)+k s (c-b))+a x^{*} (b r-k r+2 k s)+a b k s+r x^{*3} (b+2 c-k)+2 r x^{*4}$. It follows that there exists a center manifold specified locally near the origin as
\begin{eqnarray*}
y=\frac{\hat{c}_{20}}{p(x^{*})}x^{2}+O(x^{3})
\end{eqnarray*}
 such that system (\ref{a1111}) restricted to this center manifold is expressed by
\begin{eqnarray}\label{a200}
\frac{dx}{dt}=\hat{c}_{20}x^{2}+o(x^{2}).
\end{eqnarray}
Condition $F''(x^{*}) \neq 0$ tells us that $\hat{c}_{20} \neq 0$. Therefore, (\ref{a200}) is topologically equivalent to
\begin{eqnarray*}
\frac{dx}{dt}=\pm x^{2}+o(x^{2}),
\end{eqnarray*}
which shows that  the equilibrium $E^{*}=(x^{*}, nx^{*})$ is a saddle-node point;

Proof of statement $(II)$. According to $F(x^{*})=F'(x^{*})=p(x^{*})=0$, we can easily obtain that $r=-\frac{k s (a (b+2 x^{*})+x^{*2} (c-b))}{x^{*} (a+x^{*} (c+x^{*})) (b-k+2 x^{*})}, m=-\frac{s (b+x^{*})^2 (a k+x^{*2} (c-k+2 x^{*}))}{x^{*} (a+x^{*} (c+x^{*})) (b-k+2 x^{*})}, n=\frac{s (a+x^{*} (c+x^{*}))}{q x^{*}}$. Plugging this into system (\ref{A1}) and making the following transformation
\begin{eqnarray*}
X=x-x^{*}, Y=y-nx^{*},
\end{eqnarray*}
we have
\begin{eqnarray*}
\nonumber&&\frac{dx}{dt}=a_{10}x+a_{01}y+a_{20}x^{2}+a_{11}xy+o(|x,y|^{2}),\\
&&\frac{dy}{dt}=b_{10}x+b_{01}y+b_{20}x^{2}+b_{11}xy+b_{02}y^2+o(|x,y|^{2}),
\end{eqnarray*}
where
\begin{eqnarray*}
&&a_{10}=s,\ \
a_{01}=-\frac{q x^{*}}{a+x^{*} (c+x^{*})},\\
&&a_{20}=\frac{s}{(a+x^{*} (c+x^{*}))^2} (\frac{1}{x^{*} (b-k+2 x^{*})} (a+x^{*} (c+x^{*})) (a (b+2 x^{*})+x^{*2} (c-b))+\\
&& \ \ \ \ \ \ \ \ \frac{b}{x^{*} (b+x^{*}) (b-k+2 x^{*})}(a+x^{*} (c+x^{*}))(a k+x^{*2} (c-k+2 x^{*}))+a (c+3 x^{*})-x^{*3}),\\
&&a_{11}=\frac{q (x^{*2}-a)}{(a+x^{*} (c+x^{*}))^2},\
b_{10}=\frac{s^2 (a+x^{*} (c+x^{*}))}{q x^{*}},\
b_{01}=-s,\\
&&b_{20}=-\frac{s^2 (a+x^{*} (c+x^{*}))}{q x^{*2}},\
b_{11}=\frac{2 s}{x^{*}},\
b_{02}=-\frac{q}{a+x^{*} (c+x^{*})}.
\end{eqnarray*}

Introducing the change of variables and rescaling of time $X=x,Y=-\frac{b_{10}}{a_{10}}x+y, dt=\frac{1}{a_{01}}d\tau$, we have that
\begin{equation}\label{a0000}
\begin{array}{rrl}
&&\dfrac{dx}{dt}=y+c_{20}x^{2}+c_{11}xy+o(|x,y|^{2}),\vspace{2mm}\\
&&\dfrac{dy}{dt}=d_{20}x^{2}+d_{11}xy+d_{02}y^2+o(|x,y|^{2}),
\end{array}
\end{equation}
where
\begin{eqnarray*}
&&c_{20}=\frac{a_{11} b_{10}+a_{10} a_{20}}{a_{01} a_{10}},\ \
c_{11}=\frac{a_{11}}{a_{01}},\ \
d_{20}=\frac{a_{10}^2 b_{20}+a_{10} b_{10} (b_{11}-a_{20})+b_{10}^2 (b_{02}-a_{11})}{a_{01} a_{10}^2},\\
&&d_{11}=\frac{a_{10} b_{11}-b_{10} (a_{11}-2 b_{02})}{a_{01} a_{10}},\ \
d_{02}=\frac{b_{02}}{a_{01}}.
\end{eqnarray*}

By Lemma 3.1 in Perko \cite{perko2001nonlinear}, system (\ref{a0000}) is equivalent to
\begin{eqnarray*}
\nonumber&&\frac{dx}{dt}=y+o(|x,y|^{2}),\\
&&\frac{dy}{dt}=e_{20}x^{2}+e_{11}xy+o(|x,y|^{2}),
\end{eqnarray*}
where $e_{20}=d_{20}=\frac{s (a+x^{*} (c+x^{*}))^2}{2 q^2 x^{*2}} F''(x^{*}),\ e_{11}=d_{11}+2c_{20}=-\frac{a+x^{*} (c+x^{*})}{q x^{*}}p'(x^{*})$. From $F''(x^{*})\neq 0$ and $p'(x^{*})\neq 0$, $E^{*}=(x^{*}, nx^{*})$ is a cusp of codimension 2.

Finally, the proof of statement $(III)$. Since $F(x^{*})=F'(x^{*})=p(x^{*})=p'(x^{*})=0$, we have
\begin{equation}\label{1000}
\begin{array}{rrl}
&&r=\dfrac{s}{2 x^{*2} (a+x^{*} (c+x^{*}))^2}\left(a^2 (b+x^{*}) (b+2 x^{*})+a x^{*} (3 b^2+9 b x^{*}+8 x^{*2}) (c+2 x^{*})\right.\vspace{2mm}\\
&& \ \ \ \ \ \ \left. -x^{*3}  (b^2 (c+3 x^{*})+3 b x^{*} (c+3 x^{*})+2 x^{*} (-c^2-2 c x^{*}+x^{*2}))\right), \vspace{2mm}\\
&&m=\dfrac{s (b+x^{*})^3 (a^2+3 a x^{*} (c+2 x^{*})+x^{*3} (-(c+3 x^{*})))}{2 x^{*2} (a+x^{*} (c+x^{*}))^2},\
n=\dfrac{s (a+x^{*} (c+x^{*}))}{q x^{*}}, \vspace{2mm}\\
&&k=\dfrac{1}{a^2 (b-x^{*})+a x^{*} (3 b (c+2 x^{*})+x^{*} (c+6 x^{*}))+x^{*3} (x^{*} (x^{*}-c)-(b (c+3 x^{*})))}\vspace{2mm}\\
&& \ \ \ \ \ \ \times \left(a^2 (b +x^{*}) (b+2 x^{*})+a x^{*} (3 b^2+9 b x^{*}+8 x^{*2}) (c+2 x^{*})-x^{*3} (b^2 (c+3 x^{*})\right. \vspace{2mm}\\
&& \ \ \ \ \ \ \left. +3 b x^{*} (c+3 x^{*})+2 x^{*} (-c^2-2 c x^{*}+x^{*2}))\right).
\end{array}
\end{equation}

Similar to $(II)$, we can translate the positive equilibrium point $E^{*}=(x^{*}, nx^{*})$ to the origin by introducing new variables $X=x-x^{*}$ and $Y=y-y^{*}$. For convenience, in the subsequent steps, we will continue to denote the variables as $x$, $y$, and $t$ instead of $X$, $Y$, and $\tau$, respectively. We then perform a Taylor expansion around the origin, resulting in the following equivalent system:
\begin{equation}\label{c4}
\begin{array}{rrl}
&&\dfrac{dx}{dt}=a^{*}_{10}x+a^{*}_{01}y+a^{*}_{20}x^{2}+a^{*}_{11}xy+a^{*}_{30}x^{3}+a^{*}_{21}x^{2}y+a^{*}_{40}x^{4}+a^{*}_{31}x^{3}y+o(|x,y|^{4}),\vspace{2mm}\\
&&\dfrac{dy}{dt}=b^{*}_{10}x+b^{*}_{01}y+b^{*}_{20}x^{2}+b^{*}_{11}xy+b^{*}_{02}y^{2}+b^{*}_{30}x^{3}+b^{*}_{21}x^{2}y
+b^{*}_{12}xy^{2}+b^{*}_{40}x^{4}\\
&& \ \ \ \ \ \ \ \ +b^{*}_{31}x^{3}y+b_{22}x^{2}y^{2}+o(|x,y|^{4}),
\end{array}
\end{equation}
where $a^{*}_{ij}$ and $b^{*}_{ij}$ are given in \textbf{Appendix \ref{AppB}}.

Setting $X=x, Y=\frac{dx}{dt}$, we rewrite system (\ref{c4}) as
\begin{equation}\label{c5}
\begin{array}{rrl}
&&\dfrac{dx}{dt}=y, \vspace{2mm}\\
&&\dfrac{dy}{dt}=c^{*}_{20}x^{2}+c^{*}_{02}y^{2}+c^{*}_{30}x^{3}+c^{*}_{21}x^{2}y+c^{*}_{12}xy^{2}+c^{*}_{40}x^{4}
+c^{*}_{31}x^{3}y+c^{*}_{22}x^{2}y^{2}\\
&& \ \ \ \ \ \ \ \ +o(|x,y|^{4}),
\end{array}
\end{equation}
where $c^{*}_{ij}$ are given in \textbf{Appendix \ref{AppB}}.

By a time rescaling $dt=(1-c^{*}_{02}x)d\tau$ and a linear transformation $X=x, Y=(1-c^{*}_{02}x)y$ in~(\ref{c5}), we get
\begin{equation}\label{c6}
\begin{array}{rrl}
&&\dfrac{dx}{dt}=y,\vspace{2mm}\\
&&\dfrac{dy}{dt}=c^{*}_{20}x^{2}+(c^{*}_{30}-2c^{*}_{02}c^{*}_{20})x^{3}+c^{*}_{21}x^{2}y
+(c^{*}_{12}-c^{*2}_{02})xy^2+(c^{*2}_{02} c^{*}_{20}-2c^{*}_{02}c^{*}_{30}+c^{*}_{40})x^{4}\\
&& \ \ \ \ \ \ \ \ +(c^{*}_{31}-c^{*}_{02}c^{*}_{21})x^{3}y+(c^{*}_{22}-c^{*3}_{02})x^{2}y^{2}+o(|x,y|^{4}).
\end{array}
\end{equation}

Since $c^{*}_{20}=\frac{s^2 (x^{*2}-a)}{2 x^{2} (a+x^{2} (c+x^{2}))}= \frac{s}{2}F''(x^{*}) \neq 0$, by transforming $X=x, Y=\frac{y}{\sqrt{c^{*}_{20}}}, \tau=\sqrt{c^{*}_{20}} t$ (if $c^{*}_{20}>0$) or $X=-x, Y=-\frac{y}{\sqrt{-c^{*}_{20}}}, \tau=\sqrt{-c^{*}_{20}} t$ (if $c^{*}_{20}<0$), system~(\ref{c6}) can be rewritten as
\begin{equation}\label{c7}
\begin{array}{rrl}
&&\dfrac{dx}{dt}=y,\vspace{2mm}\\
&&\dfrac{dy}{dt}=x^{2}\pm\frac{c^{*}_{30}-2 c^{*}_{02} c^{*}_{20}}{c^{*}_{20}}x^{3}+\frac{c^{*2}_{02} c^{*}_{20}-2 c^{*}_{02} c^{*}_{30}+c^{*}_{40}}{c^{*}_{20}} x^{4}+y\left(\frac{c^{*}_{21}}{\sqrt{\pm c^{*}_{20}}} x^{2}\pm \frac{c^{*}_{31}-c^{*}_{02} c^{*}_{21}}{\sqrt{\pm c^{*}_{20}}} x^{3}\right)\\
&& \ \ \ \ \ \ \ \ +(c^{*}_{12}-c^{*2}_{02})xy^{2}\pm (c^{*}_{22}-c^{*3}_{02})x^{2}y^{2}+o(|x,y|^{4}).
\end{array}
\end{equation}

By Proposition 5.3 in Lemontagne et al. \cite{2008Bifurcation}, an equivalent form of~(\ref{c7}) can be obtained as follows
\begin{eqnarray*}
\nonumber&&\frac{dx}{dt}=y,\\
&&\frac{dy}{dt}=x^{2}+Fx^{3}y+o(|x,y|^{4}),
\end{eqnarray*}
where
$$
F=\frac{c^{*}_{21} (c^{*}_{02} c^{*}_{20}-c^{*}_{30})+c^{*}_{20} c^{*}_{31}}{(\pm c^{*}_{20})^{\frac{3}{2}}}
=\frac{s^3 A}{4 x^{*4} (b+x^{*})^2 (a+x^{*} (c+x^{*}))^6 (\pm c^{*}_{20})^{\frac{3}{2}}},
$$
and $A=a^6 (2 b^2-x^{*2})+a^5 x^{*} (2 b^2 (5 c+3 x^{*})+b x^{*} (c-20 x^{*})-9 x^{*2} (c+4 x^{*}))+a^4 x^{*2} (3 x^{*2} (24b^2+35 b c-3 c^2)+20 b^2 c^2+x^{*3} (184 b-63 c)+b c x^{*} (56 b+29 c)-17 x^{*4})+a^3 x^{*4} (-2 b^2(7 c^2+14 c x^{*}+30 x^{*2})+b (43 c^3+280 c^2 x^{*}+670 c x^{*2}+424 x^{*3})+x^{*} (11 c^3+78 c^2 x^{*}+330 c x^{*2}+344 x^{*3}))-a^2 x^{*4} (24 b^2 c^4+2 x^{*4} (275 b^2+569 b c+18 c^2)+c x^{*3} (952 b^2+450b c+9 c^2)+b c^3 x^{*}194 b+9 c)+b c^2 x^{*2} (618 b+91 c)+6 x^{*5} (160 b+39 c)+363 x^{*6})+ax^{*6} (8 b^2 c^4+5 x^{*4}(62 b^2+45 b c-6 c^2)+c x^{*3} (434 b^2+40 b c+c^2)+2 b c^3 x^{*} (34 b+3 c)+7 b c^2 x^{*2} (34 b+3 c)+x^{*5} (332 b-33 c)+76 x^{*6})+x^{*9} (-3 x^{*3} (12 b^2-3 b c+c^2)+c x^{*2} (-32 b^2+5 b c-3 c^2)-b c^3 (2 b+c)-5 bc^2 x^{*} (2 b+c)+3 x^{*4} (3 c-8 b)-3 x^{*5})$.

Taking $b\neq b_{\pm}$ into consideration, we can conclude that $A\neq 0$ and thus $E^{*}=(x^{*}, y^{*})$ is a cusp of codimension 3. This completes the proof.
\end{proof}

\section{Bifurcation analysis}

Based on the cases $(II)$ and $(III)$ of Theorem \ref{tha1}, we observe that system (\ref{A1}) may undergo Bogdanov-Takens bifurcations of codimension two and three at the equilibrium $E^*=(x^{*},y^{*})$, if it exists. These bifurcations will be discussed in detail in this section.

\subsection{Bogdanov-Takens bifurcation of codimension 2}

\begin{theorem}\label{tha2}
 Let $F$ be as in \eqref{a00} and let $p$ be as in \eqref{a0}.
If $F(x^{*})=F'(x^{*})=p(x^{*})=0$, $F''(x^{*})\neq 0, p'(x^{*})\neq 0$ and
$$
\begin{array}{rcl}
k &\neq  &\dfrac{1}{a^2 b+a x^{*} (2 b (c+2 x^{*})+x^{*} (c+4 x^{*}))+x^{*4} (c-b)}\left(a^2 (b+2 x^{*})^2+a x^{*} (2 b^2 (c+2 x^{*})\right.\\
&&\left. +b x^{*} (7 c+12 x^{*})+4 x^{*2} (2 c+3 x^{*}))- (b-c)x^{*4} (b+2 c+4 x^{*})\right),
\end{array}
$$
 then $E^{*}=(x^{*}, nx^{*})$ is a cusp of codimension 2. Moreover, if we choose $n$ and $m$ as bifurcation parameters, then system (\ref{A1}) undergoes a Bogdanov-Takens bifurcation of codimension 2.
\end{theorem}

\begin{proof}
Selecting $n$ and $m$ as bifurcation parameters,  we obtain the unfolding system as follows
\begin{equation}\label{a555}
\begin{array}{rrl}
&&\dfrac{dx}{dt}=\left(r \left(1-\dfrac{x}{k}\right)-\dfrac{m+\delta_{1}}{x+b}\right)x-\dfrac{q x y}{x^2+c x+a},\vspace{2mm}\\
&&\dfrac{dy}{dt}=sy \left(1-\dfrac{y}{(n+\delta_{2})x}\right),
\end{array}
\end{equation}
where $r,s,k,m,a,b,q,n>0, c>-2\sqrt{a}$ and $\delta=(\delta_{1},\delta_{2})\sim(0,0)$.

Since $F(x^{*})=F'(x^{*})=p(x^{*})=0$ we know that $r=-\frac{k s (a (b+2 x^{*})+x^{*2} (c-b))}{x^{*} (a+x^{*} (c+x^{*})) (b-k+2 x^{*})}, m= -\frac{s (b+x^{*})^2 (a k+x^{*2} (c-k+2 x^{*}))}{x^{*} (a+x^{*} (c+x^{*})) (b-k+2 x^{*})}$ and $n=\frac{s (a+x^{*} (c+x^{*}))}{q x^{*}}$.  Next, we apply the transformation $X=x-x^{*}$ and $Y=y-nx^{*}$ to~(\ref{a555}) and then expand the resulting equations using Taylor series around the origin. As a result, system (\ref{a555}) is converted into the following form (For convenience, in each subsequent transformation, we will rename $X$, $Y$, and $\tau$ as $x$, $y$, and $t$, respectively):
\begin{equation}\label{a6}
\begin{array}{rrl}
&&\dfrac{dx}{dt}=\bar{a}_{00}+\bar{a}_{10}x+\bar{a}_{01}y+\bar{a}_{20}x^{2}+\bar{a}_{11}xy
+R_{1}(x,y,\delta_{1},\delta_{2}),\vspace{2mm}\\
&&\dfrac{dy}{dt}=\bar{b}_{00}+\bar{b}_{10}x+\bar{b}_{01}y+\bar{b}_{20}x^{2}+\bar{b}_{11}xy
+\bar{b}_{02}y^{2}+R_{2}(x,y,\delta_{1},\delta_{2}).
\end{array}
\end{equation}
Here $\bar{a}_{ij}$ and $\bar{b}_{ij}$ are given in \textbf{Appendix \ref{AppB}}, $R_{1}(x,y,\delta_{1},\delta_{2})$ and $R_{2}(x,y,\delta_{1},\delta_{2})$ represent polynomials in $(x,y,\delta_1, \delta_2)$ of degree of at least 3.

Performing the nonsingular change of coordinate $X=x,Y=\frac{dx}{dt}$, system~\eqref{a6} is transformed into
\begin{equation}\label{a7}
\begin{array}{rrl}
&&\dfrac{dx}{dt}=y,\vspace{2mm}\\
&&\dfrac{dy}{dt}=\bar{c}_{00}+\bar{c}_{10}x+\bar{c}_{01}y+\bar{c}_{20}x^{2}+\bar{c}_{11}xy
+\bar{c}_{02}y^{2}+R_{3}(x,y,\delta_{1},\delta_{2}),
\end{array}
\end{equation}
where $\bar{c}_{ij}$ are given in \textbf{Appendix \ref{AppB}} and $R_{3}(x,y,\delta_{1},\delta_{2})$
represents a polynomial in $(x,y,\delta_1, \delta_2)$ of degree of at least 3.

Next, we can perform a time re-parametrization by introducing $dt=(1-\bar{c}_{02}x)d\tau$, along with the transformation $X=x, Y=(1-\bar{c}_{02}x)y$. This allows us to change system (\ref{a7}) into the following form:
\begin{equation}\label{a9}
\begin{array}{rrl}
&&\dfrac{dx}{dt}=y,\vspace{2mm}\\
&&\dfrac{dy}{dt}=\bar{d}_{00}+\bar{d}_{10}x+\bar{d}_{01}y+\bar{d}_{20}x^{2}+\bar{d}_{11}xy
+R_{4}(x,y,\delta_{1},\delta_{2}),
\end{array}
\end{equation}
where $\bar{d}_{ij}$ are given in \textbf{Appendix \ref{AppB}}  and $R_{4}(x,y,\delta_{1},\delta_{2})$
is a polynomial in $(x,y,\delta_1, \delta_2)$ of degree of at least 3.

Letting $X=x+\frac{\bar{d}_{10}}{2\bar{d}_{20}}, Y=y$, we obtain that
\begin{equation}\label{a1010}
\begin{array}{rrl}
&&\dfrac{dx}{dt}=y,\vspace{2mm}\\
&&\dfrac{dy}{dt}=\bar{e}_{00}^{*}+\bar{e}_{01}y+\bar{e}_{20}x^{2}+\bar{e}_{11}xy+R_{5}(x,y,\delta_{1},\delta_{2}),
\end{array}
\end{equation}
where $\bar{e}_{ij}$ are given in \textbf{Appendix \ref{AppB}} and $R_{5}(x,y,\delta_{1},\delta_{2})$ is a polynomial in $(x,y,\delta_1, \delta_2)$ of degree of at least 3.

Letting $X=\frac{\bar{e}_{11}^{2}}{\bar{e}_{20}}x, Y=\frac{\bar{e}_{11}^{3}}{\bar{e}_{20}^{2}} y, \tau=\frac{\bar{e}_{20}}{\bar{e}_{11}}t$,   system (\ref{a1010}) is equivalent to the following system
\begin{equation}\label{a11}
\begin{array}{rrl}
&&\dfrac{dx}{dt}=y,\vspace{2mm}\\
&&\dfrac{dy}{dt}=\varphi_{1}+\varphi_{2}y+x^{2}+xy+R_{6}(x,y,\delta_{1},\delta_{2}),
\end{array}
\end{equation}
where $R_{6}(x,y,\delta_{1},\delta_{2})$ represents a polynomial in $(x,y,\delta_1, \delta_2)$ of degree of at least 3 with coefficients that depend smoothly on $\delta_{1}$ and $\delta_{2}$ and $\varphi_{1}=\frac{\bar{e}_{00}\bar{e}_{11}^{4}}{\bar{e}_{20}^{3}}, \ \varphi_{2}=\frac{\bar{e}_{01}\bar{e}_{11}}{\bar{e}_{20}}$. Using the software Mathematica, we can calculate that
\begin{eqnarray*}
\left|\frac{\partial(\varphi_{1},\varphi_{2})}{(\delta_{1},\delta_{2})}\right|_{\delta=0}
=-\frac{q u_2 u_3^5}{2 s^2 x^{*4} (b+x^{*})^2 (a+x^{*} (c+x^{*}))^3 (b-k+2 x^{*}) u_1^5},
\end{eqnarray*}
where
\begin{eqnarray*}
&&u_1=a^2 k+a (b^2 (c+3 x^{*})-b (c+3 x^{*}) (k-3 x^{*})+x^{*2} (3 c-3 k+8 x^{*}))+x^{*3} (-(b-c)) (b+\\
&& \ \ \ \ \ \ \ c-k+3 x^{*}),\\
&&u_2=a^2 (b^2-b k+4 b x^{*}+4 x^{*2})+a x^{*} (4 x^{*2} (3 b+2 c-k)-k x^{*}(4b+c)+2 b c (b-k)+b x^{*} (4\\
&& \ \ \ \ \ \ \ \times b+7 c)+12 x^{*3})+x^{*4} (c-b) (b+2 c-k+4 x^{*}),\\
&&u_3=a^2 (b^2-b (k-3 x^{*})+x^{*} (k+2 x^{*}))+a x^{*} (3 b^2 (c+2 x^{*})-3 b (c+2 x^{*}) (k-3 x^{*})+x^{*}(-c\\
&& \ \ \ \ \ \ \ \times k+8 c x^{*}-6 k x^{*}+16 x^{*2}))+x^{*3} (-b^2 (c+3 x^{*})+b (c+3 x^{*}) (k-3 x^{*})+x^{*} (2 c^2-c k\\
&& \ \ \ \ \ \ \ +4 c x^{*}+k x^{*}-2 x^{*2})).
\end{eqnarray*}

It follows from the hypotheses $k \neq \frac{1}{a^2 b+a x^{*} (2 b (c+2 x^{*})+x^{*} (c+4 x^{*}))+x^{*4} (c-b)} (a^2 (b$ $+2 x^{*})^2+a x^{*} (2 b^2 (c+2 x^{*})+b x^{*} (7 c+12 x^{*})+4 x^{*2} (2 c+3 x^{*}))-x^{*4} (b-c) (b+2 c+4 x^{*}))$, $F''(x^{*})\neq0$, and $p'(x^{*})\neq 0$ that
$u_1\neq0, u_2\neq0$ and $u_3 \neq 0$. By the results of Bogdanov \cite{bogdanov1981bifurcation,1975Versal} and Takens~\cite{takens2001forced}, we conclude that system (\ref{A1}) undergoes a Bogdanov-Takens bifurcation of codimension 2 when $(\delta_{1},\delta_{2})$ changes in a small neighborhood of $(0, 0)$.
\end{proof}

\subsection{Bogdanov-Takens bifurcation of codimension 3}

 This subsection focuses on the Bogdanov-Takens bifurcation of codimension 3. In order to facilitate the understanding of the analysis process, we first present the relevant definition and property. Please refer to Perko \cite{perko2001nonlinear} and Li et al. \cite{2021Incoherent} for further details.

\begin{definition}\label{def3.1}
The bifurcation that results from unfolding the following normal form of a cusp of codimension 3,
\begin{equation}\label{b2}
\begin{array}{rrl}
&&\dfrac{dx}{dt}=y,\vspace{2mm}\\
&&\dfrac{dy}{dt}=x^{2}\pm x^{3}y,
\end{array}
\end{equation}
is called a cusp type degenerate Bogdanov-Takens bifurcation of codimension 3.
\end{definition}

\begin{proposition}\label{pro1}
A universal unfolding of the normal form (\ref{b2}) is expressed by
\begin{eqnarray}\label{b1}
\begin{array}{l}
\left\{
\begin{array}{l}
\dfrac{dx}{dt}=y, \vspace{2mm}\\
\dfrac{dy}{dt}=\zeta_{1}+\zeta_{2}y+\zeta_{3}xy+x^{2}\pm x^{3}y+R(x,y,\rho),
\end{array}
\right.
\end{array}
\end{eqnarray}
where $\rho=(\rho_{1}, \rho_{2}, \rho_{3})\sim(0,0,0)$, $\frac{D(\zeta_{1},\zeta_{2},\zeta_{3})}{D(\rho_{1},\rho_{2},\rho_{3})} \neq 0$ for small $\rho$ and
\begin{eqnarray}\label{a13}
\nonumber&&R(x,y,\rho)=y^{2}O(|x,y|^{2})+O(|x,y|^{5})+O(\rho)(O(y^{2})+O(|x,y|^{3}))+O(\rho^{2})O(|x,y|).
\end{eqnarray}
\end{proposition}
Our main result is described in the following Theorem.

\begin{theorem}\label{tha33}
 Let $F$ be as in \eqref{a00} and let $p$ be as in \eqref{a0}, and define $b_{\pm}$ as in~\eqref{a22}.
If $F(x^{*})=F'(x^{*})=p(x^{*})=p'(x^{*})= 0, F''(x^{*}) \neq 0$ and $b\neq b_{\pm}$, then $E^{*}=(x^{*}, nx^{*})$ is a cusp of codimension 3. Further, if we choose $n, b$ and $m$ as bifurcation parameters and assume that $b \neq b^{*}_{\pm}$, where
\begin{eqnarray}\label{100}
b^{*}_{\pm}=\frac{x^{*} (-m\pm \sqrt{n} (a+x^{*} (c+x^{*})))}{2 l}
\end{eqnarray}
with $l=2 a^3+a^2  (7 c+6 x^{*})+a x^{*} (9 c^2+28 c x^{*}+30 x^{*2})-x^{*4} (c^2+3 c x^{*}+6 x^{*2}),
m=a^3+6 a^2 c x^{*}-a^2 x^{*2}+13 a c^2 x^{*2}+40 a c x^{*3}+51 a x^{*4}-3 c^2 x^{*4}-6 c x^{*5}-11 x^{*6},
n=25 a^4+2 a^3 x^{*} (43 c+46 x^{*})+a^2 x^{*2} (97 c^2+274 c x^{*}+286 x^{*2})-14 a x^{*4} (5 c^2+17 c x^{*}+22 x^{*2})+x^{*6} (9 c^2+22 c x^{*}+49 x^{*2})$, then model (\ref{A1}) undergoes a Bogdanov-Takens bifurcation of codimension 3 at $E^{*}$.
\end{theorem}

\begin{proof}
See \textbf{Appendix \ref{AppC}}.
\end{proof}

\subsection{Hopf bifurcation of codimension 3}

Throughout this section we will denote $E_{*}=(x_{*},nx_{*})$ to either $E_1=(x_{1},y_{1})$ or $E_3=(x_{3},y_{3})$ for convenience.
According to the analysis in \textbf{Appendix \ref{AppA}}, we can conclude that a Hopf bifurcation may occur
at either $E_*$  since $\det(J(E_{*}))=-sF'(x_{*})>0$.  Our conclusions are presented as follows.

\begin{theorem}\label{tha3}
Let $E_{*}=(x_{*},nx_{*})$ be an equilibrium of (\ref{A1}) accounting for either $E_1=(x_{1},y_{1})$ or $E_3=(x_{3},y_{3})$. Assume that $F(x_{*})=p(x_{*})=0$ and $F'(x_{*})<0$. 
%
%
 %
 %
Finally, define
\begin{eqnarray}\label{A33}
\eta_{11}=-(r^2 x^{3}_{*} \eta^{0}_{11}+k r x_{*}s \eta^{1}_{11}+k^2 s^{2} \eta^{2}_{11})
\end{eqnarray}
where $\eta^{0}_{11}, \eta^{1}_{11}, \eta^{2}_{11}$ are specified in \textbf{Appendix \ref{AppD}}. Then we have:\\
$(I)$ if $\eta_{11}<0$, then $E_{*}$ is a stable weak focus with multiplicity one and one stable limit cycle bifurcates from $E_{*}$ in a supercritical Hopf bifurcation;\\
$(II)$ if $\eta_{11}>0$, then $E_{*}$ is an unstable weak focus with multiplicity one and one unstable limit cycle bifurcates from $E_{*}$ in a subcritical Hopf bifurcation;\\
$(III)$ if $\eta_{11}=0$, then $E_{*}$ is a weak focus with multiplicity at least two and system (\ref{A1}) may exhibit a degenerate Hopf bifurcation.
\end{theorem}

\begin{proof}
According to~\eqref{a0}, we obtain that $\det(J(E_{*}))={-sF'(x_{*})}>0$. Moreover, $F(x_{*})=p(x_{*})=0$ implies $m=-\frac{(b+x_{*})^2}{k x_{*} (-a+b (c+2 x_{*})+x_{*}^2)}$ $\times(a (k s+r x_{*})+x_{*} (c k (s-r)+2 c r x_{*}+k x_{*} (s-2 r)+3 r x_{*}^2))$ and $n=\frac{(a+x_{*} (c+x_{*}))^2}{k q x_{*}^2 (-a+b (c+2 x_{*})+x_{*}^2)} (b (k s+r x_{*})+x_{*} (k (s-r)+2 r x_{*}))$ and, hence,  $\mathrm{tr}(J(E_{*}))=0$.


Performing the coordinate transformation $X=x-x_{*}, Y=y-nx_{*}$ to shift the positive equilibria $E_{*}$ to the origin and expressing the resulting system around the origin
by Taylor expansion,  we have (for convenience, we rename $X,Y$ as $x,y$, respectively):
\begin{equation}\label{a30}
\begin{array}{rcl}
\dfrac{dx}{dt}&=&\alpha_{10}x+\alpha_{01}y+\alpha_{20}x^{2}+\alpha_{11}xy+\alpha_{30}x^{3}+\alpha_{21}x^{2}y
+\alpha_{40}x^{4}+\alpha_{31}x^{3}y\\
&&+o(|x,y|^{4}),\\
\dfrac{dy}{dt}&=&\beta_{10}x-\alpha_{10}y+\beta_{20}x^{2}+\beta_{11}xy+\beta_{02}y^{2}+\beta_{30}x^{3}
+\beta_{21}x^{2}y+\beta_{12}xy^{2}+\beta_{40}x^{4}\\
&& +\beta_{31}x^{3}y+\beta_{22}x^{2}y^{2}+o(|x,y|^{4}),
\end{array}
\end{equation}
where $\alpha_{ij}$ and $\beta_{ij}$ are specified in \textbf{Appendix \ref{AppD}}.

In view of $-\alpha_{10}^2-\alpha_{01}\beta_{10}=\det(J(E_{*}))>0$, we define $\omega=\sqrt{-\alpha_{10}^{2}-\alpha_{01}\beta_{10}}$. Through the scale transformations $x=-\alpha_{01}X, y=\alpha_{10}X-\omega Y$ and $dt=\frac{1}{\omega}d\tau$, system (\ref{a30}) is   represented by (we still denote $X,Y,\tau$ by $x,y,t$, respectively):
\begin{eqnarray}
\nonumber&&\frac{dx}{dt}=y+\delta_{20}x^{2}+\delta_{11}xy+\delta_{30}x^{3}+\delta_{21}x^{2}y+\delta_{40}x^{4}+\delta_{31}x^{3}y+o(|x,y|^{4}),\\
\nonumber&&\frac{dy}{dt}=-x+\gamma_{20}x^{2}+\gamma_{11}xy+\gamma_{02}y^{2}+\gamma_{30}x^{3}+\gamma_{21}x^{2}y+\gamma_{12}xy^{2}+\gamma_{40}x^{4}\\
&&\nonumber \ \ \ \ \ \ \ \ +\gamma_{31}x^{3}y+\gamma_{22}x^{2}y^{2}+o(|x,y|^{4}),
\end{eqnarray}
where $\delta_{ij}$ and $\gamma_{ij}$ are given in \textbf{Appendix \ref{AppD}}.

By the formula in Perko \cite{perko2001nonlinear}, we   obtain the first Lyapunov coefficient as follows:
\begin{eqnarray*}
\eta_{1}=\frac{q^2 s\ \eta _{11}}{8 \omega ^3 k^2 x_{*}  (b+x_{*})^2 (a+c x_{*}+x_{*}^2)^4 (-a+b c+2 b x_{*}+x_{*}^2)^2},
\end{eqnarray*}
where $\eta_{11}$ is defined in (\ref{A33}).
Since all the parameters and $\omega$ are greater than zero, the sign of $\eta_{1}$ is the same as that of $\eta_{11}$ and thus we complete the proof.
\end{proof}

Based on case $(III)$ of Theorem \ref{tha3}, we further analyze the higher codimension that can be reached in a Hopf bifurcation at $E_{*}$. Under the condition $\eta_{11}=0$, we can calculate the second Lyapunov coefficient by Maple and Mathematica as follows:
\begin{eqnarray*}
\eta_{2}=\frac{q^4 s^2 \ \eta _{22}}{288 k^5 x_{*}^3 \omega ^7 (b+x_{*})^4 (a+c x_{*}+x_{*}^2)^8 A},
\end{eqnarray*}
where $A=(-a+b c+2 b x_{*}+x_{*}^2)^5 (b k s+b r x_{*}-k r x_{*}+k s x_{*}+2 r x_{*}^2)$ and $\eta _{22}$ is too long to be included here. In particular $n>0$ implies $A>0$. Specifically, we come to the following conclusion.

\begin{theorem}\label{tha4}
Assume that $E_{*}=(x_{*},nx_{*})$ is an equilibrium of (\ref{A1}) accounting for either $E_1=(x_{1},y_{1})$ or $E_3=(x_{3},y_{3})$. Suppose that $F(x_{*})=p(x_{*})=0, F'(x_{*})<0$ and $\eta_{11}=0$ as defined in (\ref{A33}). Then the following statements hold:\\
$(I)$ If $\eta_{22}<0$, then $E_{*}$ is a stable weak focus with multiplicity 2. System (\ref{A1}) undergoes a degenerate Hopf bifurcation of codimension 2 and there can be up to two limit cycles bifurcating from $E_{*}$, the outermost being stable;\\
$(II)$ If $\eta_{22}>0$,  then $E_{*}$ is an unstable weak focus with multiplicity 2. System (\ref{A1}) undergoes a degenerate Hopf bifurcation of codimension 2 and there can be up to two limit cycles bifurcating from $E_{*}$, the outermost being unstable;\\
$(III)$ If $\eta_{22}=0$, then $E_{*}$ is a weak focus with multiplicity at least 3 and system (\ref{A1}) may undergo a degenerate Hopf bifurcation of codimension at least 3.
\end{theorem}

\section{Numerical bifurcation analysis}

In this section, we will carry out numerical bifurcation analysis with AUTO07P~\cite{doedel2007auto} on system \eqref{A1}. Due to the exploratory nature of this part, parameter values are initially chosen  to verify and complement the theoretical results.

\subsection{Mushroom and isola dynamics of limit cycles}

\textbf{Figure \ref{bx}} shows $b$ as the primary bifurcation parameter.
The black curve corresponds to an equilibrium branch which contains two supercritical Hopf bifurcation points labelled as $HB_{1}$ and $HB_{2}$, respectively; in particular, the point $HB_{1}$ is located very close to a fold or saddle-node point (not labelled). The equilibrium curve terminates at a transcritical bifurcation $TC$ at which this positive steady state collides with the equilibrium at the origin.
The branches of limit cycles bifurcated from the two Hopf points form a single (green) curve with a ``mushroom''-like shape. The presence of this curve is usually referred to as a {\it mushroom bifurcation}~\cite{2021Incoherent, otero2023automated}.
This curve exhibits multiple saddle-node bifurcation points of limit cycles (labelled as $SN_i$, $i=1,2,3,4$). There is a narrow interval of values of $b$ between the points $SN_1$ and $SN_2$ for which three concentric limit cycles coexist: A large amplitude stable limit cycle surrounding a
 middle-sized unstable limit cycle that encloses a smaller stable limit cycle, all surrounding an unstable focus. This implies the possibility of oscillatory-type multistability for both populations; the basin boundary between each asymptotic scenario is the unstable periodic orbit. A qualitatively similar dynamic structure occurs for values of $b$ between $SN_4$ and $HB_2$.

\begin{figure}[h!]
\begin{center}
\includegraphics[width=\textwidth]{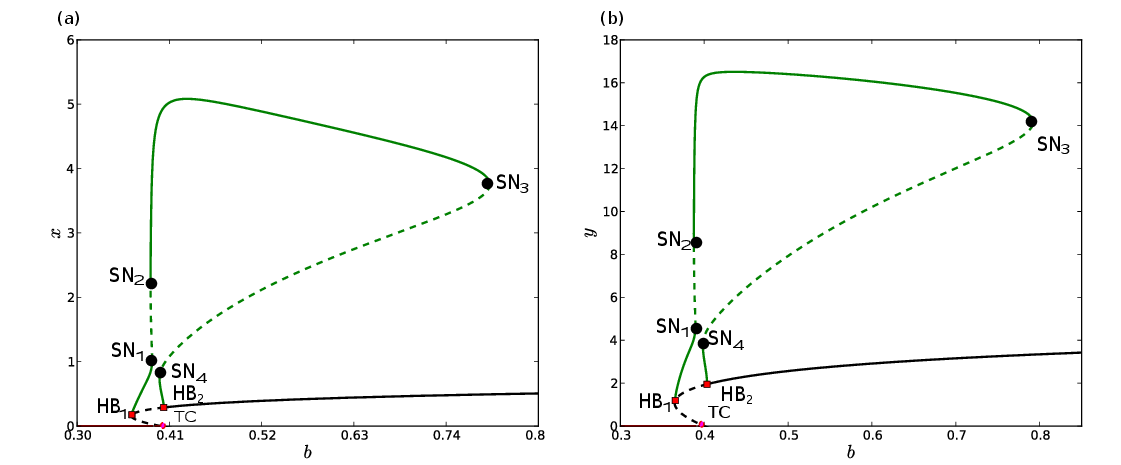}
\end{center}
\caption{ \ One-parameter bifurcation diagram of system (\ref{A1}) with respect to $b$.  Panel (a) shows $b$ vs $x$ and panel~(b) shows $b$ vs $y$. Parameter values are $r=0.5, \ k=6.79211, \ q=0.3, \ a=3, \ s=0.1096,
 \ n=6.752, \ m=0.2, \ c=0.001$.}
\label{bx}
\end{figure}

The existence of these mushroom-like structures usually heralds the  appearance of {\it isolas}~\cite{avitabile2012numerical,gray1985sustained}. Indeed,
under suitable parameter perturbations, the fold points $SN_1$ and $SN_4$ can become progressively closer to one another until the thin ``neck" of the mushroom is cut off effectively creating a closed curve of limit cycle with two folds. {\bf Figure}~\ref{b_vs_m_isolas} shows a sample of such nested structures known as isolas for different values of parameter $m$ within the range $0.20349<m<0.228811$ (The smallest isola in this set corresponds to $m=0.228811$.) The folds at each closed curve correspond to the saddle-node points of limit cycle $SN_2$ and $SN_3$ from {\bf Figure}~\ref{bx}. For each of these isolas, there is a critical value $b$ for which the possible amplitudes of the asymptotic oscillations of both populations reach their maxima. On the other hand, the  smaller $m$ is (i.e., progressively ``weaker'' Allee effect), the larger the possible amplitudes of the stable periodic solutions.

\begin{figure}[h!]
\begin{center}
\includegraphics[width=\textwidth]{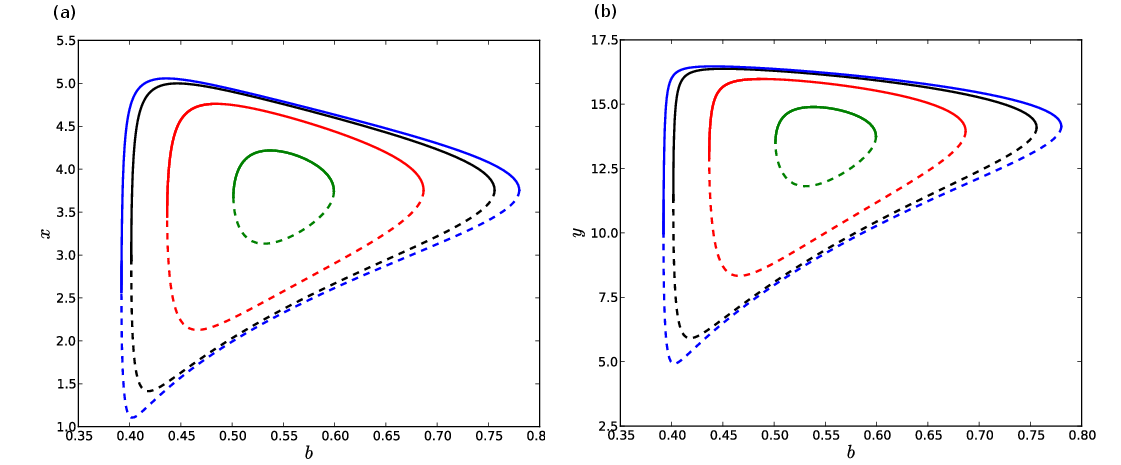}
\end{center}
\caption{ \ A nested of isolas of system (\ref{A1}) with respect to $b$ for $m=2.01387\times10^{-1}$, $m=2.04370\times10^{-1}$, $m=2.14270\times10^{-1}$, and $m=2.26336\times10^{-1}$. (a) $b$ vs. $x$; (b) $b$ vs. $y$. The other parameters are fixed to the same values of \textbf{Figure \ref{bx}}.}
\label{b_vs_m_isolas}
\end{figure}

Considering $b$ and $m$ as the primary bifurcation parameters and keeping the other parameters fixed, we obtain the bifurcation diagram of \textbf{Figure \ref{b_vs_m_all}}(a). It contains a Hopf bifurcation curve H (red), a saddle-node bifurcation curve SN (blue), a homoclinic bifurcation curve Hom (green), and a saddle-node bifurcation curve of limit cycles SNL (black). There are also a Bogdanov-Takens bifurcation point BT, a generalized Hopf point GH and a cusp point of limit cycles CPL. The grey line $m=br$ marks the boundary between  strong (left handside) and weak (right handside) Allee effect on the prey in the absence of predators.

\begin{figure}[h!]
\begin{center}
\includegraphics[width=\textwidth]{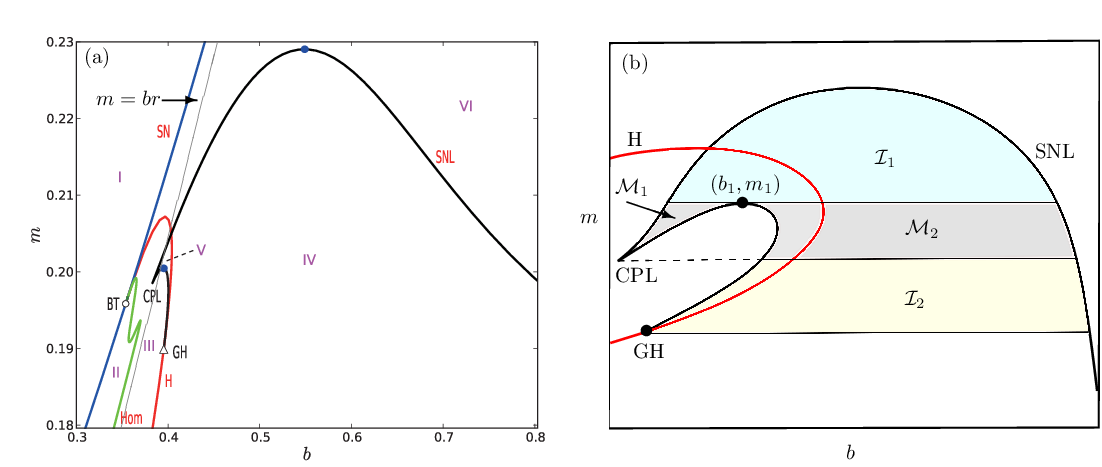}
\end{center}
\caption{ \ (a) Two-parameter bifurcation diagram of (\ref{A1}) with respect to $b$ and $m$. Here SN, H, Hom and SNL denote, respectively, saddle-node bifurcation curve, Hopf bifurcation curve, homoclinic bifurcation curve and saddle-node bifurcation curve of limit cycle. The other parameters are fixed to the same values of \textbf{Figure~\ref{bx}}. Panel (b) is a topological sketch showing the regions for mushroom and isola dynamics.}
\label{b_vs_m_all}
\end{figure}

The entire bifurcation diagram in \textbf{Figure \ref{b_vs_m_all}}(a) is divided into six regions labeled as I to VI. For each region we present representative phase portraits in~\textbf{Figure \ref{bm_phase}}:
Region~I consists of a saddle and an unstable node along the $x$ axis -- as well as the origin -- and no equilibria in the interior of the first quadrant.
As the point $(b,m)$ crosses from Region I to Region~II a saddle-node bifurcation occurs and two new equilibria appear with positive coordinates.
In Region~III a large amplitude stable limit cycle exists -- born at the homoclinic bifurcation along the curve $Hom$ as the point $(b,m)$ enters from Region II -- that encloses an unstable focus.
Passing from Region III to IV involves a subcritical Hopf bifurcation: a small unstable limit cycle appears in IV which is surrounded by the larger (stable) one.
Region V is bounded by the (red) branch of supercritical Hopf bifurcation and the segment of the (black) saddle-node curve of limit cycles which contains the cusp singularity ($CPL$) and emerges from the $GH$ point:  This is the region with the coexistence of three limit cycles and oscillatory multistability described before in {\bf Figure}~\ref{bx}.
Finally, in Region VI, the two larger limit cycles have collided and disappeared as the point $(b,m)$ crosses the $SNL$ curve from Region V.

\textbf{Figure \ref{b_vs_m_all}} is interesting in that the regions for mushroom and isola dynamics are specified in a two-parameter context in panel (b). (Notice also that most of these regions lie in the weak Allee area.) The (gray) disconnected region $\mathcal{M}=\mathcal{M}_1\cup\mathcal{M}_2$ corresponds to mushroom dynamics. (Notice that the one-parameter bifurcation diagrams in {\bf Figure}~\ref{bx} correspond to a slice of $\mathcal{M}$ for $m=0.2$ fixed.)
$\mathcal{M}$ is formed by two open connected components $\mathcal{M}_1$ and $\mathcal{M}_2$, one in Region IV and the other in Region V. Its specification is as follows: Assume that the $SNL$ curve is defined by $\varphi(b,m)=0$, where $\varphi$ is a smooth function on $(b,m)$ such that $\nabla\varphi=0$ at the point CPL. Without loss of generality, assume that regions IV and V lie in the open region $\varphi(b,m)>0$. Then, the open connected subset $\mathcal M_1$  is defined as $\mathcal{M}_1=\{(b,m) \in {\rm V}: m<m_1\} \subset {\rm V}$, where the value $m_1$ satisfies $\varphi(b_1,m_1)=0$ and $\frac{\partial \varphi}{\partial m}(b_{1},m_{1})=0$, indicating that the curve SNL has a fold with respect to $m$ at $(b_{1},m_{1})$. Furthermore, the set $\mathcal{M}_2=\{(b,m) \in {\rm IV}: m_c<m<m_1\} \subset {\rm IV}$, where $m_c$ corresponds to the value of $m$ at the cusp point ${\rm CPL}=(b_c,m_c)$.
On the other hand, Region $\mathcal{I}=\mathcal{I}_1\cup\mathcal{I}_2$ gives rise to isolas of limit cycles. Region $\mathcal{I}$ is a disconnected set composed of two connected components in \textbf{Figure \ref{b_vs_m_all}}. Here $\mathcal{I}_1=\{(b,m) \in {\rm IV}: m_1<m\} $ and $\mathcal{I}_2=\{(b,m) \in {\rm IV}: m_{GH}<m<m_1\} $, where  $m_{GH}$ corresponds to the value of $m$ at the point ${\rm GH}=(b_{GH},m_{GH})$. In particular, the isolas shown in {\bf Figure}~\ref{b_vs_m_isolas} are found for values of $(b,m)\in\mathcal{I}_1$.

\begin{figure}[h!]
\begin{center}
\includegraphics[width=\textwidth]{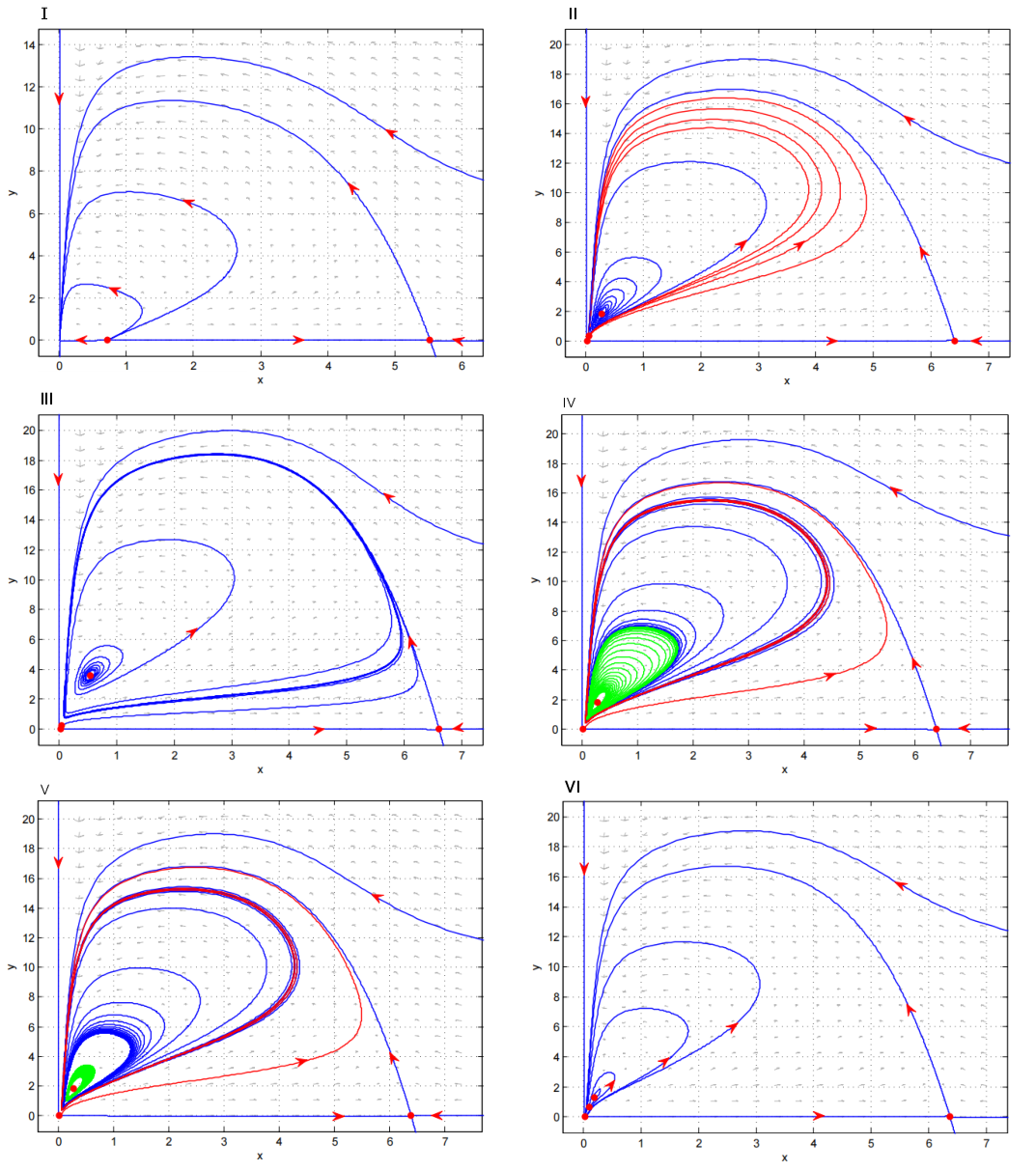}
\end{center}
\caption{ \ Phase portraits of regions for $(b,m)$ in regions I-VI. The other parameters are fixed to the same values of \textbf{Figure \ref{bx}}.}
\label{bm_phase}
\end{figure}

\subsection{Further  isolas}

\begin{figure}[h!]
\begin{center}
\includegraphics[width=\textwidth]{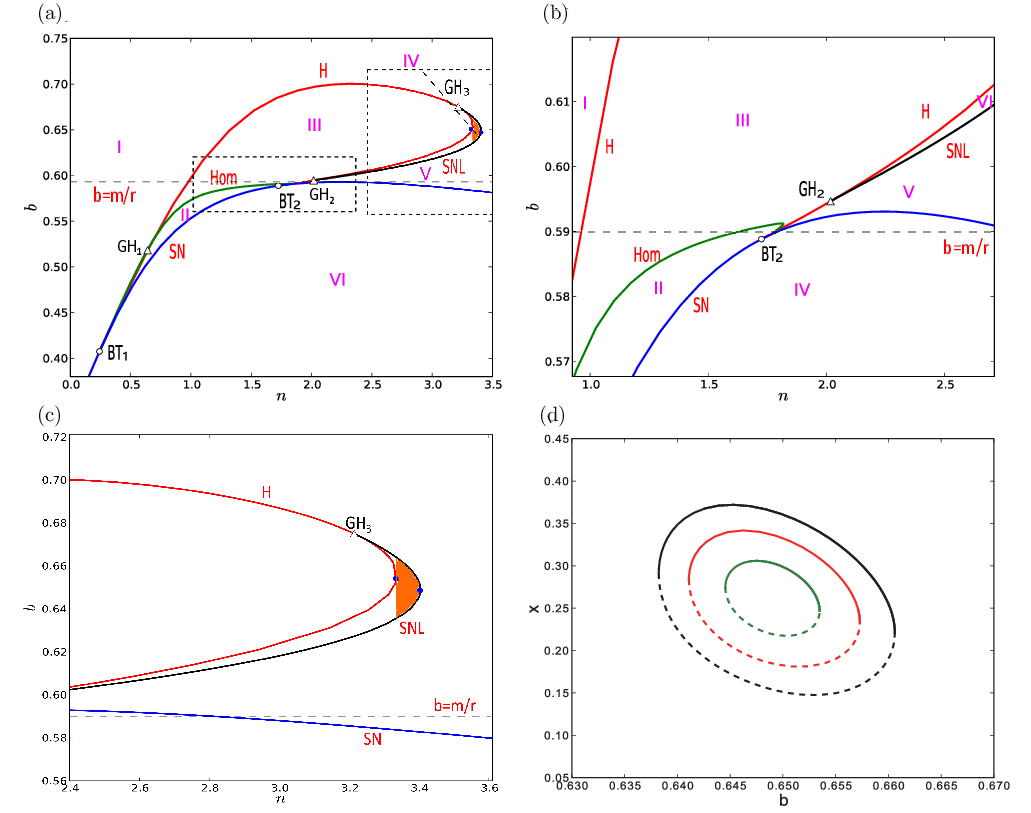}
\end{center}
\caption{\ Two-parameter bifurcation diagram of (\ref{A1}) with respect to $n$ and $b$. Panel (b) is an enlargement of panel~(a) near the codimension two points $BT_2$ and $GH_2$ while panel (c) is an enlargement near the isola region. Panel (d) shows a sample of such isolas of limit cycles. Parameter values are $r=0.65,  k=1.64445,  q=0.25,  a=0.8,  s=0.03, m=0.3855,  c=0.01$.}
\label{n_b_all}
\end{figure}

Considering $n$ and $b$ as the primary control parameters we obtain the bifurcation diagram in \textbf{Figure~\ref{n_b_all}}(a) -- and the enlargements in panels (b) and (c) -- where  two Bogdanov-Takens points and three generalized Hopf points are present. The discontinuous line corresponds to $b=m/r$, i.e., the boundary between weak and strong Allee effects; in particular, the region above this line coincides with the weak Allee area.
The entire bifurcation diagram in \textbf{Figure \ref{n_b_all}}(a) is divided into six regions: I-VI.
The saddle-node bifurcation curve of limit cycle $SNL$ connects two of these generalized Hopf bifurcation points, $GH_2$ and $GH_3$, indicating that the closed region IV between the Hopf bifurcation curve and curve $SNL$ represents a coexistence region of two limit cycles.
The orange region in IV determines the existence of isolas. It can be defined as follows: Let $\psi(n,b)=0$ represent the Hopf bifurcation curve, and let $\varphi(n,b)=0$ denote the $SNL$ curve, where $\psi$ and $\varphi$ are smooth functions. Specifically, assuming that region ${\rm IV}=\{(n,b):\psi(n,b)>0 \ \text{and} \ \varphi(n,b)>0\}$, the isola region is defined as $\mathcal{I}=\{(n,b) \in {\rm IV}: n_{1}<n<n_{2}\} \subseteq {\rm IV}$. Here, $n_{1}$ satisfies $\frac{\partial \psi}{\partial b}(n_{1},b_{1})=0$, indicating that the curve H has a fold with respect to $n$ at $(n,b)=(n_{1},b_{1})$, and $n_{2}$ satisfies $\frac{\partial \varphi}{\partial b}(n_{2},b_{2})=0$, indicating that the curve SNL has a fold with respect to $n$ at $(n,b)=(n_{2},b_{2})$. (Both fold points  $(n_{1},b_{1})\approx(3.32694,0.659059)$ and $(n_{2},b_{2})\approx(3.40459,0.648936)$ are marked as blue dots in \textbf{Figure \ref{n_b_all}}(a) and (c).) Hence, isolas are parameterized by $n$ and can effectively be traced as the point $(n,b)$ is moved vertically between the two branches of the curve SNL for each $n\in\ ]n_{1},n_{2}[$ fixed. In \textbf{Figure~\ref{n_b_all}}(d) we see a sample of these isolas for a range of values of $n$. The smallest shown isola corresponds to $n=3.39773$; hence, as the proportionality constant $n$ is decreased towards $n_1$, the possible amplitudes of the stable oscillations increase.

\begin{figure}[h!]
\begin{center}
\includegraphics[width=\textwidth]{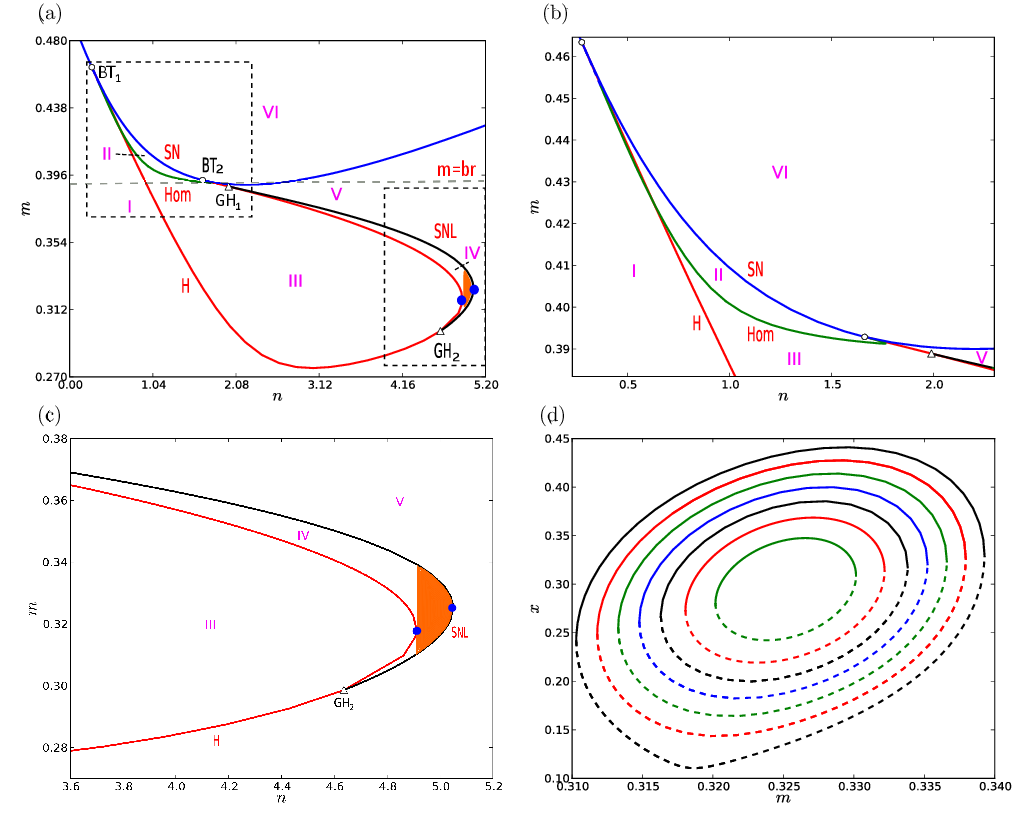}
\end{center}

\caption{\ Two-parameter bifurcation diagram of system (\ref{A1}) with respect to $n$ and $m$. Panels~(b) and (c) are enlargements of panel~(a). Panel (d) shows a sample of isolas of limit cycles present in the orange region in parameter space. Parameter values are $r=0.65, k=1.64445, q=0.25, a=0.8,  s=0.03, b=0.6, c=0.01$.}
\label{nm}
\end{figure}

 \textbf{Figure \ref{nm}} shows the bifurcation diagram with respect to variation of both $n$ and $m$. Here we find two Bogdanov-Takens (BT) points and two  generalized Hopf bifurcation (GH) points as organizing centers for the dynamics. In particular, region IV contains a set of parameter values which allows isolas of limit cycles.
The isola region in the $(n,m)$-plane corresponds to the orange sector  in  panels (a) and (b) and can be defined in a similar way to the previous cases. Namely, it corresponds to the region enclosed by the curves H and SNL for values of $n$ between their corresponding fold points (marked as blue dots in \textbf{Figure~\ref{nm}}).
More precisely, let $\psi(n,m)=0$ represent the Hopf bifurcation curve and suppose that $\varphi(n,m)=0$ denotes the SNL curve, where $\psi$ and $\varphi$ are smooth functions. Specifically, assuming that
${\rm IV}=\{(n,m):\phi(n,m)>0 \ \text{and} \ \psi(n,m)>0\},$
 the isola region is specified as $\mathcal{I}=\{(n,m) \in {\rm IV}: n_{1}<n<n_{2}\} \subset {\rm IV}$. Here, $n_{1}$ satisfies $\frac{\partial \psi}{\partial m}(n_{1},m_{1})=0$, indicating that the curve H has a fold with respect to $m$ at $(n,m)=(n_{1},m_{1})$, and $n_{2}$ satisfies $\frac{\partial \psi}{\partial m}(n_{2},m_{2})=0$, indicating that the curve SNL has a fold with respect to $m$ at $(n,m)=(n_{2},m_{2})$.
In panel (d) we see a sample of these isolas  in the $(m,x)$-plane for a range of values of $n$. As in the case of  \textbf{Figure~\ref{n_b_all}}, as  $n$ decreases towards $n_1$, and for each fixed $m$, the amplitudes of the stable limit cycles increase. Finally, notice that the isola region lies below the horizontal line $m=br$ and, hence, within the weak Allee scenario.

\section{Discussion}

In this paper, we thoroughly investigated the rich dynamics of a predator-prey model with additive Allee effect and a generalized Holling IV functional response.
We proved that the model is well posed in the sense that every solution with a realistic initial condition remains bounded in the first quadrant. This result was obtained by finding a positively invariant compact set. On the other hand, to study the dynamics near the origin we performed a suitable change of coordinates and time rescaling to obtain a qualitatively equivalent system which can be extended to the entire first quadrant. Here we made use of the blow-up technique to desingularize the equilibrium point at the origin. We found that condition $br-m<0$ favors the chance of extinction of both populations as the origin is a non hyperbolic attractor. This condition is equivalent to having a prey growth rate $r$ less than $\frac{m}{b}$, the ratio between the two Allee parameters; that is, the more pronounced the Allee effect (i.e., larger $m$ or smaller $b$), the greater the chances of extinction of both species. Conversely, if $r>\frac{m}{b}$, our findings indicate that neither positive solution
can converge to the origin.

In addition to the origin, our model can have up to two other equilibria in the absence of predator, which are (under genericity conditions) an unstable node and a saddle.
Moreover, there can be up to four positive equilibrium points (i.e., in the interior of the first quadrant) and we obtained analytical conditions to determine the local stability of any of these equilibria. Further, we performed a bifurcation analysis to reveal the existence of codimension two and three Bogdanov-Takens bifurcations and codimension-three generalized Hopf bifurcation. It is an interesting realization that all these analytical conditions for existence and stability of positive equilibria and associated bifurcations involve algebraic expressions that depend on all parameters.

 We employed numerical bifurcation analysis to further explore various bifurcation phenomena. As well as a number of bifurcations of equilibria, we found saddle-node bifurcation of limit cycles, and a codimension-2 cusp of limit cycles. Significantly, while previous works have reported mushroom and isola phenomena associated with equilibrium points~\cite{das2023origin,2021Incoherent} in different biological contexts, this is the first time (as far as we know) that mushroom and isolas of limit cycles have been found in population dynamics.
 Moreover, we can pin down the following (not necessarily mutually exclusive) biological scenarios based on bifurcation diagrams and phase portraits:

{\bf (1) Regime with Extinction of the Prey and Predators:} In this regime, there is an open set of initial conditions where both prey and predators may go extinct. This scenario corresponds to a strong Allee effect ($br<m$).

{\bf (2) Regime of Steady Coexistence:}  In this regime, the prey and predators coexist and tend to a stable equilibrium. These scenarios are present in region I in {\bf Figure}~\ref{nm} (a) (strong and weak Allee effect), in region I in {\bf Figure}~\ref{n_b_all} (a) (strong and weak Allee effect), and in region VI in {\bf Figure}~\ref{b_vs_m_all} (a) (weak Allee effect) -- see also {\bf Figure}~\ref{bm_phase}(VI).

{\bf (3) Regime of Oscillation:} In this regime, there is an open set of initial conditions where both the prey and predators tend to a stable oscillatory regime given by a stable limit cycle. We can observe this stable asymptotic pattern in regimes II, III, and IV in {\bf Figure}~\ref{nm}~(a) (strong and weak Allee effect), regimes IV and V in {\bf Figure} \ref{n_b_all}~(a) (strong and weak Allee effect), and regimes III, IV, and V in {\bf Figure} \ref{b_vs_m_all} (a) (weak Allee effect) -- see also {\bf Figure}~\ref{bm_phase}.
Based on the presence of isola dynamics of limit cycles, each of these closed bifurcation curves has an associated critical value of the Allee parameter $b$ in which the possible amplitudes of the asymptotic oscillations of both populations attain their maxima. In terms of the other Allee parameter, $m$, the weaker the Allee effect is (i.e., decreasing $m$ within the isola regime), the larger the amplitudes of the stable oscillations. Similarly, as $n$ tends to the minimum value that allows an isola, the possible amplitudes of the stable cycles increase. Since $n$ is a measure of the dependance of the predator carrying capacity on the prey abundance, the latter outcome may be the case if the predator is generalist and feeds on more sources than the actual prey considered in this model, leading to larger volumes of biomass that this ecological interaction may potentially sustain periodically.

{\bf (4) Regime of Oscillatory Multistability:} A large amplitude stable limit cycle surrounds an unstable limit cycle that encloses a smaller stable limit cycle. This pattern was found for specific values of  $b$ which favor this configuration within the mushroom bifurcation of limit cycles. This scenario is organized by a codimension-2 cusp bifurcation point of limit cycles. It should be noted that in our study, the coexisting limit cycles correspond to oscillations of great amplitude, which distinguishes our mechanism from that proposed by Aguirre et al.~\cite{aguirre2009three} -- In their work, two limit cycles bifurcate from a Hopf bifurcation point, while the third limit cycle emerges from a homoclinic bifurcation.

Based on all the above observations, we can draw the following conclusions about our model:

(a) Extinction of both populations may only occur if the Allee effect is strong.

(b) However, long term coexistence is possible even if the Allee effect is strong. (The actual basins of attraction of the origin and the positive steady states may be separated by stable manifolds of saddle objects, periodic orbits or homoclinic connections;~see for instance~\cite{contreras2018allee}.)

(c) A weak Allee effect $(0<m<br)$ can result in complex dynamics nonetheless, including the presence of isolas, mushrooms and cusps of limit cycles.

Future research directions may include analytically studying the boundary between the two-limit-cycle region and the three-limit-cycle region, as well as detecting possible transitions leading to mushroom bifurcations and the formation of isolas of limit cycles. This last avenue of investigation could well be implemented on conceptual return maps accounting for the local dynamics near the periodic orbits in suitable cross sections based  on~\cite{golubitsky1978theory}.
Another question that was beyond the scope of this installment is the determination of parametric curves corresponding to maxima of cycle amplitudes; a scheme similar to that in~\cite{contreras2020finding} may be implemented to follow this maximum values in two parameters by continuation. The identification of basins of attraction of equilibria and limit cycles is another open challenge; in combination with the presented bifurcation analysis, this would shed light on the design of control strategies to avoid extinction of both species through human intervention both at the level of initial conditions and parameter modulation.

\section*{Acknowledgments}
The authors are grateful to Professor Shigui Ruan for his helpful and insightful suggestions. This work was supported by the National NSF of China (No.~1167111 4), NSF of Zhejiang (No.~LY20A010002). PA thanks Proyecto Basal CMM-Universidad de Chile.

\bibliographystyle{siamplain}
\bibliography{references}

\begin{appendices}
\let\include\input
\section{Appendix A: Existence of positive equilibria for system (\ref{A1})}\label{AppA}

We study the positive roots of   Eq. (\ref{A5}), i.e.,
\begin{eqnarray*}
\bar{F}(x):= a_{4}x^4+a_{3}x^{3}+a_{2}x^{2}+a_{1}x+a_{0}=0,
\end{eqnarray*}
where
\begin{eqnarray*}
&&a_{4}=r, \ \ a_{3}=r (b+c-k), \ \ a_{2}=r (a+b c)+k (-r (b+c)+m+n q), \\
&&a_{1}=a r (b-k)+c k (m-b r)+b k n q, \ \ a_{0}=a k (m-b r).
\end{eqnarray*}
According to Descartes' rule of signs, the system (\ref{A1}) can have at most four positive roots. If there are indeed four positive roots, we denote them in ascending order as $x_{1}<x_{2}<x_{3}<x_{4}$, and the corresponding equilibria of the system are labeled as $E_{1}=(x_{1},y_{1}), E_{2}=(x_{2},y_{2}), E_{3}=(x_{3},y_{3}),$ and $E_{4}=(x_{4},y_{4})$. To simplify the discussion, we introduce the following notations and equations:
\begin{eqnarray}
\label{A2} &&\bar{F}_{1}(x):=\bar{F}'(x)=4a_{4}x^3+3a_{3}x^{2}+2a_{2}x+a_{1}=0,\\
\label{A6} &&\bar{F}_{2}(x):=\bar{F}''(x)=2(6a_{4}x^2+3a_{3}x+a_{2})=0.
\end{eqnarray}

Let us start by discussing the positivity of $\bar{F}_{2}(x)$. To locate the positive roots of (\ref{A6}), we  consider
\begin{eqnarray*}
6a_{4}x^2+3a_{3}x+a_{2}=0.
\end{eqnarray*}
Set
\begin{eqnarray*}
\Delta_{\bar{F}_{2}}=9a_{3}^{2}-24a_{2}a_{4}.
\end{eqnarray*}
Then equation (\ref{A6}) has no positive root if $\Delta_{\bar{F}_{2}}<0$, has one positive root of multiplicity 2 if $\Delta_{\bar{F}_{2}}=0$ and $a_{3}<0$, and has two positive roots (denoted by $\bar{x}_{1}$ and $\bar{x}_{2}$) if $\Delta_{\bar{F}_{2}}>0$, $a_{2}>0$ and $a_{3}<0$, where
\begin{eqnarray*}
\bar{x}_{1}=\frac{-3a_{3}-\sqrt{9a_{3}^{2}-24a_{2}a_{4}}}{12a_{4}},\\
\bar{x}_{2}=\frac{-3a_{3}+\sqrt{9a_{3}^{2}-24a_{2}a_{4}}}{12a_{4}}.
\end{eqnarray*}
On the basis of these facts, we discuss the existence of positive roots of (\ref{A5}) in three cases: $\Delta_{\bar{F}_{2}}<0$; $\Delta_{\bar{F}_{2}}=0$ and $a_{3}<0$; $\Delta_{\bar{F}_{2}}>0$, $a_{2}>0$ and $a_{3}<0$. Let $x_{i,i+1}$ be the coincidence point of $x_{i}$ and $x_{i+1}$ $(i=1,2,3)$, that is a multiple root of (\ref{A5}). We need to consider the following scenarios:\\

\noindent
$(C_1) \ \Delta_{\bar{F}_{2}}<0$.\\
In this case, (\ref{A6}) has no positive root and $\bar{F}_{2}(x)>0$ in $[0,+\infty)$, which means that $\bar{F_{1}}(x)$ is an increasing function in $[0,+\infty)$. Note that $\bar{F}_{1}(0)=a_{1}$ and $\bar{F}_{1}(x)\rightarrow +\infty$ as $x\rightarrow +\infty$. We have the following cases:\\

\noindent
$(C_{1a})$  If $\bar{F}_{1}(0)=a_{1}\geq 0$, $\bar{F}_{1}(x) \geq 0$ in $[0,+\infty)$, then $\bar{F}(x)$ is monotonically monotonically increasing in $[0,+\infty)$. Again note that $\bar{F}(0)=a_{0}$ and $\bar{F}(x)\rightarrow +\infty$ as $x\rightarrow +\infty$. We have the following subcases:\\
$(i)$    If $\bar{F}(0)=a_{0}<0$, (\ref{A5}) has a unique positive root;\\
$(ii)$   If $a_{0}\geq 0$, (\ref{A5}) has no positive root.\\

\noindent
$(C_{1b})$  When $a_{1}<0$, there is a unique positive root $\hat{x}_{1}$ in (\ref{A2}), which means that $\bar{F}(x)$ is a function that monotonically decreases first and then monotonically increases. It is worth noting that $\bar{F}(0)=a_{0}$ and $\bar{F}(x)\rightarrow +\infty$ as $x\rightarrow +\infty$. We get\\
$(i)$    If $\bar{F}(0)=a_{0}>0$, (\ref{A5}) has two positive roots if $\bar{F}(\hat{x}_{1})<0$, has one positive root of multiplicity 2 if $\bar{F}(\hat{x}_{1})=0$, has no positive root if $\bar{F}(\hat{x}_{1})>0$;\\
$(ii)$   If $a_{0} \leq 0$, (\ref{A5}) has a unique positive root.\\

\noindent
$(C_2)$  $\Delta_{\bar{F}_{2}}=0$ and $a_{3}<0$.\\
In this case, (\ref{A6}) has one positive root of multiplicity 2, $\bar{x}_{1}=\bar{x}_{2}=-\frac{a_{3}}{4a_{4}}$, and $\bar{F}_{2}(x) \geq 0$ in $[0,+\infty)$, which means that $\bar{F}_{1}(x)$ is an increasing function. Notice that $\bar{F}_{1}(0)=a_{1}$, $\bar{F}_{1}(\bar{x}_{1})=\frac{a_{3}^{3}-4a_{2}a_{3}a_{4}+8a_{1}a_{4}^{2}}{8a_{4}^{2}}$ and $\bar{F}_1(x)\rightarrow +\infty$ as $x\rightarrow +\infty$. We can conclude that:\\

\noindent
$(C_{2a})$  If $\bar{F}_{1}(0)=a_{1}\geq0$, $\bar{F}_{1}(x) \geq 0$ in $[0,+\infty)$, this is similar to $(C_{1a})$ and we have:\\
$(i)$    If $\bar{F}(0)=a_{0}<0$, (\ref{A5}) has a unique positive root;\\
$(ii)$   If $a_{0}\geq 0$, (\ref{A5}) has no positive root.\\

\noindent
$(C_{2b})$  When $a_{1}<0$, we have:\\
$(I)$    If $\bar{F}_{1}(\bar{x}_{2}) \neq0$, (\ref{A2}) has a positive root $\hat{x}_{2}$. This is similar to $(C_{1b})$ and we have\\
$(i)$    For $\bar{F}(0)=a_{0}>0$: (\ref{A5}) has two positive roots if $\bar{F}(\hat{x}_{2})<0$, has one positive root of multiplicity 2 if $\bar{F}(\hat{x}_{2})=0$, has no positive root if $\bar{F}(\hat{x}_{2})>0$;\\
$(ii)$   For $a_{0} \leq 0$: (\ref{A5}) has a unique positive root.\\
$(II)$   If $\bar{F}_{1}(\bar{x}_{2}) =0$, then $\bar{x}_{2}$ is a positive root of multiplicity 3 of (\ref{A2}), and we have\\
$(i)$    For $\bar{F}(0)=a_{0}>0$: (\ref{A5}) has two positive roots if $\bar{F}(\bar{x}_{2})<0$, has one positive root of multiplicity 4 if $\bar{F}(\bar{x}_{2})=0$, has no positive root if $\bar{F}(\bar{x}_{2})>0$;\\
$(ii)$   For $a_{0} \leq 0$: (\ref{A5}) has a unique positive root.\\

\noindent
$(C_3) \Delta_{\bar{F}_{2}}>0$ and $a_{2}>0, a_{3}<0$.\\
In this case, $\bar{x}_{1}$ and $\bar{x}_{2}$ are the maximum and minimum points of $\bar{F}_{1}$, respectively. We discuss the distribution of positive roots of (\ref{A5}) in two cases: $\bar{F}_{1}(0) \geq 0$ and $\bar{F}_{1}(0)<0$.\\

\noindent
$(C_{3a})$  When $\bar{F}_{1}(0)=a_{1} \geq 0$, $\bar{F}_{1}(\bar{x}_{1})>0$, based on the sign of $\bar{F}_{1}(\bar{x}_{2})$, we have:\\
$(I)$    If $\bar{F}_{1}(\bar{x}_{2})>0$, then $\bar{F}_{1}(x)\geq 0, x \in [0,+\infty)$, which means that $\bar{F}(x)$ is a monotonically increasing function in $[0,+\infty)$. Thus $\bar{F}(x)$ has no positive root if $a_{0} \geq 0$, has one positive root if $a_{0}<0$;\\
$(II)$   If $\bar{F}_{1}(\bar{x}_{2})=0$, then $\bar{x}_{2}$ is a positive root of multiplicity 2 of (\ref{A2}). Therefore, we can derive the distribution of positive roots of (\ref{A5}) as follows:\\
$(i)$    For $\bar{F}(\bar{x}_{2})<0$: a unique positive root;\\
$(ii)$   For $\bar{F}(\bar{x}_{2})=0$: one positive root of multiplicity 3;\\
$(iii)$  For $\bar{F}(\bar{x}_{2})>0$: a unique positive root if $\bar{F}(0)=a_{0}<0$, no positive root if $a_{0} \geq 0$.\\
$(III)$  If $\bar{F}_{1}(\bar{x}_{2})<0$, then (\ref{A2}) has two positive roots, which are denoted by $\tilde{x}_{1}, \tilde{x}_{2}$. Using $\bar{F}(0), \bar{F}(\tilde{x}_{1})$ and $\bar{F}(\tilde{x}_{2})$, the distribution of the positive roots of (\ref{A5}) can be summarized as follows:\\
$(i)$    For $\bar{F}(0)=a_{0} \geq 0, \bar{F}(\tilde{x}_{1})>0, \bar{F}(\tilde{x}_{2}) > 0$: no positive root;\\
$(ii)$   For $a_{0} \geq 0, \bar{F}(\tilde{x}_{1})>0, \bar{F}(\tilde{x}_{2}) = 0$: one positive root of multiplicity 2, which is denoted by $x_{3,4}=\tilde{x}_{2}$;\\
$(iii)$  For $a_{0} \geq 0, \bar{F}(\tilde{x}_{1})>0, \bar{F}(\tilde{x}_{2}) < 0$: two positive roots, which are denoted by $x_{3}<x_{4}$;\\
$(iv)$   For $a_{0} < 0,    \bar{F}(\tilde{x}_{1})>0, \bar{F}(\tilde{x}_{2}) > 0$: one positive root, which is denoted by $x_{1}$;\\
$(v)$    For $a_{0} < 0,    \bar{F}(\tilde{x}_{1})>0, \bar{F}(\tilde{x}_{2}) = 0$: two positive roots, one of them is a positive of multiplicity 2, which are denoted by $x_{1}<x_{3,4}=\tilde{x}_{2}$;\\
$(vi)$   For $a_{0} < 0,    \bar{F}(\tilde{x}_{1})>0, \bar{F}(\tilde{x}_{2}) < 0$: three positive roots, which are denoted by $x_{1}<x_{3}<x_{4}$;\\
$(vii)$  For $a_{0} < 0,    \bar{F}(\tilde{x}_{1})=0, \bar{F}(\tilde{x}_{2}) < 0$: two positive roots, one of them is a positive root of multiplicity 2, which are denoted by $x_{1,2}=\tilde{x}_{2}<x_{4}$;\\
$(viii)$ For $a_{0} < 0,    \bar{F}(\tilde{x}_{1})<0, \bar{F}(\tilde{x}_{2}) < 0$: a unique positive root $x_{4}$.\\

\noindent
$(C_{3b})$  When $a_{1}<0$,  based on the signs of $\bar{F}_{1}(\bar{x}_{1})$ and $\bar{F}_{1}(\bar{x}_{2})$, we can derive the following:\\

\noindent
$(I)$    If $\bar{F}_{1}(\bar{x}_{1})>0, \bar{F}_{1}(\bar{x}_{2})<0$, then (\ref{A2}) has three positive roots, which are denoted by $\hat{x}_{3}<\hat{x}_{4}<\hat{x}_{5}$. Thus we can obtain the following conclusion:\\
$(I_a)$   When $\bar{F}(0)=a_{0}>0$, we can get the distribution of positive roots of (\ref{A5}) as follows:\\
$(i)$    For $\bar{F}(\hat{x}_{3})<0, \bar{F}(\hat{x}_{4})>0, \bar{F}(\hat{x}_{5})<0$: four positive roots, which are denoted by $x_{1}<x_{2}<x_{3}<x_{4}$;\\
$(ii)$   For $\bar{F}(\hat{x}_{3})<0, \bar{F}(\hat{x}_{4})>0, \bar{F}(\hat{x}_{5})=0$: three positive roots (one of them being a positive root of multiplicity 2) which are denoted by $x_{1}<x_{2}<x_{3,4}=\hat{x}_{5}$;\\
$(iii)$  For $\bar{F}(\hat{x}_{3})<0, \bar{F}(\hat{x}_{4})>0, \bar{F}(\hat{x}_{5})>0$: two positive roots, which are denoted by $x_{1}<x_{2}$;\\
$(iv)$   For $\bar{F}(\hat{x}_{3})<0, \bar{F}(\hat{x}_{4})=0, \bar{F}(\hat{x}_{5})<0$: three positive roots (one of them being a positive root of multiplicity 2) which are denoted by $x_{1}<x_{2,3}=\hat{x}_{4}<x_{4}$;\\
$(v)$    For $\bar{F}(\hat{x}_{3})<0, \bar{F}(\hat{x}_{4})<0, \bar{F}(\hat{x}_{5})<0$: two positive roots, which are denoted by $x_{1}<x_{4}$;\\
$(vi)$   For $\bar{F}(\hat{x}_{3})=0, \bar{F}(\hat{x}_{4})>0, \bar{F}(\hat{x}_{5})<0$: three positive roots (one of them being a positive root of multiplicity 2) which are denoted by $x_{1,2}=\hat{x}_{3}<x_{3}<x_{4}$;\\
$(vii)$  For $\bar{F}(\hat{x}_{3})=0, \bar{F}(\hat{x}_{4})>0, \bar{F}(\hat{x}_{5})=0$: two positive roots of multiplicity 2, which are denoted by $x_{1,2}=\hat{x}_{3}<x_{3,4}=\hat{x}_{5}$;\\
$(viii)$ For $\bar{F}(\hat{x}_{3})=0, \bar{F}(\hat{x}_{4})>0, \bar{F}(\hat{x}_{5})>0$: one positive root of multiplicity 2, which is denoted by $x_{1,2}=\hat{x}_{3}$;\\
$(ix)$   For $\bar{F}(\hat{x}_{3})>0, \bar{F}(\hat{x}_{4})>0, \bar{F}(\hat{x}_{5})<0$: two positive roots, which are denoted by $x_{3}<x_{4}$;\\
$(x)$    For $\bar{F}(\hat{x}_{3})>0, \bar{F}(\hat{x}_{4})>0, \bar{F}(\hat{x}_{5})=0$: one positive root of multiplicity 2, which is denoted by $x_{3,4}=\hat{x}_{5}$;\\
$(xi)$   For $\bar{F}(\hat{x}_{3})>0, \bar{F}(\hat{x}_{4})>0, \bar{F}(\hat{x}_{5})>0$: no positive root.\\
$(I_b)$   When $a_{0} \leq 0$, we obtain the distribution of positive roots of (\ref{A5}) as follows:\\
$(i)$    For $\bar{F}(\hat{x}_{3})<0, \bar{F}(\hat{x}_{4})>0, \bar{F}(\hat{x}_{5})<0$: three positive roots, which are denoted by $x_{2}<x_{3}<x_{4}$;\\
$(ii)$   For $\bar{F}(\hat{x}_{3})<0, \bar{F}(\hat{x}_{4})>0, \bar{F}(\hat{x}_{5})=0$: two positive roots (one of them being a positive root of multiplicity 2) which are denoted by $x_{2}<x_{3,4}=\hat{x}_{5}$;\\
$(iii)$  For $\bar{F}(\hat{x}_{3})<0, \bar{F}(\hat{x}_{4})>0, \bar{F}(\hat{x}_{5})>0$: a unique positive root, which is denoted by $x_{2}$;\\
$(iv)$   For $\bar{F}(\hat{x}_{3})<0, \bar{F}(\hat{x}_{4})=0, \bar{F}(\hat{x}_{5})<0$: two positive roots (one of them being a positive root of multiplicity 2) which are denoted by $x_{2,3}=\hat{x}_{4}<x_{4}$;\\
$(v)$    For $\bar{F}(\hat{x}_{3})<0, \bar{F}(\hat{x}_{4})<0, \bar{F}(\hat{x}_{5})<0$: a unique positive root, which is denoted by $x_{4}$.\\

\noindent
$(II)$   If $\bar{F}_{1}(\bar{x}_{1})>0, \bar{F}_{1}(\bar{x}_{2})=0$, then $\bar{F}_{1}(x)$ has two positive roots $\hat{x}_{3}$ and $\bar{x}_{2}$, in which $\bar{x}_{2}$ is a positive root of multiplicity 2. Similarly, we have the following result:\\
$(II_a)$  When $\bar{F}(0)=a_{0}>0$, based on the sign of $\bar{F}(\bar{x}_{2})$, we can derive the distribution of positive roots of (\ref{A5}) as follows.\\
$(i)$    For $\bar{F}(\bar{x}_{2})>0$: two positive roots if $\bar{F}(\hat{x}_{3})<0$, which are denoted by $x_{1}<x_{2}$, one positive root of multiplicity 2 if $\bar{F}(\hat{x}_{3})=0$, and no positive root if $\bar{F}(\hat{x}_{3})>0$;\\
$(ii)$   For $\bar{F}(\bar{x}_{2})=0$: two positive roots (one of them being a positive root of multiplicity 3)   which are denoted by $x_{1}<x_{2,3,4}=\bar{x}_{2}$;\\
$(iii)$  For $\bar{F}(\bar{x}_{2})<0$: two positive roots, which are denoted by $x_{1}<x_{4}$;\\
$(II_b)$  When $a_{0}\leq0$, based on the sign of $\bar{F}(\bar{x}_{2})$, we can get the distribution of positive roots of (\ref{A5}) as follows:\\
$(i)$    For $\bar{F}(\bar{x}_{2})>0$: one positive root, which is denoted by $x_{2}$;\\
$(ii)$   For $\bar{F}(\bar{x}_{2})=0$: one positive root of multiplicity 3,   which is denoted by $x_{2,3,4}=\bar{x}_{2}$;\\
$(iii)$  For $\bar{F}(\bar{x}_{2})<0$: one positive root, which is denoted by $x_{4}$.\\

\noindent
$(III)$  If $\bar{F}_{1}(\bar{x}_{1})>0, \bar{F}_{1}(\bar{x}_{2})>0$, then $\bar{F}_{1}(x)$ has one positive root $\hat{x}_{3}$. Thus, we have\\
$(i)$    If $\bar{F}(0)=a_{0}>0$: (\ref{A5}) has two positive roots if $\bar{F}(\hat{x}_{3})<0$, has one positive root of multiplicity 2 if $\bar{F}(\hat{x}_{3})=0$, has no positive root if $\bar{F}(\hat{x}_{3})>0$;\\
$(ii)$   If $a_{0} \leq 0$: (\ref{A5}) has a unique positive root.\\

\noindent
$(IV)$   If $\bar{F}_{1}(\bar{x}_{1})=0, \bar{F}_{1}(\bar{x}_{2})<0$, $\bar{F}_{1}(x)$ has two positive roots $\bar{x}_{1}$ and $\hat{x}_{5}$, in which $\bar{x}_{1}$ has multiplicity 2. Then we have:\\
$(IV_a)$  When $\bar{F}(0)=a_{0}>0$, based on the signs of $\bar{F}(\bar{x}_{1})$ and $\bar{F}(\hat{x}_{5})$, we get the distribution of positive roots of (\ref{A5}) as follows.\\
$(i)$    For $\bar{F}(\bar{x}_{1})>0, \bar{F}(\hat{x}_{5})<0$: two positive roots, which are denoted by $x_{3}<x_{4}$;\\
$(ii)$   For $\bar{F}(\bar{x}_{1})>0, \bar{F}(\hat{x}_{5})=0$: one positive root of multiplicity 2, which is denoted by $x_{3,4}=\hat{x}_{5}$;\\
$(iii)$  For $\bar{F}(\bar{x}_{1})>0, \bar{F}(\hat{x}_{5})>0$: no positive root.\\
$(iv)$   For $\bar{F}(\bar{x}_{1})=0$: two positive roots (one of them being a positive root of multiplicity 3) which are denoted by $x_{1,2,3}=\bar{x}_{1}<x_{4}$;\\
$(v)$    For $\bar{F}(\bar{x}_{1})<0$: two positive roots, which are denoted by $x_{1}<x_{4}$.\\
$(IV_b)$  When $a_{0}\leq0$: one positive root, which is denoted by $x_{4}$.\\

\noindent
$(V)$    When $\bar{F}_{1}(\bar{x}_{1})<0, \bar{F}_{1}(\bar{x}_{2})<0$, $\bar{F}_{1}(x)$ has one positive root $\hat{x}_{5}$.\\
$(i)$   If $\bar{F}(0)=a_{0}>0$: two positive roots if $\bar{F}(\hat{x}_{5})<0$, which are denoted by $x_{3}<x_{4}$, one positive root of multiplicity 2 if $\bar{F}(\hat{x}_{5})=0$, and no positive root if $\bar{F}(\hat{x}_{5})>0$;\\
$(ii)$   If $a_{0} \leq0$: a unique positive root $x_{4}$.

\vspace{-50mm}
 
\section{Appendix B: Coefficients in the proof of Theorem \ref{tha1} and Theorem~\ref{tha2}}\label{AppB}

\begin{eqnarray*}
&&\hat{a}_{10}=\frac{r x^{*} (a+x^{*} (c+x^{*})) (-b+k-2 x^{*})}{k (a (b+2 x^{*})+x^{*2} (c-b))},\
\hat{a}_{01}=\frac{r x^{*} (a+x^{*} (c+x^{*})) (b-k+2 x^{*})}{k n (a (b+2 x^{*})+x^{*2} (c-b))},\\
&&\hat{a}_{20}=\frac{r}{k (b+x^{*}) (a+x^{*} (c+x^{*})) (a (b+2 x^{*})+x^{*2} (c-b))} (x^{*} (b+x^{*}) (x^{*3}-a (c+3 x^{*})) (b-\\
&&k+2 x^{*})-(a+x^{*} (c+x^{*})) (a (b^2-b k+3 b x^{*}+2 x^{*2})+x^{*2} (-b^2+b (k-3 x^{*})+c x^{*}))),\\
&&\hat{a}_{11}=\frac{r (a-x^{*2}) (b-k+2 x^{*})}{k n (a (b+2 x^{*})+x^{*2} (c-b))},\\
&&\hat{a}_{30}=-\frac{r}{k (b+x^{*})^2 (a+x^{*} (c+x^{*}))^2 (a (b+2 x^{*})+x^{*2} (c-b))} (x^{*} (b+x^{*})^2 (a^2-a (c^2+4 c x^{*}\\
&&+6 x^{*2})+x^{*4}) (b-k+2 x^{*})+b (a+x^{*} (c+x^{*}))^2 (a k+x^{*2} (c-k+2 x^{*}))),\\
&&\hat{a}_{21}=\frac{r (x^{*3}-a (c+3 x^{*})) (b-k+2 x^{*})}{k n (a+x^{*} (c+x^{*})) (a (b+2 x^{*})+x^{*2} (c-b))};\
\hat{b}_{10}=n s,\
\hat{b}_{01}=-s,\
\hat{b}_{20}=-\frac{n s}{x^{*}},\\
&&\hat{b}_{11}=\frac{2 s}{x^{*}},\
\hat{b}_{02}=-\frac{s}{n x^{*}},\
\hat{b}_{30}=\frac{n s}{x^{*2}},\
\hat{b}_{21}=-\frac{2 s}{x^{*2}},\
\hat{b}_{12}=\frac{s}{n x^{*2}};\\
&&\hat{c}_{20}=\frac{\hat{a}_{10} (\hat{a}_{01} \hat{b}_{11}-\hat{a}_{11} \hat{b}_{01})-\hat{a}_{10}^2 \hat{b}_{02}+\hat{a}_{01} (\hat{a}_{20} \hat{b}_{01}-\hat{a}_{01} \hat{b}_{20})}{\hat{a}_{01} (\hat{a}_{10}+\hat{b}_{01})},\\
&&\hat{c}_{11}=\frac{-2 \hat{a}_{01}^2 \hat{b}_{20}+\hat{a}_{01} \hat{b}_{01} (2 \hat{a}_{20}-\hat{b}_{11})+\hat{a}_{11} \hat{b}_{01}^2+\hat{a}_{10} (\hat{a}_{01} \hat{b}_{11}-\hat{b}_{01} (\hat{a}_{11}-2 \hat{b}_{02}))}{\hat{a}_{01} (\hat{a}_{10}+\hat{b}_{01})},\\
&&\hat{c}_{02}=\frac{-\hat{a}_{01}^2 \hat{b}_{20}+\hat{a}_{01} \hat{b}_{01} (\hat{a}_{20}-\hat{b}_{11})+\hat{b}_{01}^2 (\hat{a}_{11}-\hat{b}_{02})}{\hat{a}_{01} (\hat{a}_{10}+\hat{b}_{01})},\\
&&\hat{c}_{30}=\frac{\hat{a}_{10} (\hat{a}_{01} \hat{b}_{21}-\hat{a}_{21} \hat{b}_{01})-\hat{a}_{10}^2 \hat{b}_{12}+\hat{a}_{01} (\hat{a}_{30} \hat{b}_{01}-\hat{a}_{01} \hat{b}_{30})}{\hat{a}_{01} (\hat{a}_{10}+\hat{b}_{01})},\\
&&\hat{c}_{21}=\frac{-3 \hat{a}_{01}^2 \hat{b}_{30}+\hat{a}_{01} \hat{b}_{01} (3 \hat{a}_{30}-\hat{b}_{21})+\hat{a}_{21} \hat{b}_{01}^2-\hat{a}_{10}^2 \hat{b}_{12}+2 \hat{a}_{10} (\hat{b}_{01} (\hat{b}_{12}-\hat{a}_{21})+\hat{a}_{01} \hat{b}_{21})}{\hat{a}_{01} (\hat{a}_{10}+\hat{b}_{01})},\\
&&c_{12}=\frac{-3 \hat{a}_{01}^2 \hat{b}_{30}+\hat{a}_{01} \hat{b}_{01} (3 \hat{a}_{30}-2 \hat{b}_{21})+\hat{b}_{01}^2 (2 \hat{a}_{21}-\hat{b}_{12})+\hat{a}_{10} (\hat{a}_{01} \hat{b}_{21}-\hat{b}_{01} (\hat{a}_{21}-2 \hat{b}_{12}))}{\hat{a}_{01} (\hat{a}_{10}+\hat{b}_{01})},\\
&&\hat{c}_{03}=\frac{-\hat{a}_{01}^2 \hat{b}_{30}+\hat{a}_{01} \hat{b}_{01} (\hat{a}_{30}-\hat{b}_{21})+\hat{b}_{01}^2 (\hat{a}_{21}-\hat{b}_{12})}{\hat{a}_{01} (\hat{a}_{10}+\hat{b}_{01})};\\
&&\hat{d}_{01}=\frac{\hat{a}_{10} (\hat{a}_{10}+\hat{b}_{01})+\hat{a}_{01} \hat{b}_{10}+\hat{b}_{01}^2}{\hat{a}_{10}+\hat{b}_{01}}=p(x^{*}),\\
&&\hat{d}_{20}=\frac{\hat{a}_{01}^2 \hat{b}_{20}+\hat{a}_{10} \hat{a}_{01} (\hat{a}_{20}-\hat{b}_{11})+\hat{a}_{10}^2 (\hat{b}_{02}-\hat{a}_{11})}{\hat{a}_{01} (\hat{a}_{10}+\hat{b}_{01})},\\
&&\hat{d}_{11}=\frac{\hat{a}_{10} (\hat{b}_{01} (\hat{a}_{11}-2 \hat{b}_{02})+\hat{a}_{01} (2 \hat{a}_{20}-\hat{b}_{11}))+\hat{a}_{01} (2 \hat{a}_{01} \hat{b}_{20}+\hat{b}_{01} \hat{b}_{11})-\hat{a}_{10}^2 \hat{a}_{11}}{\hat{a}_{01} (\hat{a}_{10}+\hat{b}_{01})},\\
&&\hat{d}_{02}=\frac{\hat{a}_{01}^2 \hat{b}_{20}+\hat{a}_{01} (\hat{a}_{10} \hat{a}_{20}+\hat{b}_{01} \hat{b}_{11})+\hat{b}_{01} (\hat{a}_{10} \hat{a}_{11}+\hat{b}_{01} \hat{b}_{02})}{\hat{a}_{01} (\hat{a}_{10}+\hat{b}_{01})},\\
&&\hat{d}_{30}=\frac{\hat{a}_{01}^2 \hat{b}_{30}+\hat{a}_{10} \hat{a}_{01} (\hat{a}_{30}-\hat{b}_{21})+\hat{a}_{10}^2 (\hat{b}_{12}-\hat{a}_{21})}{\hat{a}_{01} (\hat{a}_{10}+\hat{b}_{01})},
\end{eqnarray*}

\begin{eqnarray*}
&&\hat{d}_{21}=\frac{\hat{a}_{10}^2 (\hat{b}_{12}-2 \hat{a}_{21})+\hat{a}_{10} (\hat{b}_{01} (\hat{a}_{21}-2 \hat{b}_{12})+\hat{a}_{01} (3 \hat{a}_{30}-2 \hat{b}_{21}))+\hat{a}_{01} (3 \hat{a}_{01} \hat{b}_{30}+\hat{b}_{01} \hat{b}_{21})}{\hat{a}_{01} (\hat{a}_{10}+\hat{b}_{01})},\\
&&\hat{d}_{12}=\frac{3 \hat{a}_{01}^2 \hat{b}_{30}+2 \hat{a}_{01} \hat{b}_{01} \hat{b}_{21}+\hat{a}_{10} (2 \hat{b}_{01} (\hat{a}_{21}-\hat{b}_{12})+\hat{a}_{01} (3 \hat{a}_{30}-\hat{b}_{21}))-\hat{a}_{10}^2 \hat{a}_{21}+\hat{b}_{01}^2 \hat{b}_{12}}{\hat{a}_{01} (\hat{a}_{10}+\hat{b}_{01})},\\
&&\hat{d}_{03}=\frac{\hat{a}_{01}^2 \hat{b}_{30}+\hat{a}_{01} (\hat{a}_{10} \hat{a}_{30}+\hat{b}_{01} \hat{b}_{21})+\hat{b}_{01} (\hat{a}_{10} \hat{a}_{21}+\hat{b}_{01} \hat{b}_{12})}{\hat{a}_{01} (\hat{a}_{10}+\hat{b}_{01})};\\
&&a^{*}_{10}=s,\
a^{*}_{01}=-\frac{q x^{*}}{a+x^{*}(c +x^{*})},\
a^{*}_{20}=\frac{s (a-x^{*2})}{2 x^{*} (a+x^{*} (c+x^{*}))},\
a^{*}_{11}=\frac{q (x^{*2}-a)}{(a+x^{*} (c+x^{*}))^2},\\
&&a^{*}_{30}=\frac{s}{2 (a+x^{*} (c+x^{*}))^3} (2 (a^2-a (c^2+4 c x^{*}+6 x^{*2})+x^{*4})-\frac{b (a+x^{*} (c+x^{*}))}{x^{*2} (b+x^{*})} (a^2+3 a x^{*}\\
&&\times (c+2 x^{*})-x^{*3} (c+3 x^{*}))),\
a^{*}_{21}=\frac{q (a (c+3 x^{*})-x^{*3})}{(a+x^{*} (c+x^{*}))^3},\\
&&a^{*}_{40}=\frac{s}{2 (a+x^{*} (c+x^{*}))^4} (-2 (a^2 (2 c+5 x^{*})-a (c^3+5 c^2 x^{*}+10 c x^{*2}+10 x^{*3})+x^{*5})\\
&&+\frac{b}{x^{*2} (b+x^{*})^2}(a+x^{*} (c+x^{*}))^2(a^2+3 a x^{*} (c+2 x^{*})-x^{*3} (c+3 x^{*}))),\\
&&a^{*}_{31}=\frac{q (a^2-a (c^2+4 c x^{*}+6 x^{*2})+x^{*4})}{(a+x^{*} (c+x^{*}))^4};\
b^{*}_{10}=\frac{s^2 (a+x^{*} (c+x^{*}))}{q x^{*}},\
b^{*}_{01}=-s,\\
&&b^{*}_{20}=-\frac{s^2 (a+x^{*} (c+x^{*}))}{q x^{*2}},\
b^{*}_{11}=\frac{2 s}{x^{*}},\
b^{*}_{02}=-\frac{q}{a+x^{*} (c+x^{*})},\
b^{*}_{30}=\frac{s^2 (a+x^{*} (c+x^{*}))}{q x^{*3}},\\
&&b^{*}_{21}=-\frac{2 s}{x^{*2}},\
b^{*}_{12}=\frac{q}{x^{*} (a+x^{*} (c+x^{*}))},\
b^{*}_{40}=-\frac{s^2 (a+x^{*} (c+x^{*}))}{q x^{*4}},\
b^{*}_{31}=\frac{2 s}{x^{*3}},\\
&&b^{*}_{22}=-\frac{q}{x^{*2} (a+x^{*} (c+x^{*}))};\
c^{*}_{20}=\frac{a^{*2}_{10} b^{*}_{02}}{a^{*}_{01}}-a^{*}_{10} b^{*}_{11}-a^{*}_{20} b^{*}_{01}+a^{*}_{11} b^{*}_{10}+a^{*}_{01} b^{*}_{20},\\
&&c^{*}_{02}=\frac{a^{*}_{11}+b^{*}_{02}}{a^{*}_{01}},\
c^{*}_{30}=-a^{*}_{30} b^{*}_{01}+a^{*}_{21} b^{*}_{10}-a^{*}_{20} b^{*}_{11}+a^{*}_{11} b^{*}_{20}+a^{*}_{01} b^{*}_{30}+\frac{a^{*2}_{10} (a^{*}_{01} b^{*}_{12}-a^{*}_{11} b^{*}_{02})}{a^{*2}_{01}}\\
&&+a^{*}_{10} (\frac{2 a^{*}_{20} b^{*}_{02}}{a^{*}_{01}}-b^{*}_{21}),\
c^{*}_{21}=\frac{1}{a^{*2}_{01}}(a^{*}_{10} (2 a^{*}_{11} b^{*}_{02}-2 a^{*}_{01} (a^{*}_{21}+b^{*}_{12})+a^{*2}_{11})+a^{*}_{01} (a^{*}_{01} (3 a^{*}_{30}+b^{*}_{21})\\
&&-a^{*}_{20} (a^{*}_{11}+2 b^{*}_{02}))),\
c^{*}_{12}=\frac{1}{a^{*2}_{01}}(a^{*}_{01} (2 a^{*}_{21}+b^{*}_{12})-a^{*}_{11} (a^{*}_{11}+b^{*}_{02})),\\
&&c^{*}_{40}=-a^{*}_{40} b^{*}_{01}+a^{*}_{31} b^{*}_{10}-a^{*}_{30} b^{*}_{11}+a^{*}_{21} b^{*}_{20}-a^{*}_{20} b^{*}_{21}+a^{*}_{11} b^{*}_{30}-a^{*}_{10} b^{*}_{31}+a^{*}_{01} b^{*}_{40}+\frac{a^{*2}_{10} a^{*2}_{11} b^{*}_{02}}{a^{*3}_{01}}\\
&&-\frac{a^{*}_{10} (a^{*}_{10} a^{*}_{21} b^{*}_{02}+a^{*}_{11} (2 a^{*}_{20} b^{*}_{02}+a^{*}_{10} b^{*}_{12}))}{a^{*2}_{01}}+\frac{a^{*2}_{10} b^{*}_{22}+2 a^{*}_{30} a^{*}_{10} b^{*}_{02}+2 a^{*}_{20} a^{*}_{10} b^{*}_{12}+a^{*2}_{20} b^{*}_{02}}{a^{*}_{01}},\\
&&c^{*}_{31}=\frac{1}{a^{*3}_{01}} (-a^{*}_{10} (2 a^{*2}_{11} b^{*}_{02}-a^{*}_{01} a^{*}_{11} (3 a^{*}_{21}+2 b^{*}_{12})+a^{*}_{01} (a^{*}_{01} (3 a^{*}_{31}+2 b^{*}_{22})-2 a^{*}_{21} b^{*}_{02})+a^{*3}_{11})\\
&&+a^{*}_{01} (a^{*}_{11} (2 a^{*}_{20} b^{*}_{02}-a^{*}_{01} a^{*}_{30})+a^{*}_{01} (-2 a^{*}_{30} b^{*}_{02}-2 a^{*}_{20} (a^{*}_{21}+b^{*}_{12})+a^{*}_{01} (4 a^{*}_{40}+b^{*}_{31}))+a^{*}_{20} a^{*2}_{11})),\\
&&c^{*}_{22}=\frac{a^{*2}_{11} b^{*}_{02}-a^{*}_{01} a^{*}_{11} (3 a^{*}_{21}+b^{*}_{12})+a^{*}_{01} (a^{*}_{01} (3 a^{*}_{31}+b^{*}_{22})-a^{*}_{21} b^{*}_{02})+a^{*3}_{11}}{a^{*3}_{01}};
\end{eqnarray*}

\begin{eqnarray*}
&&\bar{a}_{00}=-\frac{\delta _1 x^{*}}{b+x^{*}},\
\bar{a}_{10}=s-\frac{b \delta _1}{(b+x^{*})^2},\
\bar{a}_{01}=-\frac{q x^{*}}{a+x^{*} (c+x^{*})},\
\bar{a}_{20}=\frac{s (a (c+3 x^{*})-x^{*3})}{(a+x^{*} (c+x^{*}))^2}+\\
&&+\frac{s (a (b+2 x^{*})+x^{*2} (c-b))}{x^{*} (a+x^{*} (c+x^{*})) (b-k+2 x^{*})}+\frac{b}{(b+x^{*})^3} (\delta _1 -\frac{s (b+x^{*})^2 (a k+x^{*2} (c-k+2 x^{*}))}{x^{*} (a+x^{*} (c+x^{*})) (b-k+2 x^{*})}),\\
&&\bar{a}_{11}=\frac{q (x^{*2}-a)}{(a+x^{*} (c+x^{*}))^2};\
\bar{b}_{00}=\frac{\delta _2 s^2 x^{*} (a+x^{*} (c+x^{*}))}{s (a+x^{*} (c+x^{*}))+\delta _2 q x^{*}},\\
&&\bar{b}_{10}=\frac{s^3 (a+x^{*} (c+x^{*}))^2}{q x^{*}  (s (a+x^{*} (c+x^{*}))+\delta _2 q x^{*})},\
\bar{b}_{01}=\frac{s (\delta _2 q x^{*}-s (a+x^{*} (c+x^{*})))}{s (a+x^{*} (c+x^{*}))+\delta _2 q x^{*}},\\
&&\bar{b}_{20}=-\frac{s^3 (a+x^{*} (c+x^{*}))^2}{q x^{*2} (s (a+x^{*} (c+x^{*}))+\delta _2 q x^{*})},\
\bar{b}_{11}=\frac{2 s^2 (a+x^{*} (c+x^{*}))}{x^{*} (s (a+x^{*} (c+x^{*}))+\delta _2 q x^{*})},\\
&&\bar{b}_{02}=-\frac{q s}{s (a+x^{*} (c+x^{*}))+\delta _2 q x^{*}};\
\bar{c}_{00}=\frac{\bar{a}_{00}^2 \bar{b}_{02}}{\bar{a}_{01}}-\bar{a}_{00} \bar{b}_{01}+\bar{a}_{01} \bar{b}_{00},\\
&&\bar{c}_{10}=\bar{a}_{10} (\frac{2 \bar{a}_{00} \bar{b}_{02}}{\bar{a}_{01}}-\bar{b}_{01})+\bar{a}_{11} (\bar{b}_{00}-\frac{\bar{a}_{00}^2 \bar{b}_{02}}{\bar{a}_{01}^2})+\bar{a}_{01} \bar{b}_{10}-\bar{a}_{00} \bar{b}_{11},\\
&&\bar{c}_{01}=-\frac{\bar{a}_{00} (a_{11}+2 \bar{b}_2)}{\bar{a}_{01}}+\bar{a}_{10}+\bar{b}_{01},\
\bar{c}_{20}=-\bar{a}_{20} \bar{b}_{01}+\bar{a}_{11} \bar{b}_{10}-\bar{a}_{10} \bar{b}_{11}+\bar{a}_{01} \bar{b}_{20}\\
&&+\frac{\bar{a}_{00}^2 \bar{a}_{11}^2 \bar{b}_{02}}{a_{01}^3}-\frac{2 \bar{a}_{00} \bar{a}_{10} \bar{a}_{11} \bar{b}_{02}}{\bar{a}_{01}^2}+\frac{(\bar{a}_{10}^2+2 \bar{a}_{00} \bar{a}_{20}) \bar{b}_{02}}{\bar{a}_{01}},\\
&&\bar{c}_{11}=-\frac{(\bar{a}_{01} \bar{a}_{10}-\bar{a}_{00} \bar{a}_{11}) (\bar{a}_{11}+2 \bar{b}_{02})}{\bar{a}_{01}^2}+2 \bar{a}_{20}+\bar{b}_{11},\
\bar{c}_{02}=\frac{\bar{a}_{11}+\bar{b}_{02}}{\bar{a}_{01}};\
\bar{d}_{00}=\bar{c}_{00},\\
&&\bar{d}_{10}=\bar{c}_{10}-2 \bar{c}_{00} \bar{c}_{02},\
\bar{d}_{01}=\bar{c}_{01},\
\bar{d}_{20}=\bar{c}_{00} \bar{c}_{02}^2-2 \bar{c}_{10} \bar{c}_{02}+\bar{c}_{20},\
\bar{d}_{11}=\bar{c}_{11}-\bar{c}_{01} \bar{c}_{02};\\
&&\bar{e}_{00}=\bar{d}_{00}-\frac{\bar{d}_{10}^2}{4 \bar{d}_{20}},\
\bar{e}_{01}=\bar{d}_{01}-\frac{\bar{d}_{10} \bar{d}_{11}}{2 \bar{d}_{20}},\
\bar{e}_{20}\bar{d}_{20},\
\bar{e}_{11}=\bar{d}_{11}.
\end{eqnarray*}

\section{Appendix C: The proof of Theorem \ref{tha33}}\label{AppC}

\begin{proof}
Choosing $n, b$ and $m$ as bifurcation parameters, system (\ref{A1}) can be written as
\begin{equation}\label{b3}
\begin{array}{rrl}
&&\dfrac{dx}{dt}=\left(r \left(1-\dfrac{x}{k}\right)-\dfrac{m+\kappa_{1}}{x+b+\kappa_{2}}\right)x-\dfrac{q x y}{x^2+c x+a}, \vspace{2mm}\\
&&\dfrac{dy}{dt}=sy\left(1-\dfrac{y}{(n+\kappa_{3})x}\right),
\end{array}
\end{equation}
where $\kappa=(\kappa_{1},\kappa_{2},\kappa_{3})\sim (0,0,0).$

Conditions $F(x^{*})=F'(x^{*})=p(x^{*})=p'(x^{*})=0$ imply that parameters $r,m,n,k$ are defined as in~\eqref{1000}. Next through the transformation $X=x-x^{*},Y=y-n x^{*}$ and the Taylor expansion, we get an equivalent system of system (\ref{b3}) as follows (for convenience, in every subsequent transformation, we rename $X,Y$ and $\tau$ as $x,y$ and $t$, respectively)
\begin{equation}\label{c8}
\begin{array}{rrl}
&&\dfrac{dx}{dt}=\bar{a}^{*}_{00}+\bar{a}^{*}_{10}x+\bar{a}^{*}_{01}y+\bar{a}^{*}_{20}x^{2}+\bar{a}^{*}_{11}xy+\bar{a}^{*}_{30}x^{3}
+\bar{a}^{*}_{21}x^{2}y\\
&&  \ \ \ \ \ \ \ \ +\bar{a}^{*}_{40}x^{4}+\bar{a}^{*}_{31}x^{3}y+o(|x,y|^{4}),\\
 &&\dfrac{dy}{dt}=\bar{b}^{*}_{00}+\bar{b}^{*}_{10}x+\bar{b}^{*}_{01}y+\bar{b}^{*}_{20}x^{2}+\bar{b}^{*}_{11}xy+\bar{b}^{*}_{02}y^{2}+\bar{b}^{*}_{30}x^{3}
\\
&& \ \ \ \ \ \ \ \ +\bar{b}^{*}_{21}x^{2}y+\bar{b}^{*}_{12}xy^{2}+\bar{b}^{*}_{40}x^{4}+\bar{b}^{*}_{31}x^{3}y+\bar{b}^{*}_{22}x^{2}y^{2}+o(|x,y|^{4}).
\end{array}
\end{equation}
Here $\bar{a}^{*}_{ij}, \bar{b}^{*}_{ij}$ and all parameters in subsequent transformations are given below and note that $\bar{a}_{00}^{*}=\bar{b}^{*}_{00}=0$ when $\kappa=0$.

System (\ref{c8}) is transformed by the change of variables
\begin{eqnarray*}
X=x, \ Y=\frac{dx}{dt},
\end{eqnarray*}
into
\begin{equation}\label{a5a5}
\begin{array}{rrl}
&&\dfrac{dx}{dt}=y, \vspace{2mm}\\
&&\dfrac{dy}{dt}=\bar{c}^{*}_{00}+\bar{c}^{*}_{10}x+\bar{c}^{*}_{01}y+\bar{c}^{*}_{20}x^{2}+\bar{c}^{*}_{11}xy+\bar{c}^{*}_{02}y^{2}
+\bar{c}^{*}_{30}x^{3}+\bar{c}^{*}_{21}x^{2}y+\bar{c}^{*}_{12}xy^{2}\\
&& \ \ \ \ \ \ \ \ +\bar{c}^{*}_{40}x^{4}+\bar{c}^{*}_{31}x^{3}y+\bar{c}^{*}_{22}x^{2}y^{2}+O(|x,y|^{4}).
\end{array}
\end{equation}
Note that $\overline{c}_{00}=\overline{c}_{10}=\overline{c}_{01}=\overline{c}_{11}=0$ when $\kappa=0$.

Using the method in Li et al. \cite{li2015codimension}, we transform (\ref{a5a5}) into system (\ref{b1}) to state the existence of the Bogdanov-Takens bifurcation of codimension 3 in the following steps:

($I$) Eliminating the $y^{2}$-term from system (\ref{a5a5}) when $\kappa=0$ by introducing the transformation $x=X+\frac{\bar{c}_{02}}{2}X^{2}, y=Y+\bar{c}_{02}XY$,   we have
\begin{equation}\label{a8}
\begin{array}{rrl}
&&\dfrac{dx}{dt}=y, \vspace{2mm}\\
&&\dfrac{dy}{dt}=\bar{d}^{*}_{00}+\bar{d}^{*}_{10}x+\bar{d}^{*}_{01}y+\bar{d}^{*}_{20}x^{2}+\bar{d}^{*}_{11}xy+\bar{d}^{*}_{30}x^{3}+\bar{d}^{*}_{21}x^{2}y
+\bar{d}^{*}_{12}xy^{2}\\
&& \ \ \ \ \ \ \ \ +\bar{d}^{*}_{40}x^{4}+\bar{d}^{*}_{31}x^{3}y+\bar{d}^{*}_{22}x^{2}y^{2}
+O(|x,y|^{4}).
\end{array}
\end{equation}
Note that $\bar{d}_{00}=\bar{d}_{10}=\bar{d}_{01}=\bar{d}_{11}=0$ when $\kappa=0$.\\

($II$) Eliminating the $xy^{2}$-term in system (\ref{a8}) when $\kappa=0$ and making the transformation $x=X+\frac{\bar{d}^{*}_{12}}{6}X^{3}, y=Y+\frac{\bar{d}^{*}_{12}}{2}X^{2}Y$,   we have
\begin{equation}\label{a9a9}
\begin{array}{rrl}
&&\dfrac{dx}{dt}=y, \vspace{2mm}\\
&&\dfrac{dy}{dt}=\bar{e}^{*}_{00}+\bar{e}^{*}_{10}x+\bar{e}^{*}_{01}y
+\bar{e}^{*}_{20}x^{2}+\bar{e}^{*}_{11}xy+\bar{e}^{*}_{30}x^{3}+\bar{e}^{*}_{21}x^{2}y
+\bar{e}^{*}_{40}x^{4}\\
&& \ \ \ \ \ \ \ \ +\bar{e}^{*}_{31}x^{3}y+\bar{e}^{*}_{22}x^{2}y^{2}+O(|x,y|^{4}),
\end{array}
\end{equation}
where $\bar{e}^{*}_{00}=\bar{e}^{*}_{10}=\bar{e}^{*}_{01}=\bar{e}^{*}_{11}=0$ when $\kappa=0$.\\

($III$) Eliminate the $x^{2}y^{2}$-term in system (\ref{a9a9}) when $\kappa=0$. Transformation $x=X+\frac{\bar{e}_{22}}{12}X^{4}, y=Y+\frac{\bar{e}_{22}}{3}X^{3}Y$ brings the above system to
\begin{equation}\label{a100}
\begin{array}{rrl}
&&\dfrac{dx}{dt}=y, \vspace{2mm}\\
&& \dfrac{dy}{dt}=\bar{f}^{*}_{00}+\bar{f}^{*}_{10}x+\bar{f}^{*}_{01}y+\bar{f}^{*}_{20}x^{2}
+\bar{f}^{*}_{11}xy+\bar{f}^{*}_{30}x^{3}+\bar{f}^{*}_{21}x^{2}y+\bar{f}^{*}_{40}x^{4}\\
&& \ \ \ \ \ \ \ \ +\bar{f}^{*}_{31}x^{3}y+O(|x,y|^{4}),
\end{array}
\end{equation}
where $\bar{f}^{*}_{00}=\bar{f}^{*}_{10}=\bar{f}^{*}_{01}=\bar{f}^{*}_{11}=0$ when $\kappa=0$.\\

($IV$) Eliminate the $x^{3}$ and $x^{4}$-term in system (\ref{a100}) when $\kappa=0$. We   easily know that
$\bar{f}^{*}_{20}=\frac{s^2 (x^{*2}-a)}{2 x^{*} (a+x^{*} (c+x^{*}))}+O(\kappa)$, $\bar{f}^{*}_{20} \neq 0$ for small $\kappa$. It follows from
\begin{eqnarray*}
x=X-\frac{\bar{f}^{*}_{30}}{4\bar{f}^{*}_{20}}X^{2}+\frac{15\bar{f}^{*2}_{30}
-16\bar{f}^{*}_{20}\bar{f}^{*}_{40}}{80\bar{f}^{*2}_{20}}X^{3}, \ \
y=Y, \ \
t=\left(1-\frac{\bar{f}^{*}_{30}}{2\bar{f}^{*}_{20}}X+\frac{45\bar{f}^{*2}_{30}
-48\bar{f}^{*}_{20}\bar{f}^{*}_{40}}{80\bar{f}^{*2}_{20}}X^{2}\right)\tau,
\end{eqnarray*}
that
\begin{equation}\label{a11a11}
\begin{array}{rrl}
&&\dfrac{dx}{dt}=y, \vspace{2mm}\\
&&\dfrac{dy}{dt}=\bar{g}^{*}_{00}+\bar{g}^{*}_{10}x+\bar{g}^{*}_{01}y
+\bar{g}^{*}_{20}x^{2}+\bar{g}^{*}_{11}xy+\bar{g}^{*}_{30}x^{3}
+\bar{g}^{*}_{21}x^{2}y+\bar{g}^{*}_{40}x^{4}\\
&& \ \ \ \ \ \ \ \ +\bar{g}^{*}_{31}x^{3}y+O(|x,y|^{4}),
\end{array}
\end{equation}
where $\bar{g}^{*}_{00}=\bar{g}^{*}_{10}=\bar{g}^{*}_{01}=\bar{g}^{*}_{11}
=\bar{g}^{*}_{30}=\bar{g}^{*}_{40}=0$ when $\kappa=0$.\\

($V$) Eliminate the $x^{2}y$-term in system (\ref{a11a11}) when $\kappa=0$. It's easy to know that $\bar{g}^{*}_{20}=\frac{s^2 (x^{*2}-a)}{2 x^{*} (a+x^{*} (c+x^{*}))}+O(\kappa)$, $g^{*}_{20} \neq 0$ for small $\kappa$. Setting
\begin{eqnarray*}
x=X,\
y=Y+\frac{\bar{g}^{*}_{21}}{3\bar{g}^{*}_{20}}Y^{2}+\frac{\bar{g}^{*2}_{21}}{36\bar{g}^{*2}_{20}}Y^{3}, \
\tau=\left(1+\frac{\bar{g}^{*}_{21}}{3\bar{g}^{*}_{20}}Y+\frac{\bar{g}^{*2}_{21}}{36\bar{g}^{*2}_{20}}Y^{2}\right)t,
\end{eqnarray*}
we get an equivalent system to (\ref{a11a11}) as follows
\begin{equation}\label{a12}
\begin{array}{rrl}
&&\dfrac{dx}{dt}=y,  \vspace{2mm}\\
&&\dfrac{dy}{dt}=\bar{h}^{*}_{00}+\bar{h}^{*}_{10}x+\bar{h}^{*}_{01}y+\bar{h}^{*}_{20}x^{2}
+\bar{h}^{*}_{11}xy+\bar{h}^{*}_{31}x^{3}y+T_{1}(x,y,\kappa),
\end{array}
\end{equation}
where $\bar{h}^{*}_{00}=\bar{h}^{*}_{10}=\bar{h}^{*}_{01}=\bar{h}^{*}_{11}=0$ when $\kappa=0$, and $T_{1}(x,y,\kappa)$ possesses the property of (\ref{a13}).\\

($VI$) Change $\bar{g}^{*}_{20}$ and $g_{31}$ to 1 in system (\ref{a12}). A simple calculation shows that $\bar{h}^{*}_{20}=\bar{g}^{*}_{20} \neq 0$ and $\overline{h}_{31} \neq 0$ for small $\kappa$. Letting
\begin{eqnarray*}
x=\bar{h}^{*\frac{1}{5}}_{20}\bar{h}^{*-\frac{2}{5}}_{31}X,\ \
y=\bar{h}^{*\frac{4}{5}}_{20}\bar{h}^{*-\frac{3}{5}}_{31}Y,\ \
t=\bar{h}^{*-\frac{3}{5}}_{20}\bar{h}^{*\frac{1}{5}}_{31}\tau,
\end{eqnarray*}
model (\ref{a12}) can be expressed as follows
\begin{equation}\label{a14}
\begin{array}{rrl}
&&\dfrac{dx}{dt}=y,  \vspace{2mm}\\
&&\dfrac{dy}{dt}=\bar{j}^{*}_{00}+\bar{j}^{*}_{10}x+\bar{j}^{*}_{01}y+\bar{j}^{*}_{11}xy+x^{2}
+x^{3}y+T_{2}(x,y,\kappa),
\end{array}
\end{equation}
where $\bar{j}^*_{00}=\bar{j}^*_{10}=\bar{j}^*_{01}=\bar{j}^*_{11}=0$ when $\kappa=0$, and $T_{2}(x,y,\kappa)$ possesses the property of (\ref{a13}).\\

($VII$) Eliminate the $\bar{h}^{*}_{10}$-term in system (\ref{a14}). Finally, from
\begin{eqnarray*}
x=X-\frac{\bar{j}^{*}_{10}}{2},\ \
y=Y,
\end{eqnarray*}
we have that
\begin{equation}\label{a15}
\begin{array}{rrl}
&&\dfrac{dx}{dt}=y,  \vspace{2mm}\\
&&\dfrac{dy}{dt}=\bar{\psi}_{1}+\bar{\psi}_{2}y+\bar{\psi}_{3}xy+x^{2}+x^{3}y+T_{3}(x,y,\kappa),
\end{array}
\end{equation}
in which
\begin{eqnarray*}
\bar{\psi}_{1}=\bar{j}^{*}_{00}-\frac{1}{4}\bar{j}^{*2}_{10}, \ \ \
\bar{\psi}_{2}=\bar{j}^{*}_{01}-\frac{\bar{j}^{*3}_{10}}{8}-\frac{\bar{j}^{*}_{11} \bar{j}^{*}_{10}}{2}, \ \ \
\bar{\psi}_{3}=\bar{j}^{*}_{11}+\frac{3}{4}\bar{j}^{*2}_{10}.
\end{eqnarray*}
Note that
$\bar{\psi}_{1}=\bar{\psi}_{2}=\bar{\psi}_{3}=0$ when $\kappa=0$ and $T_{3}(x,y,\kappa)$ possesses the property of (\ref{a13}). Using the software Mathematica we obtain
\begin{eqnarray*}
\left|\frac{\partial(\bar{\psi}_{1},\bar{\psi}_{2},\bar{\psi}_{3})}{(\kappa_{1},\kappa_{2},\kappa_{3})}\right|_{\kappa=0}
=\frac{2^{\frac{3}{5}} q N_1 N_2 N_3^{\frac{4}{5}}}{s^2 (x^{*2}-a)^{\frac{21}{5}} (b+x^{*})^{\frac{23}{5}} (a+c x^{*}+x^{*2})^{\frac{33}{5}}},
\end{eqnarray*}
in which
\begin{eqnarray*}
&&N_1=a^2+3 a x^{*} (c+2 x^{*})-x^{*3} (c+3 x^{*}),\\
&&N_2=-a^3 (b-x^{*}) (2 b+3 x^{*})+a^2 x^{*} (-7 b^2 c+x^{*2} (b+5 c)-6 b x^{*} (b+c)+9 x^{*3})-a x^{*2} (2 x^{*2}\\
&& \ \ \ \ \ \ \ \times (15 b^2+20 b c+c^2)+9 b^2 c^2+3 x^{*3} (17 b+4 c)+b c x^{*} (28 b+13 c)+23 x^{*4})+x^{*4} (b^2 c^2\\
&& \ \ \ \ \ \ \ +x^{*3} (11 b-c)+6 b x^{*2} (b+c)+3 b c x^{*} (b+c)+3 x^{*4}),\\
&&N_3=a^6 (2 b^2-x^{*2})+a^5 x^{*} (2 b^2 (5 c+3 x^{*})+b x^{*} (c-20 x^{*})-9 x^{*2} (c+4 x^{*}))+a^4 x^{*2} (3 x^{*2}(24\\
&& \ \ \ \ \ \ \ \times b^2+35 b c-3 c^2)+20 b^2 c^2+x^{*3} (184 b-63 c)+b c x^{*} (56 b+29 c)-17 x^{*4})+a^3 x^{*4}(-2\\
&& \ \ \ \ \ \ \ \times b^2 (7 c^2+14 c x^{*}+30 x^{*2})+b (43 c^3+280 c^2 x^{*}+670 c x^{*2}+424 x^{*3})+x^{*} (11 c^3+78 c^2 x^{*}\\
&& \ \ \ \ \ \ \ +330 c x^{*2}+344 x^{*3}))-a^2 x^{*4} (24 b^2 c^4+2 x^{*4} (275 b^2+569 b c+18 c^2)+c x^{*3}(952 b^2+450\\
&& \ \ \ \ \ \ \ \times b c+9 c^2)+b c^3 x^{*} (194 b+9 c)+b c^2 x^{*2} (618 b+91 c)+6 x^{*5} (160 b+39 c)+363 x^{*6})+a\\
&& \ \ \ \ \ \ \ \times x^{*6} (8 b^2 c^4+5 x^{*4} (62 b^2+45 b c-6 c^2)+c x^{*3} (434 b^2+40 b c+c^2)+2 b c^3 x^{*}(34 b+3 c)+\\
&& \ \ \ \ \ \ \ 7 b c^2 x^{*2} (34 b+3 c)+x^{*5} (332 b-33 c)+76 x^{*6})+x^{*9} (-3 x^{*3} (12 b^2-3 b c+c^2)+c x^{*2}\times\\
&& \ \ \ \ \ \ \ (-32 b^2+5 b c-3 c^2)-b c^3 (2 b+c)-5 b c^2 x^{*} (2 b+c)+3 x^{*4} (3 c-8 b)-3 x^{*5}).
\end{eqnarray*}

Conditions $b\neq b_{\pm}, b_{\pm}^{*}$ and $m=\frac{s (b+x^{*})^3 (a^2+3 a x^{*} (c+2 x^{*})-x^{*3}(c+3 x^{*}))}{2 x^{*2} (a+x^{*} (c+x^{*}))^2}>0$ show that $N_1, N_2, N_3 \neq 0$ and obviously we have $\left|\frac{\partial(\bar{\psi}_{1},\bar{\psi}_{2},\bar{\psi}_{3})}{(\kappa_{1},\kappa_{2},\kappa_{3})}\right|
_{\kappa=0}\neq0$. Moreover, system (\ref{a15}) is equivalent to~(\ref{a7}). Based on Li et al. \cite{li2015codimension}, we can conclude that model (\ref{a15}) is the universal unfolding of a Bogdanov-Takens singularity (cusp case) of codimension 3. The remainder term $R_{3}(x,y,\kappa)$ satisfying the property of (\ref{a13}) has no influence on the bifurcation phenomena. The dynamics of system (\ref{A1}) in a small neighborhood of the positive equilibrium $E^{*}=(x^{*}, nx^{*})$ as $(m,q,s)$ varies  near $(m+\kappa_{1}, q+\kappa_{2}, s+\kappa_{3})$ are equivalent to system (\ref{a15}) in a small neighborhood of $(0, 0, 0)$ as $(\bar{\psi}_{1}, \bar{\psi}_{2}, \bar{\psi}_{3})$ varies near $(0, 0, 0)$. This completes the proof.
\end{proof}


The following are the remaining coefficients in the previous proof.

\begin{eqnarray*}
&&\bar{a}^{*}_{00}=\frac{1}{2 x^{*} (b+\kappa_2+x^{*})} \left(-2 \kappa _1 x^{*2}+\frac{\kappa _2 s (b+x^{*})^2 (a^2+3 a x^{*} (c+2 x^{*})-x^{*3} (c+3 x^{*}))}{(a+x^{*} (c+x^{*}))^2}\right),\\
&&\bar{a}^{*}_{10}=\frac{1}{2 (b+\kappa _2+x^{*})^2} (2 s (b+x^{*})^2-2 \kappa _1 (b+\kappa _2)+\frac{\kappa _2 s}{x^{*2} (a+x^{*} (c+x^{*}))^2} ((b+x^{*}) (a^2 (b^2+\\
&& \ \ \ \ \ \ \ \ 3 x^{*2})+3 a b^2 x^{*} (c+2 x^{*})+a x^{*3} (5 c+2 x^{*})-b^2 x^{*3} (c+3 x^{*})+x^{*4} (4 c^2+9 c x^{*}+7 x^{*2}))\\
&& \ \ \ \ \ \ \ \ +\kappa _2 (a^2 (b^2+b x^{*}+2 x^{*2})+a x^{*} (3 b^2 (c+2 x^{*})+3 b x^{*} (c+2 x^{*})+4 x^{*2} (c+x^{*}))-x^{*3} (b^2\\
&& \ \ \ \ \ \ \ \ \times(c+3 x^{*})-b x^{*} (c+3 x^{*})+2 x^{*} (c+x^{*})^2)))),\\
&&\bar{a}^{*}_{01}=-\frac{q x^{*}}{a+x^{*} (c+x^{*})},\\
&&\bar{a}^{*}_{20}=\frac{1}{2 (a+x^{*} (c+x^{*}))^2} (2 s (a (c+3 x^{*})-x^{*3})+\frac{s}{x^{*2}}(a^2 (x^{*}-b)-a x^{*} (3 b (c+2 x^{*})+x^{*} (c\\
&& \ \ \ \ \ \ \ \ +6 x^{*}))+x^{*3} (b (c+3 x^{*})+x^{*} (x^{*}-c)))+\frac{b+\kappa _2}{x^{*2} (b+\kappa _2+x^{*})^3} (s (b+x^{*})^3 (a^2+3 a x^{*} (c\\
&& \ \ \ \ \ \ \ \ +2 x^{*})-x^{*3} (c+3 x^{*}))+2 \kappa _1 x^{*2} (a+x^{*} (c+x^{*}))^2) ),\\
&&\bar{a}^{*}_{11}=\frac{q (x^{*2}-a)}{(a+x^{*} (c+x^{*}))^2},\\
&&\bar{a}^{*}_{30}=\frac{1}{2 (a+x^{*} (c+x^{*}))^3}(2 s (a^2-a (c^2+4 c x^{*}+6 x^{*2})+x^{*4})-\frac{(b+\kappa _2) (a+x^{*} (c+x^{*}))}{x^{*2} (b+\kappa _2+x^{*})^4}(s\\
&& \ \ \ \ \ \ \ \ \times (b+x^{*})^3 (a^2+3 a x^{*} (c+2 x^{*})+x^{*3} (-(c+3 x^{*})))+2 \kappa _1 x^{*2} (a+x^{*} (c+x^{*}))^2)),\\
&&\bar{a}^{*}_{21}=\frac{q (a (c+3 x^{*})-x^{*3})}{(a+x^{*} (c+x^{*}))^3},\\
&&\bar{a}^{*}_{40}=\frac{b+\kappa _2}{2 x^{*2} (a+x^{*} (c+x^{*}))^2 (b+\kappa _2+x^{*})^5}(s (b+x^{*})^3 (a^2+3 a x^{*}(c+2 x^{*}) -x^{*3} (c+3x^{*}))\\
&& \ \ \ \ \ \ \ \ +2 \kappa _1 x^{*2} (a+x^{*} (c+x^{*}))^2)-\frac{s}{(a+x^{*} (c+x^{*}))^4}(a^2 (2 c+5 x^{*})-a (c^3+5 c^2 x^{*}+\\
&& \ \ \ \ \ \ \ \ 10 c x^{*2}+10 x^{*3})+x^{*5}),\\
&&\bar{a}^{*}_{31}=\frac{q (a^2-a (c^2+4 c x^{*}+6 x^{*2})+x^{*4})}{(a+x^{*} (c+x^{*}))^4};\
\bar{b}^{*}_{00}=\frac{\kappa _3 s^2 x^{*} (a+x^{*} (c+x^{*}))}{s (a+x^{*} (c+x^{*}))+\kappa _3 q x^{*}},
\end{eqnarray*}

\begin{eqnarray*}
&&\bar{b}^{*}_{10}=\frac{s^3 (a+x^{*} (c+x^{*}))^2}{q x^{*} (s (a+x^{*} (c+x^{*}))+\kappa _3 q x^{*})},\
\bar{b}^{*}_{01}=\frac{s (\kappa _3 q x^{*}-s (a+x^{*} (c+x^{*})))}{s (a+x^{*} (c+x^{*}))+\kappa _3 q x^{*}},\\
&&\bar{b}^{*}_{20}=-\frac{s^3 (a+x^{*} (c+x^{*}))^2}{q x^{*2} (s (a+x^{*} (c+x^{*}))+\kappa _3 q x^{*})},\
\bar{b}^{*}_{11}=\frac{2 s^2 (a+x^{*} (c+x^{*}))}{x^{*} (s (a+x^{*} (c+x^{*}))+\kappa _3 q x^{*})},\\
&&\bar{b}^{*}_{02}=-\frac{q s}{s (a+x^{*} (c+x^{*}))+\kappa _3 q x^{*}},\
\bar{b}^{*}_{30}=\frac{s^3 (a+x^{*} (c+x^{*}))^2}{q x^{*3} (s (a+x^{*} (c+x^{*}))+\kappa _3 q x^{*})},\\
&&\bar{b}^{*}_{21}=-\frac{2 s^2 (a+x^{*} (c+x^{*}))}{x^{*2} (s (a+x^{*} (c+x^{*}))+\kappa _3 q x^{*})},\
\bar{b}^{*}_{12}=\frac{q s}{x^{*} (s (a+x^{*} (c+x^{*}))+\kappa _3 q x^{*})},\\
&&\bar{b}^{*}_{40}=-\frac{s^3 (a+x^{*} (c+x^{*}))^2}{q x^{*4} (s (a+x^{*} (c+x^{*}))+\kappa _3 q x^{*})},\
\bar{b}^{*}_{31}=\frac{2 s^2 (a+x^{*} (c+x^{*}))}{x^{*3} (s (a+x^{*} (c+x^{*}))+\kappa _3 q x^{*})},\\
&&\bar{b}^{*}_{22}=-\frac{q s}{x^{*2} (s (a+x^{*} (c+x^{*}))+\kappa _3 q x^{*})};\
\bar{c}^{*}_{00}=\frac{\bar{a}_{00}^2 \bar{b}^{*}_{02}}{\bar{a}_{01}}-\bar{a}^{*}_{00} \bar{b}^{*}_{01}+\bar{a}^{*}_{01} \bar{b}^{*}_{00},\\
&&\bar{c}^{*}_{10}=\frac{\bar{a}^{*2}_{00} \bar{b}^{*}_{12}}{\bar{a}^{*}_{01}}-\bar{a}^{*}_{00} \bar{b}^{*}_{11}+\bar{a}^{*}_{01} \bar{b}^{*}_{10}+\bar{a}^{*}_{10} \left(\frac{2 \bar{a}^{*}_{00} \bar{b}^{*}_{02}}{\bar{a}^{*}_{01}}-\bar{b}^{*}_{01}\right)+\bar{a}^{*}_{11} \left(\bar{b}^{*}_{00}-\frac{\bar{a}^{*2}_{00} \bar{b}^{*}_{02}}{\bar{a}^{*2}_{01}}\right),\\
&&\bar{c}^{*}_{01}=-\frac{\bar{a}^{*}_{00} (\bar{a}^{*}_{11}+2 \bar{b}^{*}_{02})}{\bar{a}^{*}_{01}}+\bar{a}^{*}_{10}+\bar{b}^{*}_{01},\\
&&\bar{c}^{*}_{20}=\frac{\bar{a}^{*2}_{00}\bar{a}^{*2}_{11}\bar{b}^{*}_{02}}{\bar{a}^{*3}_{01}}+\bar{a}^{*}_{11} \bar{b}^{*}_{10}+\bar{a}^{*}_{21} \bar{b}^{*}_{00}-\bar{a}^{*}_{20} \bar{b}^{*}_{01}-\bar{a}^{*}_{10} \bar{b}^{*}_{11}+\bar{a}^{*}_{01} \bar{b}^{*}_{20}-\bar{a}^{*}_{00} \bar{b}^{*}_{21}-\frac{\bar{a}^{*}_{00}}{\bar{a}^{*2}_{01}} (2 \bar{a}^{*}_{10} \bar{a}^{*}_{11} \bar{b}^{*}_{02}+\\
&& \ \ \ \ \ \ \ \ \bar{a}^{*}_{00} (\bar{a}^{*}_{21} \bar{b}^{*}_{02}+\bar{a}^{*}_{11} \bar{b}^{*}_{12}))+\frac{1}{\bar{a}^{*}_{01}}(\bar{a}^{*2}_{00} \bar{b}^{*}_{22}+2 \bar{a}^{*}_{20} \bar{a}^{*}_{00} \bar{b}^{*}_{02}+2 \bar{a}^{*}_{10} \bar{a}^{*}_{00} \bar{b}^{*}_{12}+\bar{a}^{*2}_{10} \bar{b}^{*}_{02}),\\
&&\bar{c}^{*}_{11}=2 \bar{a}^{*}_{20}+\bar{b}^{*}_{11}+\frac{1}{\bar{a}^{*2}_{01}} (\bar{a}^{*}_{00} \bar{a}^{*}_{11} (\bar{a}^{*}_{11}+2 \bar{b}^{*}_{02})-\bar{a}^{*}_{01} (\bar{a}^{*}_{10} (\bar{a}^{*}_{11}+2 \bar{b}^{*}_{02})+2 \bar{a}^{*}_{00} (\bar{a}^{*}_{21}+\bar{b}^{*}_{12}))),\\
&&\bar{c}^{*}_{02}=\frac{\bar{a}^{*}_{11}+\bar{b}^{*}_{02}}{\bar{a}^{*}_{01}},\\
&&\bar{c}^{*}_{30}=-\frac{\bar{a}^{*2}_{00} \bar{a}^{*3}_{11} \bar{b}^{*}_{02}}{\bar{a}^{*4}_{01}}+\bar{a}^{*}_{11} \bar{b}^{*}_{20}+\bar{a}^{*}_{31} \bar{b}^{*}_{00}-\bar{a}^{*}_{30} \bar{b}^{*}_{01}+\bar{a}^{*}_{21} \bar{b}^{*}_{10}-\bar{a}^{*}_{20} \bar{b}^{*}_{11}-\bar{a}^{*}_{10} \bar{b}^{*}_{21}+\bar{a}^{*}_{01} \bar{b}^{*}_{30}-\bar{a}^{*}_{00} \bar{b}^{*}_{31}\\
&& \ \ \ \ \ \ \ \ +\frac{\bar{a}^{*}_{00} \bar{a}^{*}_{11}}{\bar{a}^{*3}_{01}} (2 (\bar{a}^{*}_{10} \bar{a}^{*}_{11}+\bar{a}^{*}_{00} \bar{a}^{*}_{21}) \bar{b}^{*}_{02}+\bar{a}^{*}_{00} \bar{a}^{*}_{11} \bar{b}^{*}_{12})+\frac{1}{\bar{a}^{*}_{01}}(\bar{a}^{*2}_{10} \bar{b}^{*}_{12}+2 \bar{a}^{*}_{10} (\bar{a}^{*}_{20} \bar{b}^{*}_{02}+\bar{a}^{*}_{00} \bar{b}^{*}_{22})\\
&& \ \ \ \ \ \ \ \ +2 \bar{a}^{*}_{00} (\bar{a}^{*}_{30} \bar{b}^{*}_{02}+\bar{a}^{*}_{20} \bar{b}^{*}_{12}))-\frac{1}{\bar{a}^{*2}_{01}} (\bar{a}^{*}_{11} \bar{a}^{*2}_{10} \bar{b}^{*}_{02}+2 \bar{a}^{*}_{00} \bar{a}^{*}_{10} (\bar{a}^{*}_{21} \bar{b}^{*}_{02}+\bar{a}^{*}_{11} \bar{b}^{*}_{12})+\bar{a}^{*}_{00} (\bar{a}^{*}_{00} (\bar{a}^{*}_{31} \bar{b}^{*}_{02}\\
&& \ \ \ \ \ \ \ \ +\bar{a}^{*}_{21} \bar{b}^{*}_{12})+\bar{a}^{*}_{11} (2 \bar{a}^{*}_{20} \bar{b}^{*}_{02}+\bar{a}^{*}_{00} \bar{b}^{*}_{22}))),\\
&&\bar{c}^{*}_{21}=3 \bar{a}^{*}_{30}+\bar{b}^{*}_{21}+\frac{1}{\bar{a}^{*3}_{01}} {(-\bar{a}^{*2}_{01}(\bar{a}^{*}_{20} (\bar{a}^{*}_{11}+2 \bar{b}^{*}_{02})+2 \bar{a}^{*}_{10} (\bar{a}^{*}_{21}+\bar{b}^{*}_{12})+\bar{a}^{*}_{00} (3 \bar{a}^{*}_{31}+2 \bar{b}^{*}_{22}))}+\bar{a}^{*}_{01}\\
&& \ \ \ \ \ \ \ \ \times (\bar{a}^{*}_{10} \bar{a}^{*}_{11} (\bar{a}^{*}_{11}+2 \bar{b}^{*}_{02})+\bar{a}^{*}_{00} (\bar{a}^{*}_{21} (3 \bar{a}^{*}_{11}+2 \bar{b}^{*}_{02})+2 \bar{a}^{*}_{11} \bar{b}^{*}_{12}))-\bar{a}^{*}_{00} \bar{a}^{*2}_{11} (\bar{a}^{*}_{11}+2 \bar{b}^{*}_{02})),\\
&&\bar{c}^{*}_{12}=\frac{1}{\bar{a}^{*2}_{01}}(\bar{a}^{*}_{01} (2 \bar{a}^{*}_{21}+\bar{b}^{*}_{12})-\bar{a}^{*}_{11} (\bar{a}^{*}_{11}+\bar{b}^{*}_{02})),
\end{eqnarray*}
\begin{eqnarray*}
&&\bar{c}^{*}_{40}=\frac{\bar{a}^{*2}_{00} \bar{a}^{*4}_{11} \bar{b}^{*}_{02}}{\bar{a}^{*5}_{01}}+\bar{a}^{*}_{11} \bar{b}^{*}_{30}-\bar{a}^{*}_{40} \bar{b}^{*}_{01}+\bar{a}^{*}_{31} \bar{b}^{*}_{10}-\bar{a}^{*}_{30} \bar{b}^{*}_{11}+\bar{a}^{*}_{21} \bar{b}^{*}_{20}-\bar{a}^{*}_{20} \bar{b}^{*}_{21}-\bar{a}^{*}_{10} \bar{b}^{*}_{31}+\bar{a}^{*}_{01} \bar{b}^{*}_{40}-\\
&& \ \ \ \ \ \ \ \ \frac{\bar{a}^{*}_{00} \bar{a}^{*2}_{11}}{\bar{a}^{*4}_{01}} (2 \bar{a}^{*}_{10} \bar{a}^{*}_{11} \bar{b}^{*}_{02}+\bar{a}^{*}_{00} (3 \bar{a}^{*}_{21} \bar{b}^{*}_{02}+\bar{a}^{*}_{11} \bar{b}^{*}_{12})) +\frac{1}{\bar{a}^{*}_{01}} (\bar{a}^{*2}_{10} \bar{b}^{*}_{22}+2 \bar{a}^{*}_{30} \bar{a}^{*}_{10} \bar{b}^{*}_{02}+\bar{a}^{*2}_{20} \bar{b}^{*}_{02}+2\\
&& \ \ \ \ \ \ \ \ \times \bar{a}^{*}_{00} (\bar{a}^{*}_{40} \bar{b}^{*}_{02}+\bar{a}^{*}_{30} \bar{b}^{*}_{12})+2 \bar{a}^{*}_{20} (\bar{a}^{*}_{10} \bar{b}^{*}_{12}+\bar{a}^{*}_{00} \bar{b}^{*}_{22}))+\frac{1}{\bar{a}^{*3}_{01}} (\bar{a}^{*2}_{10} \bar{a}^{*2}_{11} \bar{b}^{*}_{02}+2 \bar{a}^{*}_{00} \bar{a}^{*}_{10} \bar{a}^{*}_{11} (2 \bar{a}^{*}_{21} \bar{b}^{*}_{02}\\
&& \ \ \ \ \ \ \ \ +\bar{a}^{*}_{11} \bar{b}^{*}_{12})+\bar{a}^{*}_{00} (\bar{a}^{*2}_{11} (2 \bar{a}^{*}_{20} \bar{b}^{*}_{02}+\bar{a}^{*}_{00} \bar{b}^{*}_{22})+2 \bar{a}^{*}_{00} \bar{a}^{*}_{11} (\bar{a}^{*}_{31} \bar{b}^{*}_{02}+\bar{a}^{*}_{21} \bar{b}^{*}_{12})+\bar{a}^{*}_{00} \bar{a}^{*2}_{21} \bar{b}^{*}_{02}))-\frac{1}{\bar{a}^{*2}_{01}}\\
&& \ \ \ \ \ \ \ \ \times (\bar{a}^{*2}_{10} (\bar{a}^{*}_{21} \bar{b}^{*}_{02}+\bar{a}^{*}_{11} \bar{b}^{*}_{12})+2 \bar{a}^{*}_{10} (\bar{a}^{*}_{00} (\bar{a}^{*}_{31} \bar{b}^{*}_{02}+\bar{a}^{*}_{21} \bar{b}^{*}_{12})+\bar{a}^{*}_{11} (\bar{a}^{*}_{20} \bar{b}^{*}_{02}+\bar{a}^{*}_{00} \bar{b}^{*}_{22}))+\bar{a}^{*}_{00} (2 \bar{a}^{*}_{11}\\
&& \ \ \ \ \ \ \ \ \times  \bar{a}^{*}_{30} \bar{b}^{*}_{02}+2 \bar{a}^{*}_{20} (\bar{a}^{*}_{21} \bar{b}^{*}_{02}+\bar{a}^{*}_{11} \bar{b}^{*}_{12})+\bar{a}^{*}_{00} (\bar{a}^{*}_{31} \bar{b}^{*}_{12}+\bar{a}^{*}_{21} \bar{b}^{*}_{22}))),\\
&&\bar{c}^{*}_{31}=4 \bar{a}^{*}_{40}+{\bar{b}^{*}_{31}}+\frac{1}{\bar{a}^{*4}_{01}}
{(-\bar{a}^{*3}_{01}(\bar{a}^{*}_{30} (\bar{a}^{*}_{11}+2 \bar{b}^{*}_{02})+2 \bar{a}^{*}_{20} (\bar{a}^{*}_{21}+\bar{b}^{*}_{12})+\bar{a}^{*}_{10} (3 \bar{a}^{*}_{31}+2 \bar{b}^{*}_{22}))}+\bar{a}^{*2}_{01}\\
&& \ \ \ \ \ \ \ \ \times (\bar{a}^{*}_{11} (2 \bar{a}^{*}_{20} \bar{b}^{*}_{02}+\bar{a}^{*}_{10} (3 \bar{a}^{*}_{21}+2 \bar{b}^{*}_{12})+2 \bar{a}^{*}_{00} (2 \bar{a}^{*}_{31}+\bar{b}^{*}_{22}))+2 (\bar{a}^{*}_{10} \bar{a}^{*}_{21} \bar{b}^{*}_{02}+\bar{a}^{*}_{00} (\bar{a}^{*}_{31} \bar{b}^{*}_{02}+\bar{a}^{*}_{21}\\
&& \ \ \ \ \ \ \ \ \times (\bar{a}^{*}_{21}+\bar{b}^{*}_{12})))+\bar{a}^{*}_{20} \bar{a}^{*2}_{11})-\bar{a}^{*}_{11} \bar{a}^{*}_{01} (\bar{a}^{*}_{10} \bar{a}^{*}_{11} (\bar{a}^{*}_{11}+2 \bar{b}^{*}_{02})+2 \bar{a}^{*}_{00} (2 \bar{a}^{*}_{21} (\bar{a}^{*}_{11}+\bar{b}^{*}_{02})+\bar{a}^{*}_{11}\times\\
&& \ \ \ \ \ \ \ \ \bar{b}^{*}_{12}))+\bar{a}^{*}_{00} \bar{a}^{*3}_{11} (\bar{a}^{*}_{11}+2 \bar{b}^{*}_{02})),\\
&&\bar{c}^{*}_{22}=\frac{1}{\bar{a}^{*3}_{01}} (\bar{a}^{*2}_{11} \bar{b}^{*}_{02}-\bar{a}^{*}_{01} \bar{a}^{*}_{11} (3 \bar{a}^{*}_{21}+\bar{b}^{*}_{12})+\bar{a}^{*}_{01} (\bar{a}^{*}_{01} (3 \bar{a}^{*}_{31}+\bar{b}^{*}_{22})-\bar{a}^{*}_{21} \bar{b}^{*}_{02})+\bar{a}^{*3}_{11});\\
&&\bar{d}^{*}_{00}=\bar{c}^{*}_{00},\
\bar{d}^{*}_{10}=\bar{c}^{*}_{10}-\bar{c}^{*}_{00} \bar{c}^{*}_{02},\
\bar{d}^{*}_{01}=\bar{c}^{*}_{01},\
\bar{d}^{*}_{20}=\bar{c}^{*}_{00} \bar{c}^{*2}_{02}-\frac{\bar{c}^{*}_{10} \bar{c}^{*}_{02}}{2}+\bar{c}^{*}_{20},\
\bar{d}^{*}_{11}=\bar{c}^{*}_{11},\\
&&\bar{d}^{*}_{30}=\frac{1}{2} (\bar{c}^{*}_{10}-2 \bar{c}^{*}_{00} \bar{c}^{*}_{02}) \bar{c}^{*2}_{02}+\bar{c}^{*}_{30},\
\bar{d}^{*}_{21}=\frac{\bar{c}^{*}_{02} \bar{c}^{*}_{11}}{2}+\bar{c}^{*}_{21},\
\bar{d}^{*}_{12}=2 \bar{c}^{*2}_{02}+\bar{c}^{*}_{12},\\
&&\bar{d}^{*}_{40}=\bar{c}^{*}_{00} \bar{c}^{*4}_{02}+\frac{1}{4} (\bar{c}^{*}_{02} (\bar{c}^{*}_{20}-2 \bar{c}^{*}_{02} \bar{c}^{*}_{10})+2 \bar{c}^{*}_{30}) \bar{c}^{*}_{02}+\bar{c}^{*}_{40},\
\bar{d}^{*}_{31}=\bar{c}^{*}_{02} \bar{c}^{*}_{21}+\bar{c}^{*}_{31},\\
&&\bar{d}^{*}_{22}=-\bar{c}^{*3}_{02}+\frac{3 \bar{c}^{*}_{12} \bar{c}^{*}_{02}}{2}+\bar{c}^{*}_{22};\
\bar{e}^{*}_{00}=\bar{d}^{*}_{00},\
\bar{e}^{*}_{10}=\bar{d}^{*}_{10},\
\bar{e}^{*}_{01}=\bar{d}^{*}_{01},
\bar{e}^{*}_{20}=\bar{d}^{*}_{20}-\frac{\bar{d}^{*}_{00} \bar{d}^{*}_{12}}{2},\\
&&\bar{e}^{*}_{11}=\bar{d}^{*}_{11},\
\bar{e}^{*}_{30}=\bar{d}^{*}_{30}-\frac{\bar{d}^{*}_{10} \bar{d}^{*}_{12}}{3},\
\bar{e}^{*}_{21}=\bar{d}^{*}_{21},\
\bar{e}^{*}_{40}=\frac{1}{4} \bar{d}^{*}_{00} \bar{d}^{*2}_{12}-\frac{\bar{d}^{*}_{20} \bar{d}^{*}_{12}}{6}+\bar{d}^{*}_{40},\\
&&\bar{e}^{*}_{31}=\frac{\bar{d}^{*}_{11} \bar{d}^{*}_{12}}{6}+\bar{d}^{*}_{31},\
\bar{e}^{*}_{22}=\bar{d}^{*}_{22};\
\bar{f}^{*}_{00}=\bar{e}^{*}_{00},\
\bar{f}^{*}_{10}=\bar{e}^{*}_{10},\
\bar{f}^{*}_{01}=\bar{e}^{*}_{01},\
\bar{f}^{*}_{20}=\bar{e}^{*}_{20},\
\bar{f}^{*}_{11}=\bar{e}^{*}_{11},\\
&&\bar{f}^{*}_{30}=\bar{e}^{*}_{30}-\frac{\bar{e}^{*}_{00} \bar{e}^{*}_{22}}{3},\
\bar{f}^{*}_{21}=\bar{e}^{*}_{21},\
\bar{f}^{*}_{40}=\bar{e}^{*}_{40}-\frac{\bar{e}^{*}_{10} \bar{e}^{*}_{22}}{4},\
\bar{f}^{*}_{31}=\bar{e}^{*}_{31};\
\bar{g}^{*}_{00}=\bar{f}^{*}_{00},\\
&&\bar{g}^{*}_{10}=\bar{f}^{*}_{10}-\frac{\bar{f}^{*}_{00} \bar{f}^{*}_{30}}{2 \bar{f}^{*}_{20}},\
\bar{g}^{*}_{01}=\bar{f}^{*}_{01},\
\bar{g}^{*}_{20}=\frac{9 \bar{f}^{*}_{00} \bar{f}^{*2}_{30}}{16 \bar{f}^{*2}_{20}}+\bar{f}^{*}_{20}-\frac{3 (5 \bar{f}^{*}_{10} \bar{f}^{*}_{30}+4 \bar{f}^{*}_{00} \bar{f}^{*}_{40})}{20 \bar{f}^{*}_{20}},\\
&&\bar{g}^{*}_{11}=\bar{f}^{*}_{11}-\frac{\bar{f}^{*}_{01} \bar{f}^{*}_{30}}{2 \bar{f}^{*}_{20}},\
\bar{g}^{*}_{30}=\frac{\bar{f}^{*}_{10} (35 \bar{f}^{*2}_{30}-32 \bar{f}^{*}_{20} \bar{f}^{*}_{40})}{40 \bar{f}^{*2}_{20}},\\
&&\bar{g}^{*}_{21}=\bar{f}^{*}_{21}-\frac{3 (20 \bar{f}^{*}_{11} \bar{f}^{*}_{20} \bar{f}^{*}_{30}+\bar{f}^{*}_{01} (16 \bar{f}^{*}_{20} \bar{f}^{*}_{40}-15 \bar{f}^{*2}_{30}))}{80 \bar{f}^{*2}_{20}},\
\bar{g}^{*}_{40}=\frac{\bar{f}^{*}_{10} \bar{f}^{*}_{30} (16 \bar{f}^{*}_{20} \bar{f}^{*}_{40}-15 \bar{f}^{*2}_{30})}{64 \bar{f}^{*3}_{20}},\\
&&\bar{g}^{*}_{31}=\frac{7 \bar{f}^{*}_{11} \bar{f}^{*2}_{30}}{8 \bar{f}^{*2}_{20}}+\bar{f}^{*}_{31}-\frac{5 \bar{f}^{*}_{21} \bar{f}^{*}_{30}+4 \bar{f}^{*}_{11} \bar{f}^{*}_{40}}{5 \bar{f}^{*}_{20}};\
\bar{h}^{*}_{00}=\bar{g}^{*}_{00},\
\bar{h}^{*}_{10}=\bar{g}^{*}_{10},\
\bar{h}^{*}_{01}=\bar{g}^{*}_{01}-\frac{\bar{g}^{*}_{00} \bar{g}^{*}_{21}}{\bar{g}^{*}_{20}},\\
&&\bar{h}^{*}_{20}=\bar{g}^{*}_{20},\
\bar{h}^{*}_{11}=\bar{g}^{*}_{11}-\frac{\bar{g}^{*}_{10} \bar{g}^{*}_{21}}{\bar{g}^{*}_{20}},\
\bar{h}^{*}_{31}=\bar{g}^{*}_{31}-\frac{\bar{g}^{*}_{21} \bar{g}^{*}_{30}}{\bar{g}^{*}_{20}};\
\bar{j}^{*}_{00}=\bar{h}^{*}_{00} \bar{h}_{31}^{*\frac{4}{5}}  \bar{h}_{20}^{*-\frac{7}{5}},\\
&&\bar{j}^{*}_{10}=\bar{h}^{*}_{10} \bar{h}_{31}^{*\frac{2}{5}}  \bar{h}_{20}^{*-\frac{6}{5}},\
\bar{j}^{*}_{01}=\bar{h}^{*}_{01} \bar{h}_{31}^{*\frac{1}{5}}  \bar{h}_{20}^{*-\frac{3}{5}},\
\bar{j}^{*}_{11}=\bar{h}^{*}_{11} \bar{h}_{20}^{*-\frac{2}{5}} \bar{h}_{31}^{*-\frac{1}{5}}.
\end{eqnarray*}


\section{Appendix D: Coefficients in the proof of Theorem \ref{tha3}}\label{AppD}

\begin{eqnarray*}
&&\alpha_{10}=s, \
\alpha_{01}=-\frac{q x_{*}}{a+x_{*} (c+x_{*})},\\
&&\alpha_{20}=-\frac{1}{k}\frac{b}{x_{*} (b+x_{*}) (-a+b (c+2 x_{*})+x_{*}^2)}(a (k s+r x_{*})+x_{*} (c k (s-r)+2 c r x_{*}+k x_{*} (s\\
&& \ \ \ \ \ \ \ \ \ -2 r)+3 r x_{*}^2))+\frac{(x_{*}^3-a (c+3 x_{*})) (b (k s+r x_{*})+x_{*} (k (s-r)+2 r x_{*}))+r)}{x_{*} (a+x_{*} (c+x_{*})) (-a+b (c+2x_{*})+x_{*}^2)},\\
&&\alpha_{11}=\frac{q (x_{*}^2-a)}{(a+x_{*} (c+x_{*}))^2},\\
&&\alpha_{30}=\frac{1}{k x_{*} (-a+b (c+2 x_{*})+x_{*}^2)}\bigg(\frac{1}{(a+x_{*} (c+x_{*}))^2}{(a^2-a (c^2+4 c x_{*}+6 x_{*}^2)+x_{*}^4) (b (k s+} \bigg.\\
&& \ \ \ \ \ \ \ \ \ {r x_{*})+x_{*} (k (s-r)+2 r x_{*}))}+\frac{b}{(b+x_{*})^2}(a (k s+r x_{*})+x_{*} (c k (s-r)+2 c r x_{*}+k x_{*}\\
&& \ \ \ \ \ \ \ \ \bigg.\times (s-2 r)+3 r x_{*}^2))\bigg),\\
&&\alpha_{21}=\frac{q (a (c+3 x_{*})-x_{*}^3)}{(a+x_{*} (c+x_{*}))^3},\\
&&\alpha_{40}=-\frac{1}{k x_{*} (-a+b (c+2 x_{*})+x_{*}^2)}\bigg((\frac{1}{(a+x_{*} (c+x_{*}))^3}(b (k s+r x_{*})+x_{*} (k (s-r)+2 r x_{*}))\bigg.\\
&& \ \ \ \ \ \ \ \ \times (a^2 (2 c+5 x_{*})-a (c^3+5 c^2 x_{*}+10 c x_{*}^2+10 x_{*}^3)+x_{*}^5)+\frac{b}{(b+x_{*})^3}(a (k s+r x_{*})+x_{*}\\
&& \ \ \ \ \ \ \ \ \bigg.\times (c k (s-r)+2 c r x_{*}+k x_{*} (s-2 r)+3 r x_{*}^2))\bigg),\\
&&\alpha_{31}=\frac{q (a^2-a (c^2+4 c x_{*}+6 x_{*}^2)+x_{*}^4)}{(a+x_{*} (c+x_{*}))^4};\\
&&\beta_{10}=\frac{s (a+x_{*} (c+x_{*}))^2 (b (k s+r x_{*})+x_{*} (k (s-r)+2 r x_{*}))}{k q x^{2}_{*} (-a+b (c+2 x_{*})+x_{*}^2)},\\
&&\beta_{20}=-\frac{s (a+x_{*} (c+x_{*}))^2 (b (k s+r x_{*})+x_{*} (k (s-r)+2 r x_{*}))}{k q x_{*}^3 (-a+b (c+2 x_{*})+x_{*}^2)},\
\beta_{11}=\frac{2 s}{x_{*}},\\
&&\beta_{02}=-\frac{k q s x_{*} (-a+b (c+2 x_{*})+x_{*}^2)}{(a+x_{*} (c+x_{*}))^2 (b (k s+r x_{*})+x_{*} (k (s-r)+2 r x_{*}))},\\
&&\beta_{30}=\frac{s (a+x_{*} (c+x_{*}))^2 (b (k s+r x_{*})+x_{*} (k (s-r)+2 r x_{*}))}{k q x_{*}^4 (-a+b (c+2 x_{*})+x_{*}^2)},\
\beta_{21}=-\frac{2 s}{x_{*}^2},\\
&&\beta_{12}=\frac{k q s (-a+b (c+2 x_{*})+x_{*}^2)}{(a+x_{*} (c+x_{*}))^2 (b (k s+r x_{*})+x_{*} (k (s-r)+2 r x_{*}))},\\
&&\beta_{40}=-\frac{s (a+x_{*} (c+x_{*}))^2 (b (k s+r x_{*})+x_{*} (k (s-r)+2 r x_{*}))}{k q x_{*}^5 (-a+b (c+2 x_{*})+x_{*}^2)},\
\beta_{31}=\frac{2 s}{x_{*}^3},\\
&&\beta_{22}=-\frac{k q s (-a+b (c+2 x_{*})+x_{*}^2)}{x_{*} (a+x_{*} (c+x_{*}))^2 (b (k s+r x_{*})+x_{*} (k (s-r)+2 r x_{*}))};\\
&&\delta_{20}=\frac{\alpha_{10} \alpha_{11}-\alpha_{01} \alpha_{20}}{\omega},\
\delta_{11}=-\alpha_{11},\
\delta_{30}=\frac{\alpha_{01}^2 \alpha_{30}-\alpha_{01} \alpha_{10} \alpha_{21}}{\omega}, \
\delta_{21}=\alpha_{01} \alpha_{21},\\
&&\delta_{40}=\frac{\alpha_{01}^2 \alpha_{10} \alpha_{31}-\alpha_{01}^3 \alpha_{40}}{\omega},\
\delta_{31}=-\alpha_{01}^2 \alpha_{31};
\end{eqnarray*}

\begin{eqnarray*}
&&\gamma_{20}=\frac{-\alpha_{01}^2\beta_{20}+\alpha_{10} \alpha_{01} (\beta_{11}-\alpha_{20})+\alpha_{10}^2 (\alpha_{11}-\beta_{02})}{\omega^2},\
\gamma_{11}=-\frac{\alpha_{10} (\alpha_{11}-2 \beta_{02})+\alpha_{01} \beta_{11}}{\omega},\\
&&\gamma_{02}=-\beta_{02},\
\gamma_{30}=\frac{\alpha_{01}{(\alpha_{01}^2 \beta_{30}+\alpha_{10} \alpha_{01} (\alpha_{30}-\beta_{21})+\alpha_{10}^2 (\beta_{12}-\alpha_{21}))}}{\omega^2},\\
&&\gamma_{21}=\frac{\alpha_{01} (\alpha_{10} (\alpha_{21}-2 \beta_{12})+\alpha_{01} \beta_{21})}{\omega},\
\gamma_{12}=\alpha_{01} \beta_{12},\\
&&\gamma_{40}=-\frac{\alpha_{01}^2 (\alpha_{01}^2 \beta_{40}+\alpha_{10} \alpha_{01} (\alpha_{40}-\beta_{31})+\alpha_{10}^2 (\beta_{22}-\alpha_{31}))}{\omega ^2},\\
&&\gamma_{31}=-\frac{\alpha_{01}^2 {(\alpha_{10} (\alpha_{31}-2 \beta_{22})+\alpha_{01} \beta_{31})}}{\omega},\ \gamma_{22}=-\alpha_{01}^2 \beta_{22};\\
&&{\eta_{11}=x_{*}^{3} \eta_{11}^{0} r^2+ksx_{*} \eta_{11}^{1}r+k^2 s^2 \eta_{11}^2, \text{with}}\\
&&\eta_{11}^{0}=a^4 (-4 b^2+b (4 k-9 x_{*})+x_{*} (k-4 x_{*}))+a^3 (-x_{*} (3 b^2 (8 b+7 c)+k^2 (16 b+c)-10 b k (4\\
&& \ \ \ \ \ \ \ \ \times b+3 c))-2 x^{2}_{*} (41 b^2+4 b (6 c-11 k)+3 k (k-2 c))-4 b (b-k) (b^2-k (b+c))-8 x_{*}^3\\
&& \ \ \ \ \ \ \ \ \times (15 b+3 c-5 k)-56 x_{*}^4)+a^2 (2 x_{*}^4 (47 b^2+b (16 k-60 c)-7 (2 c^2-6 c k+k^2))+4 b^2 c^2\\
&& \ \ \ \ \ \ \ \ \times (b-k)^2+b c x_{*} (b-k) (9 b^2+24 b c-9 b k+5 c k)+x^{3}_{*} (72 b^3+20 b^2 (c-3 k)+b (-33 c^2\\
&& \ \ \ \ \ \ \ \ +112 c k-12 k^2)+c k (17 c-16 k))+2 x_{*}^2 (6 b^4+3 b^3 (9 c-4 k)+2 b^2 (c-k) (4 c-3 k)+\\
&& \ \ \ \ \ \ \ \ 13 b c k (c-k)-c^2 k^2)-2 x_{*}^5 (3 b+50 c-27 k)-48 x_{*}^6)+x_{*}^3 (-b c^3 (b-k) (b^2-k (b+c))\\
&& \ \ \ \ \ \ \ \ -2 x_{*}^5 (69 b^2+6 b (4 c-5 k)-10 c^2+6 c k+k^2)-2 b c^2 x_{*} (b-k) (3 b^2+3 b (c-k)+c (c\\
&& \ \ \ \ \ \ \ \ -2 k))+x_{*}^4 (-72 b^3+b^2 (84 k-162 c)+3 b (5 c^2+12 c k-4 k^2)+c (4 c^2-15 c k+2 k^2))\\
&& \ \ \ \ \ \ \ \ -b c x_{*}^2 (15 b^3+6 b^2 (6 c-5 k)+b (16 c^2-41 c k+15 k^2)+c k (5 k-4 c))-2 x_{*}^3 (k (-12 b^3-\\
&& \ \ \ \ \ \ \ \ 50 b^2 c+b c^2+c^3)+b (6 b^3+45 b^2 c+32 b c^2-3 c^3)+k^2 (6 b-c) (b+c))+3 x_{*}^6 (-27 b+4\\
&& \ \ \ \ \ \ \ \ \times c+3 k)-12 x_{*}^7)+a x_{*} (6 b^2 c^3 (b-k)^2+2 x_{*}^5 (225 b^2+276 b c-156 b k+12 c^2-50 c k+\\
&& \ \ \ \ \ \ \ \ 11 k^2)+x_{*}^4 (216 b^3+b^2 (679 c-272 k)+b (210 c^2-394 c k+56 k^2)+c k (19 k-18 c))+2\\
&& \ \ \ \ \ \ \ \ \times x_{*}^3 (18 b^4+6 b^3 (29 c-6 k)+2 b^2 (79 c^2-103 c k+9 k^2)+b c (c-4 k) (15 c-8 k)-c^2 k (c\\
&& \ \ \ \ \ \ \ \ -2 k))+c x_{*}^2 (58 b^4+4 b^3 (45 c-29 k)+b^2 (57 c^2-198 c k+58 k^2)+18 b c k (k-c)+c^2 k^2)\\
&& \ \ \ \ \ \ \ \ +2 b c^2 x_{*} (b-k) (3 b (5 b+6 c)-k (15 b+c))+8 x_{*}^6 (51 b+16 c-13 k)+120 x_{*}^7),\\
&&\eta_{11}^{1}=a^3 (b^3 c (k-b)+x_{*}^3 (-108 b^2-53 b c+54 b k+5 c k)+b x_{*}^2 (-43 b^2-20 b c+31 b k+16 c k)\\
&& \ \ \ \ \ \ \ \ +b^2 x_{*} (-10 b^2+b (c+10 k)+c k)+x_{*}^4 (-139 b-24 c+23 k)-58 x_{*}^5)+x_{*}^5 (b c x_{*}^2 (-21 b^2\\
&& \ \ \ \ \ \ \ \ +5 k (b+c)-10 b c-4 c^2)-b c^2 x_{*} (6 b^2+5 b c+2 b k+2 c^2-7 c k)-x_{*}^4 (30 b^2+3 b (7 c-2\\
&& \ \ \ \ \ \ \ \ \times k) +c (k-11 c))+x_{*}^3 (-12 b^3+b^2 (4 k-45 c)+b c (5 c+9 k)+c^2 (2 c-3 k))+b c^2 (b^3-\\
&& \ \ \ \ \ \ \ \ b^2 (2 c+k)-b c^2+2 c^2 k)+x_{*}^5 (-21 b+4 c+k)-3 x_{*}^6)+a^2 x_{*} (3 b^3 c^2 (b-k)+x_{*}^4 (198 b^2\\
&& \ \ \ \ \ \ \ \ -39 b c-36 b k-15 c^2+17 c k)+b^2 c x_{*} (2 b^2+29 b c-2 b k-13 c k)+x_{*}^3 (99 b^3+b^2 (71 c-\\
&& \ \ \ \ \ \ \ \ 51 k)+b c (29 k-16 c)+2 c^2 k)+b x_{*}^2 (12 b^3+b^2 (53 c-12 k)+17 b c (2 c-k)+7 c^2 k)+x_{*}^5\\
&& \ \ \ \ \ \ \ \ \times (129 b-44 c-9 k)+24 x_{*}^6) +a x_{*}^2 (3 b^3 c^3 (b-k)+x_{*}^5 (128 b^2+273 b c-34 b k+20 c^2-\\
&& \ \ \ \ \ \ \ \ 29 c k)+2 b^2 c^2 x_{*} (5 b^2+12 b c-5 b k-6 c k)+x_{*}^4 (51 b^3+b^2 (334 c-31 k)+b c (103 c-74 k)\\
&& \ \ \ \ \ \ \ \ +c^2 (2 c-7 k))+b c x_{*}^2 (19 b^3+b^2 (89 c-19 k)+5 b c (8 c-9 k)-2 c^2 k)+x_{*}^3 (6 b^4+b^3 (151 c\\
&& \ \ \ \ \ \ \ \ -6 k)+b^2 c (174 c-85 k)+2 b c^2 (8 c-9 k)-2 c^3 k)+x_{*}^6 (123 b+80 c-15 k)+42 x_{*}^7)+a^4\\
&& \ \ \ \ \ \ \ \ \times (b^3-b^2 (k+4 x_{*})+2 b x_{*} (k-6 x_{*})-5 x_{*}^3),
\end{eqnarray*}

\begin{eqnarray*}
&&\eta_{11}^{2}=a^4 (b^3-3 b x_{*}^2-x_{*}^3)+{x_{*}^6 (b-c)^2 (-3 b c x_{*}-b c (b+c)+x_{*}^3)}+a^2 x_{*}^2 (-3 b^3 c (b-2 c)\\
&& \ \ \ \ \ \ \ \ +3 b x_{*}^2 (9 b^2+2 b c-c^2)+b^2 x_{*} (4 b^2+b c+9 c^2)+x_{*}^4 (31 b-2 c)+3 b x_{*}^3 (17 b-4 c)+9 x_{*}^5)\\
&& \ \ \ \ \ \ \ \ -a^3 (b^4 c+6 b^4 x_{*}+3 b^2 x_{*}^2 (7 b+c)+3 x_{*}^4 (11 b+c)+b x_{*}^3 (35 b+11 c)+9 x_{*}^5)+a x_{*}^3 (b^3 c^2\\
&& \ \ \ \ \ \ \ \ \times (3 c-2 b)+3 x_{*}^4 (-3 b^2+7 b c+c^2)+3 b^2 x_{*}^2 (-2 b^2+2 b c+5 c^2)-3 b^2 c x_{*} (b-2 c) (b+c)\\
&& \ \ \ \ \ \ \ \ +x_{*}^3 (-15 b^3+33 b^2 c+b c^2+c^3)-3 x_{*}^5 (b-3 c)+x_{*}^6).
\end{eqnarray*}

\end{appendices}

\end{document}

%% file: ex_shared.tex
\usepackage{subfigure,amsfonts,enumerate}
\usepackage{graphicx}
\usepackage{epstopdf}
\usepackage{algorithmic}
\usepackage{appendix}
\usepackage{verbatim}
\ifpdf
  \DeclareGraphicsExtensions{.eps,.pdf,.png,.jpg}
\else
  \DeclareGraphicsExtensions{.eps}
\fi

\usepackage{enumitem}
\setlist[enumerate]{leftmargin=.5in}
\setlist[itemize]{leftmargin=.5in}

\crefname{hypothesis}{Hypothesis}{Hypotheses}

\headers{Isola and mushroom dynamics, and bifurcations}{Yancong Xu, Yue Yang, Libin Rong and Pablo Aguirre}

\title{Isola and mushroom dynamics of limit cycles and bifurcations in a predator-prey system with additive Allee effect\thanks{Corresponding author: Pablo Aguirre, Email: pablo.aguirre@usm.cl
\funding{This work was funded by the National NSF of China (No. 11671114) and  National Science Foundation grantof USA(DMS-1950254). PA thanks Proyecto Basal CMM-Universidad de Chile.}}}

\author{Yancong Xu\thanks{Department of Mathematics, China Jiliang University, Hangzhou, 310018, China
  (\email{Yancongx@cjlu.edu.cn}).}
\and Yue Yang\thanks{School of Mathematical Sciences, Qufu Normal University, Qufu, 273165, China
  (\email{Kari\_Yang@163.com}).}
\and Libin Rong\thanks{ Department of Mathematics,
University of Florida, Gainesville, FL 32611, USA
  (\email{libinrong@ufl.edu}).}
 \and Pablo Aguirre\thanks{Departamento de Matem\'{a}tica, Universidad T\'{e}cnica Federico Santa Mar\'{i}a, Casilla 110-V, Valpara\'{i}so, Chile
  (\email{pablo.aguirre@usm.cl}).} }

\usepackage{amsopn}
